\newif\ifcompositio
\newif\ifarxiv
\compositiofalse\arxivtrue

\pdfoutput=1

\ifcompositio
\documentclass{compositio}
\else
\documentclass{article}
\fi

\usepackage{amsmath, amsthm, amssymb}
\usepackage{eucal}
\ifarxiv
\usepackage{geometry}
\fi
\usepackage[unicode,bookmarksnumbered=true,colorlinks=true,allcolors=blue,linktoc=all]{hyperref}
\ifcompositio
\usepackage[capitalise]{cleveref}
\else
\usepackage[capitalise,nameinlink]{cleveref}
\fi
\usepackage{tikz-cd}
\usepackage{stmaryrd}
\usepackage{listings}
\usepackage{etoolbox}
\usepackage{thm-restate}
\usepackage{quiver}
\usepackage[nottoc,notlot,notlof]{tocbibind}
\usepackage{caption}

\crefname{subsection}{Subsection}{subsections}

\hypersetup{bookmarksdepth=2}
\newtheorem{theorem}{Theorem}

\newtheorem{thm}{Theorem}[section]
\newtheorem{lem}[thm]{Lemma}
\newtheorem{prop}[thm]{Proposition}
\newtheorem{cor}[thm]{Corollary}
\newtheorem{question}[thm]{Question}
\newtheorem{conj}[thm]{Conjecture}

\theoremstyle{definition}
\newtheorem{defn}[thm]{Definition}

\theoremstyle{remark}
\newtheorem{remark}[thm]{Remark}
\newtheorem{example}[thm]{Example}

\newcommand{\ncmd}{\newcommand}

\definecolor{DefColor}{rgb}{0.6,0.15,0.25}
\newcommand{\mdef}[1]{\textcolor{DefColor}{#1}}
\newcommand{\tdef}[1]{\mdef{\emph{#1}}}

\ifcompositio
\ncmd{\pwrap}[1]{(#1)}
\else
\ncmd{\pwrap}[1]{#1}
\fi

\ncmd{\mbb}[1]{\mathbb{#1}}
\ncmd{\mrm}[1]{\mathrm{#1}}
\ncmd{\mcl}[1]{\mathcal{#1}}
\ncmd{\mfk}[1]{\mathfrak{#1}}

\usepackage{cjhebrew}
\DeclareFontFamily{U}{rcjhbltx}{}
\DeclareFontShape{U}{rcjhbltx}{m}{n}{<->s*[1.2]rcjhbltx}{}
\DeclareSymbolFont{hebrewletters}{U}{rcjhbltx}{m}{n}
\DeclareMathSymbol{\tsadi}{\mathord}{hebrewletters}{118}

\ncmd{\todo}[1]{\textbf{TODO #1}}
\ncmd{\reftodo}[1]{\textbf{REF #1}}

\ncmd{\id}[1][]{\mrm{id}\ifstrempty{#1}{}{_{#1}}}

\ncmd{\op}{\mrm{op}}
\ncmd{\rev}{\mrm{rev}}
\ncmd{\seg}{\mrm{seg}}
\ncmd{\gpc}{\mrm{gpc}}
\ncmd{\gl}{\mrm{gl}}
\ncmd{\BC}{\mrm{BC}}
\ncmd{\ho}{\mrm{ho}}

\ncmd{\cocart}{\sqcup}
\ncmd{\infcocart}{{\cocart_\infty}}
\ncmd{\mcocart}{{\cocart_m}}
\ncmd{\mzcocart}{{\cocart_{m_0}}}
\ncmd{\atomicpres}{\mathrm{at}}
\ncmd{\st}{\mathrm{st}}
\ncmd{\atomic}[1][]{\ifstrempty{#1}{}{#1\text{-}}\atomicpres}
\ncmd{\dbl}{\mrm{dbl}}
\ncmd{\ldbl}{\mrm{ldbl}}
\ncmd{\lax}{\mrm{lax}}
\ncmd{\ofin}{1\text{-}\mathrm{fin}}
\ncmd{\mfin}{m\text{-}\mathrm{fin}}
\ncmd{\pifin}{\pi\text{-}\mathrm{fin}}

\ncmd{\Mod}{\mrm{Mod}}
\ncmd{\ModEk}{\widehat{\Mod}_{\Ek}}
\ncmd{\LMod}{\mrm{LMod}}
\ncmd{\RMod}{\mrm{RMod}}
\ncmd{\Vect}{\mrm{Vect}}
\ncmd{\Sp}{\mrm{Sp}}
\ncmd{\conSp}{\Sp_{\geq 0}}
\ncmd{\Spaces}{\mcl{S}}
\ncmd{\Spacesm}{\Spaces_m^{(p)}}
\ncmd{\Spacespifin}{\Spaces_{\pifin}^{(p)}}
\ncmd{\Spaceso}{\Spaces_1^{(p)}}
\ncmd{\paran}[1]{\ifstrempty{#1}{}{({#1})}}
\ncmd{\Op}[1][]{\mrm{Op}}
\ncmd{\PCMonm}[1][]{\mrm{PCMon}_m^{(p)} \paran{#1}}
\ncmd{\Mon}[1][]{\mrm{Mon}\paran{#1}}
\ncmd{\CMon}[1][]{\mrm{CMon}\paran{#1}}
\ncmd{\CMongl}[1][]{\mrm{CMon}^\gl\paran{#1}}
\ncmd{\CMonm}[1][]{\mrm{CMon}_m^{(p)}\paran{#1}}
\ncmd{\CMono}[1][]{\mrm{CMon}_1^{(p)}\paran{#1}}
\ncmd{\CMonmz}[1][]{\mrm{CMon}_{m_0}^{(p)}\paran{#1}}
\ncmd{\CMoninf}[1][]{\mrm{CMon}_\infty^{(p)}\paran{#1}}
\ncmd{\Perf}{\mathrm{Perf}}
\ncmd{\Poset}{\mathrm{Poset}}
\ncmd{\FinSet}{\mathrm{FinSet}}
\ncmd{\Cat}{\mathrm{Cat}}
\ncmd{\Cato}{\Cat_1}
\ncmd{\Catst}{\Cat^{\st}}
\ncmd{\Catmfin}{\Cat_{\mfin}}
\ncmd{\Catmfinst}{\Catst_{\mfin}}
\ncmd{\Catofinst}{\Catst_{\ofin}}
\ncmd{\Catpifin}{\Cat_{\pifin}}
\ncmd{\Catpifinst}{\Catst_{\pifin}}
\ncmd{\Catmfinstleqn}{\Catst_{\mfin,\leq n}}
\ncmd{\Catmzfinst}{\Catst_{m_0\text{-}\mathrm{fin}}}
\ncmd{\CatK}{\Cat_\clK}
\ncmd{\CAT}{\widehat{\Cat}}
\ncmd{\CATo}{\CAT_1}
\renewcommand{\Pr}{\mathrm{Pr}}
\ncmd{\PrL}{\Pr^\mathrm{L}}
\ncmd{\PrLtimes}{\Pr^\mathrm{L,\times}}
\ncmd{\PrLotimes}{\Pr^\mathrm{L,\otimes}}
\ncmd{\PrR}{\Pr^\mathrm{R}}
\ncmd{\add}{\mathrm{add}}
\ncmd{\tsadim}{\tsadi^{[m]}}
\ncmd{\tsadio}{\tsadi^{[1]}}
\ncmd{\tsadimz}{\tsadi^{[m_0]}}

\ncmd{\yon}{\text{\usefont{U}{min}{m}{n}\symbol{'110}}}
\DeclareFontFamily{U}{min}{}
\DeclareFontShape{U}{min}{m}{n}{<-> dmjhira}{}
\ncmd{\yonM}{\yon^\MM}
\ncmd{\yonKM}{\yon_\clK^\MM}
\ncmd{\epsKM}{\varepsilon_\clK^\MM}
\ncmd{\ryonM}{\yon_0^\MM}

\ncmd{\underlying}[1][-]{\underline{({#1})}}
\ncmd{\unfurling}{\Upsilon}

\ncmd{\BB}{\mrm{B}}
\ncmd{\UU}{\mrm{U}}
\ncmd{\BU}{\BB\UU}
\ncmd{\ZZ}{\mbb{Z}}
\ncmd{\Zp}{\ZZ_p}
\ncmd{\Fp}{\mbb{F}_p}
\ncmd{\bbC}{\mbb{C}}
\ncmd{\QQ}{\mbb{Q}}
\renewcommand{\SS}{\mbb{S}}
\ncmd{\SSpinv}{\SS[p^{-1}]}
\ncmd{\Sppinv}{\Sp[p^{-1}]}
\ncmd{\unit}{\mathbf{1}}
\ncmd{\Qb}{\overline{\QQ}}
\ncmd{\CC}{\mcl{C}}
\ncmd{\DD}{\mcl{D}}
\ncmd{\EE}{\mcl{E}}
\ncmd{\MM}{\mcl{M}}
\ncmd{\NN}{\mcl{N}}
\ncmd{\VV}{\mcl{V}}
\ncmd{\OO}{\mcl{O}}
\ncmd{\bbE}{\mbb{E}}
\ncmd{\II}{I}
\ncmd{\JJ}{J}
\ncmd{\sK}{K}
\ncmd{\clK}{\mcl{K}}
\ncmd{\clKop}{{\clK^\op}}
\ncmd{\KK}{\mrm{K}}
\ncmd{\THH}{\mrm{THH}}
\ncmd{\TC}{\mrm{TC}}
\ncmd{\rE}{\mrm{E}}
\ncmd{\En}{\rE_n}
\ncmd{\Ek}{\rE_k}
\ncmd{\JW}{\rE(n)}
\ncmd{\cJW}{\widehat{\JW}}
\ncmd{\TT}{\mrm{T}}
\ncmd{\Km}{\KK^{[m]}}
\ncmd{\Kmo}{\KK^{[1]}}
\ncmd{\Kmz}{\KK^{[m_0]}}
\ncmd{\Kmmz}{\KK^{[m_0]}_{[m]}}
\ncmd{\Kzm}{\KK_{[m]}}
\ncmd{\Kzinf}{\KK_{[\infty]}}
\ncmd{\Kzo}{\KK_{[1]}}
\ncmd{\Kinf}{\KK^{[\infty]}}
\ncmd{\KU}{\KK\UU}
\ncmd{\ku}{\mrm{ku}}
\ncmd{\Kze}{\KK(0)}
\ncmd{\Ko}{\KK(1)}
\ncmd{\Tze}{\TT(0)}
\ncmd{\To}{\TT(1)}
\ncmd{\height}{\mathrm{ht}}
\ncmd{\Kn}{\KK(n)}
\ncmd{\Kk}{\KK(k)}
\ncmd{\Tn}{\TT(n)}
\ncmd{\Tk}{\TT(k)}
\ncmd{\Tnpo}{\TT(n+1)}
\ncmd{\MU}{\mrm{MU}}
\ncmd{\BP}{\mrm{BP}}
\ncmd{\BPn}{\BP\langle n \rangle}
\ncmd{\SpKn}{\Sp_{\Kn}}
\ncmd{\SpTn}{\Sp_{\Tn}}
\ncmd{\SpTnpo}{\Sp_{\Tnpo}}
\ncmd{\SpTk}{\Sp_{\TT(k)}}
\ncmd{\SpTo}{\Sp_{\To}}
\ncmd{\SpTze}{\Sp_{\Tze}}
\ncmd{\LTn}{L_{\Tn}}
\ncmd{\LTnpo}{L_{\Tnpo}}
\ncmd{\LKo}{L_{\Ko}}
\ncmd{\Lof}{L_1^f}
\ncmd{\Lnf}{L_n^f}
\ncmd{\Lnpof}{L_{n+1}^f}
\ncmd{\LnfSp}{\Lnf\Sp}
\ncmd{\LnpofSp}{\Lnpof\Sp}
\ncmd{\Fseg}{F^{\seg}}
\ncmd{\fL}{\mathrm{L}}
\ncmd{\fLn}{\fL^n}
\ncmd{\GL}{\mrm{GL}}

\ncmd{\omegapn}{\omega^{(n)}_p}

\ncmd{\iso}{\xrightarrow{\smash{\raisebox{-0.4ex}{\ensuremath{\scriptstyle\sim}}}}}

\DeclareMathOperator{\Fun}{Fun}
\ncmd{\FunL}{\Fun^\mrm{L}}
\ncmd{\FunR}{\Fun^\mrm{R}}
\ncmd{\intL}{\mrm{iL}}
\ncmd{\FunintL}{\Fun^\intL}
\ncmd{\Funlax}{\Fun^\mrm{lax}}
\DeclareMathOperator{\Aut}{Aut}
\DeclareMathOperator{\Span}{Span}
\DeclareMathOperator{\Alg}{Alg}
\DeclareMathOperator{\CAlg}{CAlg}
\DeclareMathOperator*{\colim}{colim}
\ncmd{\Lseg}{L^{\seg}}
\DeclareMathOperator{\End}{End}
\DeclareMathOperator{\PSh}{\mcl{P}}
\DeclareMathOperator{\PShM}{\PSh^\MM}
\DeclareMathOperator{\PShK}{\PSh_\clK}
\DeclareMathOperator{\PShKM}{\PSh_\clK^\MM}

\begin{document}
	\title{Higher Semiadditive Algebraic K-Theory and Redshift}

	\ifcompositio
\else
	\author{Shay Ben-Moshe\thanks{Einstein Institute of Mathematics, The Hebrew University of Jerusalem.} \and Tomer M. Schlank\footnotemark[1]}
	\date{}
	\fi
	
	\ifarxiv
	\maketitle
	\fi

	\begin{abstract}
		We define higher semiadditive algebraic K-theory, a variant of algebraic K-theory that takes into account higher semiadditive structure, as enjoyed for example by the $\Kn$- and $\Tn$-local categories.
		We prove that it satisfies a form of the redshift conjecture.
		Namely, that if $R$ is a ring spectrum of height $\leq n$, then its semiadditive K-theory is of height $\leq n+1$.
		Under further hypothesis on $R$, which are satisfied for example by the Lubin--Tate spectrum $\En$, we show that its semiadditive algebraic K-theory is of height exactly $n+1$.
		Finally, we connect semiadditive K-theory to $\Tnpo$-localized K-theory, showing that they coincide for any $p$-invertible ring spectrum and for the completed Johnson--Wilson spectrum $\cJW$.
	\end{abstract}

	\begin{figure}[h!]
		\vspace{4mm}
		\centering
		\includegraphics[width=90mm]{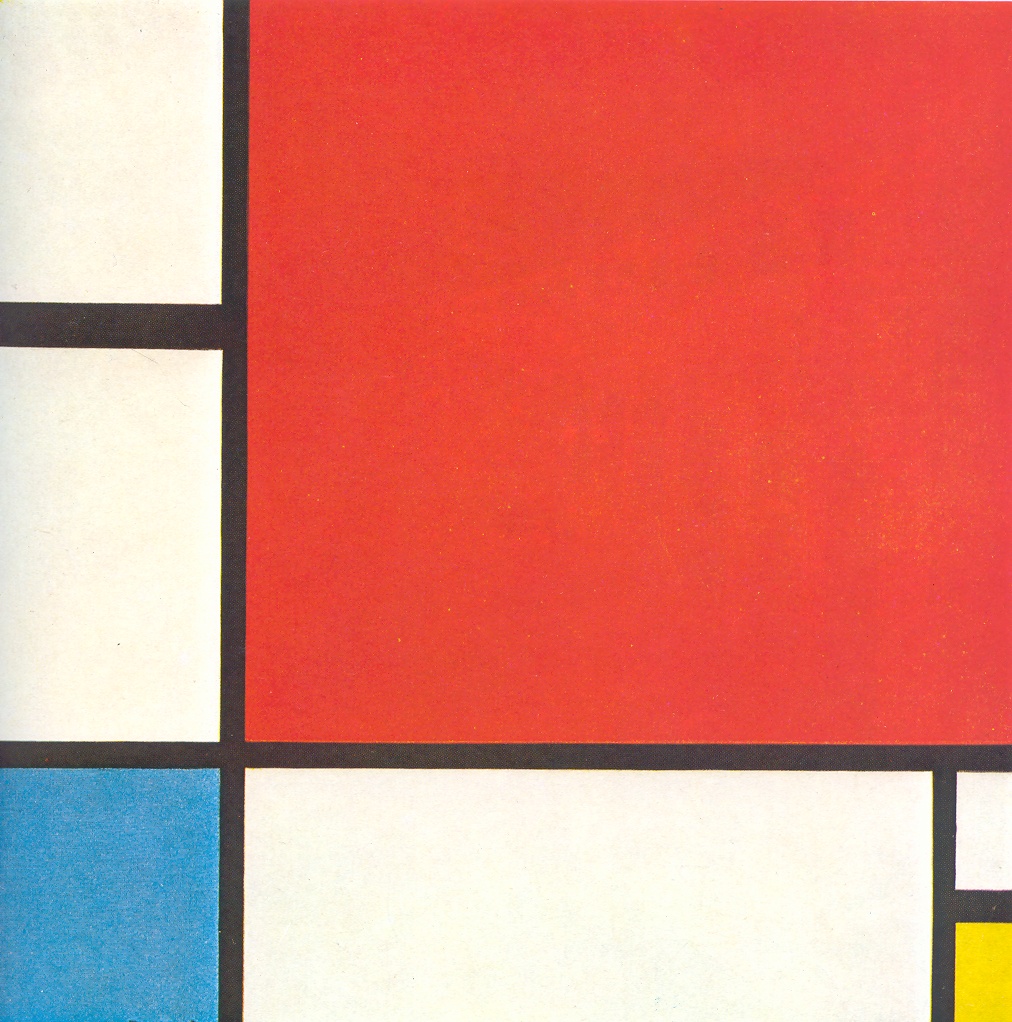}
		\caption*{Composition with Red, Blue and Yellow by Piet Mondrian.}
	\end{figure}

	\newpage
	
	\ifarxiv
	\tableofcontents
	\newpage
	\fi
	
	\section{Introduction}

\subsection{Overview}

\subsubsection{Descent and Redshift}\label{descent-redshift}

Algebraic K-theory $\KK\colon \Catst \to \Sp$ is a rich invariant of stable $\infty$-categories and thus of rings and ring spectra.
Ausoni--Rognes \cite{AR02, AR08} suggested a fascinating program concerning the interaction between algebraic K-theory and the chromatic filtration on spectra, now known as the redshift philosophy.
Namely, that algebraic K-theory increases the chromatic height of ring spectra by $1$.
They demonstrated this phenomenon at height $1$, and conjectured that it persists to arbitrary heights.
Another interesting aspect of algebraic K-theory is its descent properties.
For example, it is known by \cite{NisDesc} that it satisfies Nisnevich descent for ordinary rings, while it fails to satisfy \'etale descent due to its failure to satisfy Galois descent.
The recent breakthroughs of \cite{DescVan, Pur} have shown that chromatically localized K-theory does satisfy Galois descent under certain hypotheses, which was used to prove the following part of the redshift conjecture.

\begin{thm}[{\cite[Theorem A]{DescVan}}]
	Let $R \in \CAlg(\Sp)$ and $n \geq 0$. If $\LTn R = 0$ then $L_{\Tnpo} \KK(R) = 0$.
\end{thm}

In addition, Hahn--Wilson \cite{BPex} and Yuan \cite{Yuan} give the first examples of non-vanishing of $\TT(n+1)$-localized K-theory for ring spectra of chromatic height $n$, at arbitrary heights $n \geq 0$.
Building on this, Burklund--S.--Yuan \cite{Null} have recently proved the non-vanishing of $\TT(n+1)$-localized K-theory for all commutative ring spectra of chromatic height $n$.

\subsubsection{Higher Semiadditivity}

Hopkins--Lurie \cite[Theorem 5.2.1]{HL} and Carmeli--S.--Yanovski \cite[Theorem A]{TeleAmbi} proved that the chromatically localized $\infty$-categories $\SpKn$ and $\SpTn$ (respectively) are \emph{$\infty$-semiadditive}.
Namely, that there is a canonical natural equivalence between limits and colimits indexed by $\pi$-finite spaces (i.e.\ spaces with finitely many connected components and finitely many non-zero homotopy groups all of which are finite).
In this paper we will only make use of \emph{$p$-typical} higher semiadditivity, that is, relaxing the condition to $\pi$-finite $p$-spaces (i.e.\ $\pi$-finite spaces whose homotopy groups are all $p$-groups), which we thus simply call higher semiadditivity.

Harpaz \cite{Harpaz} studied the connection between $\infty$-semiadditivity and \emph{$\infty$-commutative monoids}.
Recall that a ($0$-)commutative monoid is, roughly speaking, the structure of summation of finite families of elements (in a coherently associative and commutative way).
Similarly, a ($p$-typical) $\infty$-commutative monoid is, roughly speaking, the structure of ``integration'' of families of elements indexed by a $\pi$-finite $p$-space (in a coherently associative and commutative way).
More precisely, given an $\infty$-category $\CC$, the $\infty$-category of ($p$-typical) $\infty$-commutative monoids in $\CC$ is defined to be
$\CMoninf[\CC] = \Fun^{\seg}(\Span(\Spacespifin)^\op,\CC)$,
the full subcategory of those functors from spans of $\pi$-finite $p$-spaces that satisfy the $\infty$-Segal condition.
In \cite[Corollary 5.19]{Harpaz} and \cite[Proposition 5.3.1]{AmbiHeight} it is shown that the property of being a ($p$-typically) $\infty$-semiadditive presentable $\infty$-category is classified by the \emph{mode} $\CMoninf[\Spaces]$ of $\infty$-commutative monoids in spaces.\footnote{The cited papers work in the non $p$-typical case, but the same proofs work for the $p$-typical case.}
That is, a presentable $\infty$-category $\CC$ is $\infty$-semiadditive if and only if it admits a (necessarily unique) module structure over $\CMoninf[\Spaces]$ in $\PrL$.
Furthermore, any object $X \in \CC$ in an $\infty$-semiadditive presentable $\infty$-category $\CC$ is canonically endowed with the structure of an $\infty$-commutative monoid, that is, there is an equivalence $\CC \iso \CMoninf[\CC]$.

Using this $\infty$-commutative monoid structure, \cite[Definition 3.1.6]{AmbiHeight} introduces the \emph{semiadditive height} of an object $X \in \CC$, denoted by $\height(X)$.
The notion of semiadditive height, which is defined in arbitrary $\infty$-semiadditive $\infty$-categories, is related to the chromatic height.
For example, all objects of $\SpKn$ and $\SpTn$ are of semiadditive height $n$ by \cite[Theorem 4.4.5]{AmbiHeight}.

A particularly interesting example of an $\infty$-semiadditive presentable $\infty$-category, which is studied in \cite{AmbiHeight}, is the mode classifying the property of being a $p$-local stable $\infty$-semiadditive presentable $\infty$-category, $\tsadi = \CMoninf[\Sp_{(p)}]$, consisting of ($p$-typical) $\infty$-commutative monoids in $p$-local spectra (see \cref{tsadim-def}).
By construction, there is a canonical map of modes $(-)^{\gpc}\colon \CMoninf[\Spaces] \to \tsadi$, which we call the \emph{group-completion}.
Additionally, there is a canonical map of modes $\LTn^{\tsadi}\colon \tsadi \to \SpTn$, which by \cite[Corollary 5.5.14]{AmbiHeight} is a smashing localization, and in particular has a fully faithful right adjoint.

Another important example of an $\infty$-semiadditive presentable $\infty$-category is $\Catpifin$, consisting of $\infty$-categories admitting colimits over all $\pi$-finite $p$-spaces (see \cite[Theorem 5.23]{Harpaz} and \cite[Proposition 2.2.7]{AmbiHeight}).
As an $\infty$-semiadditive $\infty$-category, its objects, which are themselves $\infty$-categories, can have a semiadditive height.
Additionally, there is an interplay between the semiadditive height of objects in an $\infty$-semiadditive $\infty$-category and the semiadditive height of the $\infty$-category itself as an object of $\Catpifin$, which we view as the crucial step at which redshift happens.

\begin{thm}[\pwrap{Semiadditive Redshift {\cite[Theorem B]{AmbiHeight}}}]\label{semiadditive-redshift}
	Let $\CC$ be an $\infty$-semiadditive $\infty$-category,
	then $\height(X) \leq n$ for all $X \in \CC$ if and only if $\height(\CC) \leq n + 1$ as an object of $\Catpifin$.
\end{thm}

\subsubsection{Higher Semiadditive Algebraic K-Theory of Categories}

In this paper we study the confluence of the above ideas.
As $\Catpifin$ is itself $\infty$-semiadditive, the $\infty$-categories therein admit a canonical structure of $\infty$-commutative monoids, called the \emph{$\infty$-cocartesian structure}, via the equivalence
$\Catpifinst \xrightarrow{(-)^\cocart} \CMoninf[\Catpifinst]$,
where the integration of a family of objects is given by their colimit.
We observe that the $S_\bullet$-construction preserves limits, thus it preserves $\infty$-commutative monoid structure.
This observation along with the group-completion functor described above lead us to the main definition of the present paper.
By analogy with the definition of ordinary algebraic K-theory, in \cref{sa-k-def} we define \emph{$\infty$-semiadditive algebraic K-theory} $\Kinf\colon \Catpifinst \to \tsadi$ as the composition
$$
	\Catpifinst
	\xrightarrow{(-)^\cocart} \CMoninf[\Catpifinst]
	\xrightarrow{S_\bullet} \CMoninf[\Spaces]^{\Delta^\op}
	\xrightarrow{(-)^{\gpc}} \tsadi^{\Delta^\op}
	\xrightarrow{|-|} \tsadi
	\xrightarrow{\Omega} \tsadi
	.	
$$
It is immediate from this definition that $\Kinf$ is $\infty$-semiadditive, i.e.\ preserves (co)limits indexed by $\pi$-finite $p$-spaces.

We give a second construction of $\Kinf$ which connects it to ordinary algebraic K-theory.
We show that $\Kinf(\CC)$ is obtained by taking the $\infty$-cocartesian structure on $\CC$, applying ordinary algebraic K-theory level-wise, and forcing the result to satisfy the $\infty$-Segal condition.
More precisely, we define the functor $\Kzinf\colon \Catpifinst \to \Fun(\Span(\Spacespifin)^\op, \Sp_{(p)})$ by the composition
$$
	\Catpifinst
	\xrightarrow{\mkern-1mu(-)^\cocart\mkern-3mu} \CMoninf[\Catpifinst]
	\subset \Fun(\Span(\Spacespifin)^\op, \Catpifinst)
	\xrightarrow{\KK \circ -\mkern-1mu} \Fun(\Span(\Spacespifin)^\op, \Sp_{(p)}),
$$
and construct a natural transformation $\Kzinf \Rightarrow \Kinf$.
Then, in \cref{sheafification-ord} we show that after reflecting to the subcategory $\tsadi \subset \Fun(\Span(\Spacespifin)^\op, \Sp_{(p)})$ of functors satisfying the $\infty$-Segal condition, the map becomes an equivalence $\Lseg \Kzinf(\CC) \iso \Kinf(\CC) \in \tsadi$.
Note that evaluation of the original map at $* \in \Spacespifin$ gives a comparison map $\KK(\CC) \to \underline{\Kinf(\CC)} \in \Sp_{(p)}$.

Using the second construction of $\Kinf$ and the lax symmetric monoidal structure on ordinary algebraic K-theory, in \cref{k-lax} we endow $\Kinf$ with a lax symmetric monoidal structure.
To that end, we prove certain results about Day convolution and its connection to the mode symmetric monoidal structure on $\CMoninf[\Spaces]$ (see \cref{cmon-mul}).

Similar constructions can be carried in the ($p$-typically) $m$-semiadditive context, for any $0 \leq m \leq \infty$.
We define an $m$-semiadditive version of algebraic K-theory $\Km\colon \Catmfinst \to \tsadim$ from the $\infty$-category of stable $\infty$-categories admitting colimits indexed by $m$-finite $p$-spaces to the universal $p$-local stable ($p$-typically) $m$-semiadditive presentable $\infty$-category.
The case $m=\infty$ is $\Kinf\colon \Catpifinst \to \tsadi$ mentioned above.
The case $m=0$ reproduces the $p$-localization of ordinary algebraic K-theory by \cref{K0-is-K}.
The case $m=1$ is closely related to equivariant algebraic K-theory.
In equivariant algebraic K-theory, given finite groups $H < G$, there is a corresponding transfer map.
Assuming $H$ and $G$ are $p$-groups, $\BB H \to \BB G$ is a map of $1$-finite $p$-spaces, which gives an integration operation on $\Kzo(\CC)$ for $\CC \in \Catofinst$, arising from the left Kan extension $\CC^{\BB H} \to \CC^{\BB G}$, which reproduces the transfer map on equivariant algebraic K-theory.
Thus $\Kmo$ is essentially obtained from equivariant algebraic K-theory by forcing the $1$-Segal condition.

\subsubsection{Higher Semiadditive Algebraic K-Theory of Algebras}

The above discussion gives the definition of the $\infty$-semiadditive algebraic K-theory of $\infty$-categories in $\Catpifinst$.
One rich source of such $\infty$-categories is the $\infty$-category $\LMod_R$ for an algebra $R \in \Alg(\tsadi)$.
However, even though this $\infty$-category is stable and has all colimits indexed by $\pi$-finite $p$-spaces, it is too large.
In order to avoid the Eilenberg swindle, one has to pass to a smaller $\infty$-category, which still enjoys these properties.
We show that passing to the left dualizable modules gives $\LMod_R^\ldbl \in \Catpifinst$, leading us to define
$$
	\Kinf(R) = \Kinf(\LMod_R^\ldbl).
$$
Recall that $\LTn^{\tsadi}\colon \tsadi \to \SpTn$ is a smashing localization, so that if $R \in \SpTn$, then the $\infty$-categories of modules in $\tsadi$ and in $\SpTn$ coincide.
We generalize the passage from all modules to the left dualizable modules, and make it into a lax symmetric monoidal functor, via the theory of atomic objects, as explained later in the introduction.
This allows us to endow $\Kinf\colon \Alg(\tsadi) \to \tsadi$ with a lax symmetric monoidal structure.

\subsubsection{Redshift}

Recall that $\Kinf$ is an $\infty$-semiadditive functor.
As $\infty$-semiadditive functors can only decrease semiadditive height, we immediately get that if $\CC \in \Catpifinst$ has $\height(\CC) \leq n$ as an object of $\Catpifinst$, then $\height(\Kinf(\CC)) \leq n$ (see \cref{redshift-small}).
Evidently, this does not exhibit the increase in height postulated by the redshift philosophy.
Instead, the increase in height happens at the stage of categorification, turning from considering the height of objects to the height of their $\infty$-category, as in \cref{semiadditive-redshift}.
Using \cref{semiadditive-redshift} and \cref{redshift-small} we deduce our first main result.

\begin{theorem}[\pwrap{\cref{redshift-upper}}]\label{redshift-upper-intro}
	Let $R \in \Alg(\tsadi)$ have $\height(R) \leq n$ and let $m > n$, then
	$$
	\height(\Km(R)) \leq n+1.
	$$
\end{theorem}

To give a lower bound on the height, we make use of the higher height analogues of cyclotomic extensions defined in \cite[Definition 4.7]{Cyclo}.
Recall that for $R \in \Alg(\tsadi)$ of $\height(R)=n$, there is a ($(\ZZ/p)^\times$-equivariant) splitting of algebras
$R[\BB^n C_p] \cong R \times R[\omegapn]$,
where $R[\omegapn]$ is called the (height $n$) $p$-cyclotomic extension of $R$, which generalizes ordinary cyclotomic extensions at height $0$ (i.e.\ for algebras over the rationals).
We say that $R$ has \emph{(height $n$) $p$-th roots of unity} if the cyclotomic extension splits as a product $R[\omegapn] \cong \prod_{(\ZZ/p)^\times} R$ (see \cref{p-roots}).
For example, by \cite[Proposition 5.1]{Cyclo}, the Lubin--Tate spectrum $\En$ has (height $n$) $p$-th roots of unity.
For such $R$, we get an equivalence of $R$-modules $R[\BB^n C_p] \cong R^p$, from which we immediately deduce the following strengthening of \cref{redshift-upper-intro}.

\begin{theorem}[\pwrap{\cref{redshift-lower}}]\label{redshift-lower-intro}
	Let $R \in \Alg(\tsadi)$ of $\height(R) = n$ have (height $n$) $p$-th roots of unity and let $m > n$, then
	$$
	\height(\Km(R)) = n+1.
	$$
	In particular, $\height(\Km(\En)) = n+1$.
\end{theorem}

A natural question left open is the following:

\begin{question}
	Can the assumption of having (height $n$) $p$-th roots of unity be dropped?
	Namely, is it true that if $R \in \Alg(\tsadi)$ is of height $n$, then $\Km(R)$ is of height exactly $n+1$?
\end{question}

\subsubsection{Relationship to Chromatically Localized K-Theory}

As we have seen in \cref{redshift-upper-intro} and \cref{redshift-lower-intro}, $\Km$ satisfies a form of the redshift conjecture with respect to semiadditive height.
A natural next direction is connecting these results to ordinary algebraic K-theory and the chromatic height.
Let $R \in \Alg(\SpTn)$.
The inclusion $\SpTnpo \subset \Sp$ admits a left adjoint $\LTnpo\colon \Sp \to \SpTnpo$.
Since $\KK(R) \in \Sp$, we can consider $\LTnpo \KK(R) \in \SpTnpo$.
Similarly, there is an inclusion $\SpTnpo \subset \tsadim$, which admits a left adjoint $\LTnpo^{\tsadim}\colon \tsadim \to \SpTnpo$.
Since $\Km(R) \in \tsadim$, we can consider $\LTnpo^{\tsadim} \Km(R) \in \SpTnpo$.
The comparison map between ordinary algebraic K-theory and higher semiadditive algebraic K-theory yields a comparison map\footnote{Note that the source of the comparison map is simply $\LTnpo^{\tsadimz} \Kmz(R)$ for $m_0 = 0$. Namely, the comparison map is the comparison map of two different levels of semiadditivity.}
$$
\LTnpo \KK(R) \to \LTnpo^{\tsadim} \Km(R)
\quad \in \SpTnpo.
$$
This raises two independent questions:

\begin{enumerate}
	\item Does $\Km(R)$ land in $\SpTnpo \subset \tsadim$?
 	\item Is the comparison map an equivalence?
\end{enumerate}

A positive answer to both questions will imply that $\Km(R) \cong \LTnpo \KK(R)$ (see \cref{conj-Km-Tnpo-K} below).
The first question is closely related to the Quillen--Lichtenbaum conjecture for $R$, in the guise of having a non-zero finite spectrum $X$ such that $\KK(R) \otimes X$ is bounded above, as we show in \cref{Km-Tk-local}.
The second question is equivalent to $\LTnpo \Kzm(R)$ satisfying the $m$-Segal condition.
More informally, having descent properties for chromatically localized K-theory.

Using the Galois descent results for $\Tnpo$-localized K-theory of \cite{DescVan}, this second question is answered in the affirmative for $m = 1$ in \cref{Ko-desc}.
In work in progress with Carmeli and Yanovski \cite{cycloredshift} we show that the descent result for chromatically localized K-theory generalizes from finite $p$-groups to arbitrary $\pi$-finite $p$-spaces.
This would give a positive answer to the second question for every $m \geq 1$.

Next, we focus on the case of height $0$, answering both question in the affirmative in complete generality.
Using the Quillen--Lichtenbaum property of $\SSpinv$ together with Galois descent we obtain the following:

\begin{theorem}[\pwrap{\cref{Km-height-0}}]\label{Km-height-0-intro}
	Let $R \in \Alg(\Sppinv)$ and let $m \geq 1$, then
	$$\Km(R) \cong L_{\To}\KK(R).$$
	In particular, $\Km(\Qb) \cong \KU_p$.
\end{theorem}

Finally, we study the completed Johnson--Wilson spectrum $\cJW$ at height $n \geq 1$.
In \cite{BPex}, Hahn--Wilson produced an $\bbE_3$-algebra structure on $\BPn$, for which they have proven a version of the Quillen--Lichtenbaum conjecture.
This structure also endows $\cJW$ with an $\bbE_3$-algebra structure.
Using their Quillen--Lichtenbaum result, along with a comparison of two direct computations of the higher commutative monoid structure on $\Km(\cJW)$, we obtain the following strengthening of \cref{redshift-lower-intro} for $\cJW$-algebras.
We would like to thank the anonymous referee for suggesting crucial parts of the proof of this result.

\begin{theorem}[\pwrap{\cref{Km-cJW}}]\label{Km-cJW-intro}
	Let $R \in \Alg(\LMod_{\cJW})$ where $\cJW$ is endowed with the Hahn--Wilson $\bbE_3$-algebra structure, and let $m \geq 1$, then
	$$
	\Km(R) \in \SpTnpo.
	$$
\end{theorem}

Using Galois descent for chromatically localized K-theory from \cite{DescVan} as mentioned above, we immediately get the following at $m = 1$:

\begin{theorem}[\pwrap{\cref{Ko-cJW}}]\label{Ko-cJW-intro}
	Let $R \in \Alg(\LMod_{\cJW})$ where $\cJW$ is endowed with the Hahn--Wilson $\bbE_3$-algebra structure, then
	$$
	\Kmo(R) \cong \LTnpo \KK(R).
	$$
	In particular, $\Kmo(\cJW) \cong \LTnpo \KK(\cJW)$.
\end{theorem}

As mentioned above, our upcoming work with Carmeli and Yanovski \cite{cycloredshift} implies that \cref{Ko-cJW-intro} generalizes to $m$-semiadditive K-theory for any $m \geq 1$.
This generalization, along with \cref{Km-height-0-intro} answering the case of height $0$, lead us to conjecture the following:

\begin{conj}\label{conj-Km-Tnpo-K}
	For any $R \in \Alg(\SpTn)$ and $m \geq 1$ we have
	$$
	\Km(R) \cong \LTnpo \KK(R).
	$$
\end{conj}

We would like to highlight two interesting phenomena exemplified by \cref{Km-height-0-intro} and \cref{Ko-cJW-intro}.
First, higher semiadditive algebraic K-theory lands in the highest non-zero height predicted by the redshift conjecture, without forcing it be in a pure height from the outside.
Second, algebraic K-theory can be modified to have a higher commutative monoid structure in two ways -- either by chromatically localizing it from the outside, or by internally remembering the higher commutative monoid structure on the input $\infty$-category.
These results show that these two a priori distinct objects coincide, at least in some cases.
This identification gives different approaches to study the higher commutative monoid structure, similarly to the proof of \cref{Km-cJW-intro} itself.

\subsubsection{Atomic Objects and a Monoidal Natural Yoneda}

Recall that in the construction of the higher semiadditive algebraic K-theory of $R \in \Alg(\tsadi)$ described above, we passed to the left dualizable objects.
In order to study the functoriality of this construction in $R$, as well as to generalize the construction to stable $\infty$-semiadditive presentable $\infty$-categories other than $\infty$-categories of modules, we define and study \emph{$\MM$-atomic} objects for any mode $\MM$ (see \cref{atomic-def}).
One of our main results is that $\MM$-atomic objects indeed coincide with left dualizable objects in left modules, i.e.\ $\LMod_R^{\atomic} = \LMod_R^\ldbl$ for any $R \in \Alg(\MM)$ (see \cref{atomic-ldbl}).
Another direction of generalization is the case $\MM = \Sp$, where $\Sp$-atomic objects coincide with compact objects.
We also show that for any \emph{absolute limit} of $\MM$, the $\MM$-atomic objects are closed under $\II^\op$-shaped colimits (see \cref{atomic-abs}).
These two results are then applied in \cref{tsadi-atomics} to show that for $R \in \Alg(\tsadi)$, we have $\LMod_R^\ldbl \in \Catmfinst$, so that it can be used as an input to higher semiadditive algebraic K-theory.

Another key result is the strong connection between the functor $\PShM$ taking $\MM$-valued presheaves and the functor taking $\MM$-atomic objects.
Let $\Mod_\MM^{\intL}$ denote the subcategory of $\PrL$ consisting of $\infty$-categories in the mode $\MM$ and \emph{internally left adjoint} functors (that is, left adjoint functors whose right adjoint admits a further right adjoint), which inherits a symmetric monoidal structure from $\PrL$.
We then have the following:

\begin{theorem}[\pwrap{\cref{atomic-psh-sm-adj}}]
	There is a symmetric monoidal adjunction
	$$
	\PShM\colon \Cat \rightleftarrows \Mod_\MM^{\intL}\colon (-)^{\atomic[\MM]}
	,
	$$
	i.e.\ $\PShM$ is symmetric monoidal with a lax symmetric monoidal right adjoint $(-)^{\atomic[\MM]}$.
\end{theorem}

Building on the work of Glasman \cite{Saul} and Haugseng--Hebestreit--Linskens--Nuiten \cite[Theorem 8.1]{yon} on the Yoneda embedding, the adjunction is constructed such that the unit is (the factorization through the $\MM$-atomic objects of) the Yoneda map $\yonM\colon \CC_0 \to \PShM(\CC_0)$, reproducing the ordinary Yoneda embedding for $\MM = \Spaces$.
As an immediate consequence, we obtain a monoidal and natural version of the Yoneda map for any operad $\OO$, which may be of independent interest.

\begin{theorem}[\pwrap{\cref{yoneda-O}}]
	The Yoneda map $\yonM\colon \CC_0 \to \PShM(\CC_0)$ is $\OO$-monoidal and natural in $\CC_0 \in \Alg_\OO(\Cat)$.
\end{theorem}

\subsection{Organization}

In \cref{sec-at}, we develop the notion of $\MM$-atomic objects in a presentable $\infty$-category in the mode $\MM$.
We study the connection between $\MM$-atomic objects and $\MM$-valued presheaves, and leverage this connection to endow the functor taking the $\MM$-atomic objects with a lax symmetric monoidal structure.
As a byproduct, we obtain a monoidal natural version of the Yoneda map.

In \cref{sec-day}, we recall the universal property of the Day convolution, and study its functoriality in the source and the target.

In \cref{sec-cmon}, we recall some facts about ($p$-typical) (pre-)$m$-commutative monoids, and study their multiplicative structure.
We observe that the $\infty$-category of $m$-commutative monoids can naturally be endowed with two symmetric monoidal structures, and we show that these two structures coincide.

In \cref{sec-cocart}, we recall the definition of the higher cocartesian structure, and show that it satisfies certain expected properties.
In particular, we show that tensoring a family of objects is indeed given by their colimit.

In \cref{sec-k}, we define $m$-semiadditive algebraic K-theory using the tools developed in the previous sections, and study its properties.
We construct it in two different ways, first using the $S_\bullet$-construction, and second by exhibiting it as the universal way to make ordinary algebraic K-theory into an $m$-semiadditive functor.
We leverage the second definition of $m$-semiadditive algebraic K-theory to endow it with a lax symmetric monoidal structure.

In \cref{sec-red}, we study the interplay between $m$-semiadditive algebraic K-theory and semiadditive height.
In particular, we show that it can increase the height of rings at most by one.
Furthermore, we show that if the ring has (height $n$) $p$-th roots of unity, then the height of its $m$-semiadditive algebraic K-theory is exactly $n+1$.

In \cref{sec-chrom-k}, we study the connection between higher semiadditive algebraic K-theory and chromatically localized K-theory.
We apply the Quillen--Lichtenbaum conjecture and the Galois descent result for chromatically localized K-theory, to show that the higher semiadditive algebraic K-theory of $p$-invertible algebras coincides with their $\To$-localized algebraic K-theory.
Finally, we use the Quillen--Lichtenbaum result for $\BPn$ to show that the higher semiadditive algebraic K-theory of $\cJW$-algebras lands in $\Tnpo$-local spectra, and that specifically their $1$-semiadditive algebraic K-theory coincides with their $\To$-localized algebraic K-theory.

\subsection{Conventions}

Throughout the paper, we work in the framework of $\infty$-categories, mostly following the notations of \cite{HTT, HA}.
For brevity, we use the word category to mean an $\infty$-category.
We also generally follow the notation and terminology of \cite{AmbiHeight} related to higher semiadditivity, but we diverge by working exclusively in the $p$-typical case.

\begin{enumerate}
	\item We denote the space of morphisms between two objects $X, Y \in \CC$ by $\mdef{\hom_\CC(X, Y)}$ and omit $\CC$ when it is clear from the context. If $\CC$ is $\DD$-enriched (e.g.\ in a mode $\DD = \MM$, or closed symmetric monoidal $\DD = \CC$), we denote by $\mdef{\hom^\DD_\CC(X, Y)}$ the $\DD$-object of morphisms and omit $\CC$ when it is clear from the context.
	
	\item We say that a space $A \in \Spaces$ is
	\begin{enumerate}
		\item a \tdef{$p$-space}, if all the homotopy groups of $A$ are $p$-groups.
		\item \tdef{$m$-finite} for $m \ge -2$, if $m = -2$ and $A$ is contractible,
		or $m \ge -1$, the set $\pi_{0}A$ is finite and all the fibers of the diagonal map
		$\Delta\colon A\to A\times A$ are $(m-1)$-finite.\footnote{For $m \ge 0$, this is equivalent to $A$ having finitely many components, each of them $m$-truncated with finite homotopy groups.}
		\item \tdef{$\pi$-finite} or $\infty$-finite, if it is $m$-finite for some integer $m \ge -2$.
	\end{enumerate}

	\item For $-2 \leq m \leq \infty$, we denote by $\mdef{\Spacesm} \subset \Spaces$ the full subcategory spanned by all $m$-finite $p$-spaces.

	\item We say that a category $\CC$ is \tdef{($p$-typically) $m$-semiadditive} if all $m$-finite $p$-spaces $A \in \Spacesm$ are $\CC$-ambidextrous.
	
	\item We denote by $\mdef{\Catst} \subset \Cat$ the subcategory spanned by all stable categories and exact
	functors.
	\item For a collection $\clK$ of indexing categories, we let $\mdef{\CatK} \subset \Cat$ be the subcategory spanned by all categories admitting all colimits indexed by $\II \in \clK$ and functors preserving them.
	\item For $-2 \leq m \leq \infty$, we define $\mdef{\Catmfin} = \CatK$ for $\clK = \Spacesm$, and we let $\mdef{\Catmfinst} \subset \Catmfin$ be the subcategory of those categories which are additionally stable and functors which are additionally exact.
\end{enumerate}

	\ifarxiv
	\ifcompositio
\begin{acknowledgements}
\else
\subsection{Acknowledgements}
\fi
	We are grateful to the anonymous referees for valuable suggestions and corrections, particularly for suggesting the strategy for \cref{Km-cJW-intro}.
	We would like to thank the entire Seminarak group, especially Shaul Barkan, Shachar Carmeli, Shaul Ragimov and Lior Yanovski, for useful discussions on different aspects of this project, and for their valuable comments on the paper's first draft.
	In particular, we would like to thank Lior Yanovski for sharing many ideas regarding atomic objects appearing in \cref{sec-at}, and for suggesting the proof of \cref{p-ls-prl}.
	We would also like to thank Bastiaan Cnossen and Maxime Ramzi for helpful comments on the paper's first version.
	\ifarxiv
	The second author is supported by ISF1588/18 and BSF 2018389.
	\fi
\ifcompositio
\end{acknowledgements}
\fi

	\fi
	\section{Atomic Objects}\label{sec-at}

Let $\MM$ be a mode, that is, an idempotent algebra in $\PrL$ (see \cite[Section 5]{AmbiHeight} for generalities on modes).
In this section we study \emph{$\MM$-atomic objects} (see \cref{atomic-def}), a finiteness property of objects in categories $\CC \in \Mod_\MM$ in the mode $\MM$, which generalizes both compactness (for the case $\MM = \Sp$, see \cref{stable-atomic}) and dualizability of modules (for the case $\CC = \LMod_R(\MM)$, see \cref{atomic-ldbl}).
The results of this section are subsequently used in \cref{ring-sa-k-def} to define the higher semiadditive algebraic K-theory of algebras in $\tsadim$ (see \cref{tsadim-def}), and in particular for algebras in $\SpTn$, including its lax symmetric monoidal structure.

In \cref{subsec-at} we give the definition of atomic objects (see \cref{atomic-def}) and study their basic properties.
We show that taking atomic objects is functorial in \emph{internally left adjoint} functors (see \cref{il-def}).
Analogously to the condition of being compactly generated, we study the condition of being generated under colimits and the action of $\MM$ from the $\MM$-atomic objects, which we call being \emph{$\MM$-molecular} (see \cref{molecular-def}), and we explain its relationship to internally left adjoint functors.
Lastly, in \cref{atomic-abs} we show that for any \emph{absolute limit} $\II$ of $\MM$ (see \cref{abs-def}), the atomic objects are closed under $\II^\op$-shaped colimits.
This yields a functor $(-)^{\atomic} \colon \Mod_\MM^{\intL} \to \CatK$ where $\clK$ is any small collection of opposites of absolute limits of $\MM$.

In \cref{subsec-at-psh} and \cref{subsec-at-tensor} we study the connection between $\MM$-atomic objects and \emph{$\MM$-valued presheaves} (see \cref{M-psh}), and the multiplicative structure of both functors.
The main result of this section is \cref{atomic-psh-sm-adj}, exhibiting a symmetric monoidal adjunction
$$
\PShKM\colon \CatK \rightleftarrows \Mod_\MM^{\intL}\colon (-)^{\atomic}
.
$$
Moreover, the unit of this adjunction is the Yoneda map, and as an immediate consequence, \cref{yoneda-O} shows that $\yonKM\colon \CC_0 \to \PShKM(\CC_0)$ is $\OO$-monoidal and natural in $\CC_0 \in \Alg_\OO(\CatK)$, which may be of independent interest.

Lastly, in \cref{subsec-at-mod} we study atomic objects in categories of left modules.
In \cref{atomic-ldbl} we show that atomic objects and left dualizable left modules coincide, i.e.
$\LMod_R^{\atomic} = \LMod_R^\ldbl$.

\begin{remark}\label{at-enrich}
	Many of the results of this section can be generalized to modules in $\PrL$ over any presentably monoidal category $\VV \in \Alg(\PrL)$ and $\VV$-linear functors.
	Parts of these generalizations were carried out by the first author in \cite{NatYon}, building on the works of \cite{enriched,Hin,Hin2,EnAct}.
	The main feature of modes is that being an $\MM$-module is a \emph{property} rather than extra structure, and that any left adjoint functor is automatically $\MM$-linear.
	As such, working over a mode simplifies the definitions and proofs, and avoids using enriched category theory.
	Since this suffices for our applications in the rest of the paper, we have restricted to this case.
\end{remark}

\subsection{Atomics and Internally Left Adjoints}\label{subsec-at}

\begin{lem}\label{M-adj}
	Let $F\colon \CC \rightleftarrows \DD\colon G$ be an adjunction with $\CC, \DD \in \Mod_\MM$.
	Let $X \in \CC$ and $Y \in \DD$, then there is an equivalence
	$\hom^\MM(FX, Y) \cong \hom^\MM(X, GY)$
	lifting the equivalence
	$\hom(FX, Y) \cong \hom(X, GY)$.
\end{lem}

\begin{proof}
	We prove this using the Yoneda lemma.
	Let $m \in \MM$.
	Recall that $F\colon \CC \to \DD$ is a map in $\Mod_\MM$, so that it commutes with $m \otimes -$, so we conclude that
	\begin{align*}
		\hom(m, \hom^\MM(FX, Y))
		&\cong \hom(m \otimes FX, Y)\\
		&\cong \hom(F(m \otimes X), Y)\\
		&\cong \hom(m \otimes X, GY)\\
		&\cong \hom(m, \hom^\MM(X, GY))
		.
	\end{align*}
	The fact that it is a lift of the $\Spaces$-enriched hom is the case $m = \unit_\MM$.
\end{proof}

\begin{defn}\label{atomic-def}
	Let $\CC \in \Mod_\MM$.
	An object $X \in \CC$ is called \tdef{$\MM$-atomic}, if $\hom^\MM(X, -)\colon \CC \to \MM$ commutes with colimits.
	We denote by $\mdef{\CC^{\atomic[\MM]}} \subseteq \CC$ the full subcategory of the $\MM$-atomic objects.
	When the mode is clear from the context, it is dropped from the notation.
\end{defn}

\begin{remark}
	The definition of atomic objects will be made functorial in \cref{atomic-def-fun}.
\end{remark}

\begin{remark}
	If $X \in \CC$ is atomic then $\hom^\MM(X, -)$ is a left adjoint functor, thus a morphism in $\Mod_\MM$, so that it also commutes with the action of $\MM$.
	That is, for any $m \in \MM$ we have $\hom^\MM(X, - \otimes m) \cong \hom^\MM(X, -) \otimes m$.
\end{remark}

\begin{example}
	The unit $\unit_\MM \in \MM$ is atomic because the functor $\hom^\MM(\unit_\MM, -)\colon \MM \to \MM$ is the identity functor and in particular commutes with colimits.
\end{example}

\begin{prop}
	The only $\Spaces$-atomic object in $\Spaces$ is the point $*$.
\end{prop}

\begin{proof}
	Let $X \in \Spaces^{\atomic}$ be atomic, then for any $Y \in \Spaces$ we have
	$$
	\hom(X, Y) \cong \hom(X, \colim_Y *)
	\cong \colim_Y \hom(X, *)
	\cong \colim_Y *
	\cong Y
	.
	$$
	Thus $X$ corepresents the identity functor $\id\colon \Spaces \to \Spaces$, namely $X = *$.
\end{proof}

\begin{prop}\label{stable-atomic}
	Let $\CC \in \Mod_\Sp$ be a presentable stable category, then the $\Sp$-atomics are the compact objects, i.e.\ $\CC^{\atomic[\Sp]} = \CC^{\omega}$.
\end{prop}

\begin{proof}
	Let $X \in \CC$.
	First assume that $X$ is atomic.
	Recall that $\Omega^\infty\colon \Sp \to \Spaces$ commutes with filtered colimits, so that $\hom(X, -) \cong \Omega^\infty \hom^\Sp(X, -)$ commutes with filtered colimits, i.e.\ it is compact.

	Now assume that $X$ is compact.
	Recall that for any $n \in \ZZ$, the functor $\Sigma^n\colon \Sp \to \Sp$ commutes with all limits and colimits and in particular with filtered colimits, thus $\hom(X, \Sigma^n -) \cong \Omega^\infty \Sigma^n \hom^\Sp(X, -)$ also commutes with filtered colimits.
	Additionally, the functors $\Omega^\infty \Sigma^n\colon \Sp \to \Spaces$ are jointly conservative, implying that $\hom^\Sp(X, -)$ commutes with filtered colimits.
	Furthermore, it commutes with finite limits, thus by stability also with all finite colimits, which together with filtered colimits generate all colimits.
\end{proof}

\begin{prop}\label{atomic-compact}
	Let $\CC \in \Mod_\MM$, then $\CC^{\atomic} \in \Cat$ is a small category.
\end{prop}

\begin{proof}
	Let $\kappa$ be a regular cardinal such that the unit $\unit_\MM \in \MM$ is $\kappa$-compact.
	We show that $\CC^{\atomic} \subseteq \CC^\kappa$, that is the atomics are $\kappa$-compact.
	Let $X \in \CC$ be an atomic object, so in particular $\hom^\MM(X,-)$ commutes with $\kappa$-filtered colimits.
	Since $\unit_\MM$ is $\kappa$-compact, $\hom(\unit_\MM,-)$ commutes with $\kappa$-filtered colimits,
	implying that the composition $\hom(X,-) \cong \hom(\unit_\MM, \hom^\MM(X,-))$ commutes with $\kappa$-filtered colimits.
\end{proof}

\begin{defn}\label{molecular-def}
	Let $\CC \in \Mod_\MM$.
	We say that a collection of atomic objects $B \subseteq \CC^{\atomic}$ are \tdef{$\MM$-atomic generators}, if $\CC$ is generated from $B$ under colimits and the action of $\MM$.\footnote{In a previous version of this paper, $B \subseteq \CC^{\atomic}$ was called a collection of atomic generators if they generate $\CC$ under colimits (without the action of $\MM$). It was later noticed that our proofs work under the new, weaker, assumption.}
	If such $B$ exists, we say that $\CC$ is \tdef{$\MM$-molecular}.\footnote{As a result of the weakening of the condition of being atomic generators, the condition of being molecular is correspondingly weaker.}
	If the mode is clear from the context, we call $\CC$ molecular and say that $B$ are atomic generators.
\end{defn}

\begin{example}
	Every mode $\MM$ is itself $\MM$-molecular, because the unit $\unit_\MM$ is atomic and any object $m$ can be written as $m \otimes \unit_\MM$.\footnote{In the previous version it was posed as a question whether every mode is molecular. With the new definition of atomic generators, this always holds. However, the question whether every mode is generated from atomic objects under colimits (without the action of $\MM$) is still open.}
\end{example}

\begin{defn}\label{il-def}
	Let $\CC, \DD \in \PrL$.
	We say that a functor $F\colon \CC \to \DD$ is \tdef{internally left adjoint} if it is left adjoint in $\PrL$,
	namely if it is a left adjoint functor and its right adjoint $G\colon \DD \to \CC$ is itself a left adjoint.
	We denote by $\mdef{\FunintL(\CC, \DD)} \subseteq \FunL(\CC, \DD)$ the full subcategory of internally left adjoint functors.
	We let $\mdef{\Mod_\MM^\intL}$ be the wide subcategory of $\Mod_\MM$ with the same objects, and morphisms the internally left adjoint functors.
\end{defn}

\begin{prop}\label{internally-left-atomic-pres}
	Let $\CC, \DD \in \Mod_\MM$, and let $F\colon \CC \to \DD$ be an internally left adjoint functor, then it sends atomic objects to atomic objects.
\end{prop}

\begin{proof}
	By assumption the right adjoint $G\colon \DD \to \CC$ is itself a left adjoint, thus preserves colimits.
	Let $X \in \CC^{\atomic}$ be an atomic object, then using \cref{M-adj} $\hom^\MM(FX,-) \cong \hom^\MM(X, G-)$, which is the composition of $G$ and $\hom^\MM(X, -)$, both of which preserve colimits, so that $FX$ is atomic.
\end{proof}

\begin{prop}\label{atomic-pres-il}
	Let $\CC, \DD \in \Mod_\MM$, and let $F\colon \CC \to \DD$ be a left adjoint functor.
	If $\CC$ is molecular and $F$ sends a collection of atomic generators $B \subset \CC$ to atomic objects in $\DD$, then $F$ is internally left adjoint.
\end{prop}

\begin{proof}
	We wish to show that $G$, the right adjoint of $F$, is itself a left adjoint, namely that it preserves colimits.
	Let $Y_i\colon \II \to \DD$ be a diagram, and we wish to show that $G(\colim Y_i) \cong \colim G Y_i$.
	By the Yoneda lemma, this is equivalent to checking that for every $X \in \CC$ we have
	$$
	\hom(X, G(\colim Y_i)) \cong \hom(X, \colim G Y_i).
	$$
	Since $\hom(-, -) \cong \hom(\unit_\MM, \hom^\MM(-, -))$, it suffices to check that for every $X \in \CC$ we have
	$$
	\hom^\MM(X, G(\colim Y_i)) \cong \hom^\MM(X, \colim G Y_i).
	$$
	Let $A$ denote the collection of $X \in \CC$ for which this condition holds, and we shall show that $A = \CC$.

	First, for every $X \in B$, we know that
	\begin{align*}
		\hom^\MM(X, G(\colim_I Y_i))
		&\cong \hom^\MM(F X, \colim_I Y_i)\\
		&\cong \colim_I \hom^\MM(F X, Y_i)\\
		&\cong \colim_I \hom^\MM(X, G Y_i)\\
		&\cong \hom^\MM(X, \colim_I G Y_i)
	\end{align*}
	where the first and third steps follow from \cref{M-adj},
	the second step follows from the assumption that $FX$ is atomic since $X \in B$ and $F$ sends $B$ to atomic objects,
	and the fourth step follows from the fact $X$ is atomic.
	Therefore, $B \subseteq A$.

	Second, for every $X \in A$ and $m \in \MM$, we know that
	$\hom^\MM(m \otimes X, -) \cong \hom^\MM(m, \hom^\MM(X, -))$
	so that $m \otimes X \in A$, i.e.\ $A$ is closed under the action of $\MM$.

	Third, for every diagram $X_j\colon \JJ \to \CC$ landing in $A$, we know that
	$\hom^\MM(\colim_J X_j, -) \cong \lim_{\JJ^\op} \hom^\MM(X_j, -)$
	so that $\colim_J X_j \in A$, i.e.\ $A$ is closed under colimits.
	
	We have shown that $B \subseteq A$ and that $A$ is closed under the action of $\MM$ and colimits, and by assumption $B$ are atomic generators, thus $A = \CC$ as needed.
\end{proof}

Recall that for $\CC \in \Mod_\MM$ we have an equivalence $\FunL(\MM, \CC) \iso \CC$ given by evaluation at $\unit_\MM$.
Its inverse sends $X \in \CC$ to the functor $- \otimes X\colon \MM \to \CC$ (part of the data admitting $\CC$ as an $\MM$-module).
Furthermore, the right adjoint of $- \otimes X\colon \MM \to \CC$ is $\hom^\MM(X, -)\colon \CC \to \MM$.

\begin{prop}\label{atomic-tensor}
	Let $\CC \in \Mod_\MM$,
	then the equivalence $\FunL(\MM, \CC) \iso \CC$ restricts to an equivalence $\FunintL(\MM, \CC) \iso \CC^{\atomic}$.
	In addition, if $X \in \CC$ and $m \in \MM$ are atomic then so is $m \otimes X \in \CC$.
\end{prop}

\begin{proof}
	First, by \cref{internally-left-atomic-pres} and the fact that the $\unit_\MM$ is atomic, the functor indeed lands in the full subcategory $\CC^{\atomic}$.
	In particular, it is also fully faithful as the restriction of an equivalence to two full subcategories.
	We need to show that it is essentially surjective, i.e.\ that if $X \in \CC^{\atomic}$ then $- \otimes X\colon \MM \to \CC$ is internally left adjoint.
	This holds since its right adjoint is $\hom^\MM(X, -)\colon \CC \to \MM$, which by assumption preserves colimits.

	For the last part, as $- \otimes X\colon \MM \to \CC$ is internally left adjoint, \cref{internally-left-atomic-pres} implies that it sends atomic objects to atomic objects.
\end{proof}

\begin{remark}
	In \cref{atomics-mod-m-at} we extend the last part of the proposition to show that in fact $\CC^{\atomic}$ is a module over $\MM^{\atomic}$.
\end{remark}

In light of this proposition, we construct the functor of taking atomics functorially.
We also recall from \cref{atomic-compact} that $\CC^{\atomic}$ is a small category.

\begin{defn}\label{atomic-def-fun}
	We define the functor $(-)^{\atomic} \colon \Mod_\MM^{\intL} \to \Cat$ by $(-)^{\atomic} = \FunintL(\MM, -)$.
\end{defn}

\begin{defn}\label{abs-def}
	Let $\II$ be an indexing category.
	We say that $\II$ is an \tdef{absolute limit} of $\MM$ if for any $\CC \in \Mod_\MM$, $\II$-shaped limits in $\CC$ commute with colimits.
\end{defn}

\begin{remark}
	The term absolute limit is usually used in the context of enriched categories, saying that $\II$ is an absolute limit of $\VV \in \Mon(\Cat)$ if any $\VV$-enriched functor commutes with $\II$-shaped limits.
	We will not use this condition in this paper, but for the convenience of the reader we remark on the connection between this condition and the one appearing in \cref{abs-def} when $\VV$ is a mode.

	Assume that $\II$ is an absolute limit in the ordinary sense, namely that $\VV$-enriched functors commute with $\II$-shaped limits.
	Let $\CC \in \Mod_\VV$.
	For any indexing category $\JJ$ consider $\colim_\JJ\colon \CC^\JJ \to \CC$.
	This functor commutes with colimits, and since $\VV$ is a mode, it is a morphism in $\Mod_\VV$, so, as referred to in \cref{at-enrich}, it is canonically $\VV$-enriched, and therefore commutes with $\II$-shaped limits.
	This holds for any $\JJ$, meaning that $\II$-shaped limits in $\CC$ commute with colimits, reproducing \cref{abs-def}.
	
	The implication in the other direction should follow from a working theory of enriched left Kan extensions and their compatibility with the enriched Yoneda embedding, which we are unaware of a reference for.
\end{remark}

\begin{lem}\label{abs-mod}
	If $\II$ is an absolute limit of $\MM$ and $\MM \to \NN$ is map of modes, then $\II$ is an absolute limit of $\NN$ as well.
\end{lem}

\begin{proof}
	This is immediate from the fact that $\Mod_\NN \subseteq \Mod_\MM$.
\end{proof}

\begin{lem}\label{abs-m}
	Let $\II$ be an absolute limit of $\MM$, and $\CC \in \Mod_\MM$.
	Then $m \otimes -\colon \CC \to \CC$ commutes with $\II$-shaped limits, for any $m \in \MM$.
\end{lem}

\begin{proof}
	By assumption, $\lim_\II\colon \CC^\II \to \CC$ commutes with colimits.
	Therefore, it is a map in $\Mod_\MM$, so that it also commutes $m \otimes -\colon \CC \to \CC$ for any $m \in \MM$.
\end{proof}

\begin{prop}
	Let $\II$ be an absolute limit of $\MM$, then for any $\CC \in \Mod_\MM$, the atomics $\CC^{\atomic} \subset \CC$ are closed under $\II^\op$-shaped colimits.
\end{prop}

\begin{proof}
	Let $X_i\colon \II^\op \to \CC$ be a diagram landing in the atomics.
	Recall that $\hom^\MM(-, -)\colon \CC^\op \times \CC \to \MM$ commutes with limits in the first coordinate, thus $\hom^\MM(\colim_{\II^\op} X_i, -)$ is equivalent to
	$$
	\CC
	\xrightarrow{\Delta} \CC^{\II^\op}
	\xrightarrow{(\hom^\MM(X_i, -))_\II} \MM^\II
	\xrightarrow{\lim_\II} \MM
	.
	$$
	$\Delta$ commutes with colimits since colimits in functor categories are computed level-wise.
	Since each $X_i$ is atomic, each $\hom^\MM(X_i, -)$ commutes with colimits, and as colimits in functor categories are computed level-wise, we get that $(\hom^\MM(X_i, -))_\II$ commutes with colimits.
	By assumption, $\II$ is an absolute limit of $\MM$, thus $\lim_\II$ commutes with colimits.
	This shows that $\hom^\MM(\colim_{\II^\op} X_i, -)$ commutes with colimits, i.e.\ that $\colim_{\II^\op} X_i$ is indeed atomic.
\end{proof}

\begin{remark}
	Let $F\colon \CC \to \DD$ be an internally left adjoint functor, and let $\II$ be an absolute limit of $\MM$.
	Then $F$ preserves colimits, and the atomics are closed under $\II^\op$-shaped colimits, so that the induced functor between the atomics preserves $\II^\op$-shaped colimits.
\end{remark}

The following claim immediately follows.

\begin{prop}\label{atomic-abs}
	Let $\clK \subset \{ \II^\op \mid \II \text{ absolute limit of } \MM \}$ be a collection of opposites of absolute limits of $\MM$, then the functor of taking atomics $(-)^{\atomic} \colon \Mod_\MM^{\intL} \to \Cat$ factors through $\CatK$.
\end{prop}

\begin{prop}\label{absolute-stable}
	Let $\MM$ be a stable mode, then all finite categories are absolute limits.
	Furthermore, for any $\CC \in \Mod_\MM$, $\CC^{\atomic[\MM]}$ is a stable subcategory of $\CC$.
\end{prop}

\begin{proof}
	Recall that $\CC$ itself is stable, so the first part follows from the commutativity of finite limits and colimits in stable categories.
	For the second part, first note that the zero object is obviously atomic.
	As finite limits are absolute, the atomics are closed under finite colimits, so it suffices to show that the atomics are also closed under desuspensions.
	Let $X \in \CC^{\atomic[\MM]}$, then indeed
	$\hom^\MM(\Sigma^{-1} X, -) \cong \Sigma \hom^\MM(X, -)$
	is colimit preserving, because $X$ is atomic and $\Sigma\colon \MM \to \MM$ is colimit preserving.
\end{proof}

\begin{prop}\label{atomic-smashing}
	Let $F\colon \MM \to \NN$ be a smashing localization of modes (see \cite[Definition 5.1.2]{AmbiHeight}),
	and let $\CC \in \Mod_\NN$.
	Then the $\NN$-atomics coincide with the $\MM$-atomics,
	that is $\CC^{\atomic[\NN]} = \CC^{\atomic[\MM]}$.
\end{prop}

\begin{proof}
	By \cite[Proposition 5.2.15]{AmbiHeight}, the localization is smashing if and only if the right adjoint of $F\colon \MM \to \NN$ is itself a left adjoint (i.e.\ if $F$ is internally left adjoint).
	This allows us to treat $\NN$ as a full subcategory of $\MM$ closed under both limits and colimits.
	Now, \cite[Proposition 5.2.10]{AmbiHeight} shows that $\NN$ classifies the property of being in the mode $\MM$ such that all $\MM$-enriched $\hom$'s land in $\NN$, i.e.\ $\hom^\MM(X,-) = \hom^\NN(X,-)$, thus by the closure of $\NN \subseteq \MM$ under colimits, $\hom^\MM(X,-)\colon \CC \to \MM$ commutes with colimits if and only if $\hom^\NN(X,-)\colon \CC \to \NN$ does.
\end{proof}

\subsection{Atomics and Presheaves}\label{subsec-at-psh}

Throughout this subsection, let $\clK \subset \{ \II^\op \mid \II \text{ absolute limit of } \MM \}$ be some small collection of opposites of absolute limits of $\MM$ (not necessarily all of them, for instance, $\clK$ is allowed to be empty).
We also let $\clKop = \{ \II \mid \II^\op \in \clK \}$ be the collection of all of the opposite categories.

\begin{defn}\label{M-psh}
	For $\CC_0 \in \Cat$, we define the category of \tdef{$\MM$-valued presheaves} by $\mdef{\PShM(\CC_0)} = \Fun(\CC_0^\op, \MM)$.
	If $\CC_0 \in \CatK$, we let $\mdef{\PShKM(\CC_0)} = \Fun^\clKop(\CC_0^\op, \MM)$ be the full subcategory of $\PShM(\CC_0)$ on those functors that preserve all limits indexed by $\II \in \clKop$, namely functors $F\colon \CC_0^\op \to \MM$ that send $\II^\op$-shaped colimits in $\CC_0$ to $\II$-shaped limits in $\MM$.
\end{defn}

\begin{remark}
	The definition will be made functorial in \cref{psh-fun-def}.
\end{remark}

For the case $\MM = \Spaces$, \cite[Lemma 10.6]{PK-pres} shows that $\PShK(\CC_0)$ is presentable.
From this we deduce the following:

\begin{prop}\label{pshkm-mod-m}
	There is an equivalence $\PShKM(\CC_0) \cong \PShK(\CC_0) \otimes \MM$, and in particular it is presentable and in the mode $\MM$.
\end{prop}

\begin{proof}
	Indeed, we have an equivalence
	\begin{align*}
		\PShK(\CC_0) \otimes \MM
		&\cong \FunR(\PShK(\CC_0)^\op, \MM)\\
		&\cong \FunL(\PShK(\CC_0), \MM^\op)^\op\\
		&\cong \Fun_\clK(\CC_0, \MM^\op)^\op\\
		&\cong \Fun^\clKop(\CC_0^\op, \MM)\\
		&= \PShKM(\CC_0)
		,
	\end{align*}
	where the first equality is \cite[Proposition 4.8.1.17]{HA}, the second is passing to the opposite, the third is the universal property of $\PShK$ given in \cite[Corollary 5.3.6.10]{HTT}, the fourth is by passing to the opposite, and the last is by definition.
\end{proof}

\begin{lem}\label{psh-k-closure}
	Let $\CC_0 \in \CatK$, then $\PShKM(\CC_0) \subseteq \PShM(\CC_0)$ is closed under limits and colimits, thus the inclusion has both adjoints.
\end{lem}

\begin{proof}
	Let $F_j\colon \JJ \to \PShM(\CC_0)$ be a diagram landing in $\PShKM(\CC_0)$.
	We need to show that $\colim_\JJ F_j$ and $\lim_\JJ F_j$ are again in $\PShKM(\CC_0)$, i.e.\ that they commute with all limits indexed by $\II \in \clKop$.
	Let $X_i\colon \II^\op \to \CC_0$ be a diagram.
	Using the fact that colimits and limits in functor categories are computed level-wise, and that $\II$ is an absolute limit, we get:
	$$
	\colim_\JJ F_j(\colim_{\II^\op} X_i)
	\cong \colim_\JJ \lim_\II F_j(X_i)
	\cong \lim_\II \colim_\JJ F_j(X_i)
	.
	$$
	Similarly $\lim_\JJ F_j \in \PShKM(\CC_0)$, since limits commute with limits.
\end{proof}

\begin{defn}\label{L-K-def}
	Let $\CC_0 \in \CatK$.
	We denote by $\mdef{L_\clK}\colon \PShM(\CC_0) \to \PShKM(\CC_0)$ the (internally) left adjoint of the inclusion $\PShKM(\CC_0) \subseteq \PShM(\CC_0)$.
\end{defn}

\begin{lem}\label{f*-psh-k}
	For $f\colon \CC_0 \to \DD_0$, a morphism in $\CatK$, the restriction of $f^*\colon \PShM(\DD_0) \to \PShM(\CC_0)$ to $\PShKM(\DD_0)$ lands in $\PShKM(\CC_0)$.
\end{lem}

\begin{proof}
	$f^*$ is given by pre-composition with $f^\op\colon \CC_0^\op \to \DD_0^\op$, which preserves limits indexed by $\II \in \clKop$, as the opposite of a morphism in $\CatK$.
\end{proof}

\begin{lem}\label{f*-adjoint}
	For $f\colon \CC_0 \to \DD_0$ a morphism in $\CatK$, the functor $f^*\colon \PShKM(\DD_0) \to \PShKM(\CC_0)$ preserves all limits and colimits and thus has a right adjoint $f_*$ and a left adjoint $f_!$.
\end{lem}

\begin{proof}
	By \cref{psh-k-closure}, $\PShKM(\CC_0)$ is closed under limits and colimits in $\PShM(\CC_0)$, which are thus computed level-wise, and similarly for $\DD_0$.
	Therefore, we get
	$$
	f^*(\colim_\II F_i)(c)
	= (\colim_\II F_i)(fc)
	\cong \colim_\II F_i(fc)
	= \colim_\II f^*F_i(c)
	\cong (\colim_\II f^*F_i)(c)
	,
	$$
	showing that $f^*$ commutes with colimits, and similarly for limits.
\end{proof}

\cref{f*-psh-k} shows that the functor $\Fun((-)^\op, \MM)\colon \CatK^\op \to \CAT$, sending $\CC_0$ to $\PShM(\CC_0)$ and $f$ to $f^*$, has a subfunctor $\Fun^\clKop((-)^\op, \MM)\colon \CatK^\op \to \CAT$ sending $\CC_0$ to $\PShKM(\CC_0)$ and $f$ to $f^*$.
By \cref{pshkm-mod-m}, the categories $\PShKM(\CC_0)$ are in the mode $\MM$, and by \cref{f*-adjoint}, the morphism $f^*$ is a right adjoint, so that the functor factors as $\Fun^\clKop((-)^\op, \MM)\colon \CatK^\op \to \PrR$.

\begin{defn}\label{psh-fun-def}
	We define the functor $\PShKM\colon \CatK \to \Mod_\MM$ by passing to the left adjoints in $\Fun^\clKop((-)^\op, \MM)\colon \CatK^\op \to \PrR$, that is the functor sending $\CC_0$ to $\PShKM(\CC_0)$ and $f\colon \CC_0 \to \DD_0$ to $f_!\colon \PShKM(\CC_0) \to \PShKM(\DD_0)$.
\end{defn}

\begin{prop}\label{pshm-il}
	The functor $\PShKM\colon \CatK \to \Mod_\MM$ lands in $\Mod_\MM^{\intL}$.
\end{prop}

\begin{proof}
	\cref{f*-adjoint} shows that $f_!\colon \PShKM(\CC_0) \to \PShKM(\DD_0)$ is internally left adjoint.
\end{proof}

\begin{prop}\label{sh-natural}
	There is a natural transformation $L_\clK\colon \PShM \Rightarrow \PShKM$ of functors $\CatK \to \Mod_\MM$, making the construction of \cref{L-K-def} natural.
\end{prop}

\begin{proof}
	Since $\Fun^\clKop((-)^\op, \MM)\colon \CatK^\op \to \PrR$ is a subfunctor of $\Fun((-)^\op, \MM)\colon \CatK^\op \to \PrR$, there is a natural transformation from the former to the latter given by the inclusion.
	Applying \cite[Corollary 4.7.4.18 (3)]{HA} for $S = \CatK^\op$ (and considering our functors as landing in $\CAT$) shows that by passing to the left adjoints in the target and in the natural transformation, we obtain a natural transformation $L_\clK\colon \PShM \Rightarrow \PShKM$.
	Indeed, $L_\clK$ was defined in \cref{L-K-def} as the left adjoint of the inclusion.
\end{proof}

\begin{defn}
	We define the \tdef{Yoneda map} $\mdef{\yonKM}\colon \CC_0 \to \PShKM(\CC_0)$ as the composition
	$$
		\CC_0
		\xrightarrow{\yon} \PSh(\CC_0)
		\to \PSh(\CC_0) \otimes \MM
		\cong \PShM(\CC_0)
		\xrightarrow{L_\clK} \PShKM(\CC_0)
		,
	$$
	where $\yon$ is the ordinary Yoneda embedding, and the second map is given by tensoring with the unit map $\Spaces \to \MM$.
\end{defn}

\begin{remark}
	Generally, the Yoneda map $\yonKM\colon \CC_0 \to \PShKM(\CC_0)$ is \emph{not} fully faithful.
	For example, in the case where $\clK = \emptyset$ and $\CC_0 = *$, the map is $\yonKM\colon * \to \MM$, which induces $\hom(*, *) \to \hom(\unit_\MM, \unit_\MM)$ that is usually not an equivalence.
\end{remark}

\begin{prop}\label{yoneda-natural}
	The Yoneda map can be upgraded to a natural transformation $\yonKM\colon \iota_{\clK} \Rightarrow \PShKM$ from the inclusion $\iota_{\clK}\colon \CatK \to \Cat \to \CAT$ to $\CatK \xrightarrow{\PShKM} \PrL \to \CAT$.
\end{prop}

\begin{proof}
	The natural transformation is obtained by the following diagram:
	
	\[\begin{tikzcd}
		&&& {} \\
		\CatK & \Cat & \PrL & {} & \PrL & \CAT \\
		&& {} \\
		&& {}
		\arrow[hook, from=2-1, to=2-2]
		\arrow["\PSh", from=2-2, to=2-3]
		\arrow["\PShKM"', curve={height=40pt}, from=2-1, to=2-5]
		\arrow[""{name=0, anchor=center, inner sep=0}, "{- \otimes \MM}"', curve={height=10pt}, from=2-3, to=2-5]
		\arrow[hook, from=2-5, to=2-6]
		\arrow["{L_\clK}"', shift right=4, Rightarrow, from=2-3, to=3-3]
		\arrow["\yon"'{pos=0.3}, shorten <=2pt, shorten >=12pt, Rightarrow, from=1-4, to=2-4]
		\arrow[""{name=1, anchor=center, inner sep=0}, curve={height=-10pt}, equals, from=2-3, to=2-5]
		\arrow["\iota", curve={height=-40pt}, from=2-2, to=2-6]
		\arrow["u"', shorten <=2pt, shorten >=2pt, Rightarrow, from=1, to=0]
	\end{tikzcd}\]

	Here $\yon\colon \iota \Rightarrow \PSh$ is the ordinary Yoneda natural transformation constructed in \cite[Theorem 8.1]{yon},
	the natural transformation $u\colon \id \Rightarrow - \otimes \MM$ is the unit map of the free-forgetful adjunction $- \otimes \MM\colon \PrL \rightleftarrows \Mod_\MM\colon \underlying$ given by tensoring with $\Spaces \to \MM$,
	and $L_\clK\colon \PShM \Rightarrow \PShKM$ is the natural transformation of \cref{sh-natural}.
\end{proof}

\begin{prop}
	For $X \in \CC_0$ and $F \in \PShKM(\CC_0)$ we have $\hom^\MM(\yonKM(X), F) \cong F(X)$.
\end{prop}

\begin{proof}
	We first reduce to the case where $\clK = \emptyset$.
	Since $L_\clK\colon \PShM(\CC_0) \to \PShKM(\CC_0)$ is the left adjoint of the inclusion, using \cref{M-adj} we get
	$$
	\hom^\MM(\yonKM(X), F)
	\cong \hom^\MM(L_\clK \yonM(X), F)
	\cong \hom^\MM(\yonM(X), F)
	.
	$$
	We finish the proof by showing that the latter is equivalent to $F(X)$, using the Yoneda lemma in the category $\MM$.
	Indeed, let $m \in \MM$ be any object, then
	\begin{align*}
		\hom(m, \hom^\MM(\yonM(X), F))
		&\cong \hom(m \otimes \yonM(X), F)\\
		&\cong \hom(\yonM(X), \hom^\MM(m, F))\\
		&\cong \hom(\yon(X), \hom(m, F))\\
		&\cong \hom(m, F)(X)\\
		&\cong \hom(m, F(X))
		,
	\end{align*}
	where the first and second step use the exponential adjunction, the third uses the free-forgetful adjunction $\CC \to \CC \otimes \MM$, the fourth uses the ordinary Yoneda lemma for $\CC_0$ and the last step uses that the action of $\MM$ is level-wise.
\end{proof}

\begin{cor}\label{yoneda-atomic}
	The Yoneda map $\yonKM\colon \CC_0 \to \PShKM(\CC_0)$ lands in the atomics.
\end{cor}

\begin{proof}
	By \cref{psh-k-closure}, $\PShKM(\CC_0)$ is closed under colimits and limits in $\PShM(\CC_0)$,
	which are therefore computed level-wise,
	so that $\hom^\MM(\yonKM(X), F) \cong F(X)$ commutes with all colimits in the $F$-coordinate.
\end{proof}

We use the same notation $\yonKM\colon \CC_0 \to \PShKM(\CC_0)^{\atomic}$ to denote the factorization.
Recall from \cref{yoneda-natural} that the Yoneda map gives a natural transformation $\yonKM\colon \iota_{\clK} \Rightarrow \PShKM$ of functors $\CatK \to \CAT$.
Since taking the atomics lands in $\CatK$ by \cref{atomic-abs}, together with \cref{yoneda-atomic}, we obtain a natural transformation $\yonKM\colon \id \Rightarrow \PShM(-)^{\atomic}$ of functors $\CatK \to \CatK$.

\begin{prop}\label{psh-molecular}
	For any $\CC_0 \in \CatK$ the category $\PShKM(\CC_0)$ is molecular, with atomic generators $\yonKM(X)$ for $X \in \CC_0$.
\end{prop}

\begin{proof}
	We first show the result for $\PShM(\CC_0)$, i.e.\ for the case $\clK=\emptyset$.
	Recall that $\PShM(\CC_0) \cong \PSh(\CC_0) \otimes \MM$ is generated under colimits from the image of $\PSh(\CC_0) \times \MM$, i.e.\ from objects of the form $F \otimes m$.
	Second, $\PSh(\CC_0)$ is generated under colimits from objects of the form $\yon(X)$ for $X \in \CC$.
	Therefore, $\PShM(\CC_0)$ is generated under colimits and the action of $\MM$ from objects of the form $\yonM(X)$ for $X \in \CC$, which are indeed atomic by \cref{yoneda-atomic}.

	For the general case, recall that $L_\clK\colon \PShM(\CC_0) \to \PShKM(\CC_0)$ is an internally left adjoint functor so it sends atomic objects to atomic objects by \cref{internally-left-atomic-pres}, thus $\yonKM(X)$ is atomic for any $X \in \CC$.
	Since it preserves colimits and the action of $\MM$, and $\yonM(X)$ generate $\PShM(\CC_0)$ under these operations, their images $\yonKM(X)$ generate the essential image of $L_\clK$ under these operations.
	In addition, $L_\clK$ is essentially surjective, so that $\yonKM(X)$ are atomic generators of $\PShKM(\CC_0)$ as needed.
\end{proof}

\begin{prop}\label{atomic-psh-adj}
	There is an adjunction
	$$
	\PShKM\colon \CatK \rightleftarrows \Mod_\MM^{\intL}\colon (-)^{\atomic}
	$$
	with unit
	$\yonKM\colon \id \Rightarrow \PShKM(-)^{\atomic}$.
\end{prop}

\begin{proof}
	To check that the data in the theorem supports an adjunction, it suffices to check that for any $\CC_0 \in \CatK$ and $\DD \in \Mod_\MM^{\intL}$, the canonical map
	\begin{equation}\label{pmk-at-check}
		\FunintL(\PShKM(\CC_0), \DD)
		\to \Fun_\clK(\PShKM(\CC_0)^{\atomic}, \DD^{\atomic})
		\xrightarrow{- \circ \yonKM} \Fun_\clK(\CC_0, \DD^{\atomic})
	\end{equation}
	is an equivalence (in fact, it suffices to show this for the hom spaces, rather then the functor categories, but we show that the stronger statement holds).
	Note that
	\begin{equation}\label{pmk-at-lift}
		\FunL(\PShKM(\CC_0), \DD)
		\cong \FunL(\PShK(\CC_0), \DD)
		\cong \Fun_\clK(\CC_0, \DD)
		.
	\end{equation}
	Furthermore, both the first and last categories in \cref{pmk-at-check} are full subcategories of the first and last categories in \cref{pmk-at-lift}, showing that the composition in \cref{pmk-at-check} is also fully faithful.

	To finish the argument, we need to show that \cref{pmk-at-check} is essentially surjective.
	To that end, let $F\colon \CC_0 \to \DD^{\atomic}$ be a functor preserving $\II^\op$-shaped colimits for $\II^\op \in \clK$.
	We can post-compose it with the inclusion $\DD^{\atomic} \to \DD$, and using \cref{pmk-at-lift} we get a left adjoint functor $\tilde{F}\colon \PShKM(\CC_0) \to \DD$, and we need to show that it is in fact internally left adjoint.
	By construction, for any $X \in \CC_0$ we have that $\tilde{F}(\yonKM(X)) \cong F(X) \in \DD^{\atomic}$ is atomic.
	\cref{psh-molecular} shows that these are atomic generators for $\PShKM(\CC_0)$, so \cref{atomic-pres-il} shows that $\tilde{F}$ is indeed internally left adjoint.
\end{proof}

\subsection{Tensor Product of Atomics}\label{subsec-at-tensor}

\begin{prop}\label{il-sm}
	The symmetric monoidal structure on $\Mod_\MM$ restricts to a symmetric monoidal structure on the subcategory $\Mod_\MM^{\intL}$.
\end{prop}

\begin{proof}
	Since $\Mod_\MM^{\intL}$ is a wide subcategory of $\Mod_\MM$, all we need to show is that if $L_i\colon \CC_i \to \DD_i, i=1,2$ are in $\Mod_\MM^{\intL}$, then so is $L_1 \otimes L_2\colon \CC_1 \otimes \CC_2 \to \DD_1 \otimes \DD_2$.
	Let $R_i$ be the right adjoints of $L_i$, which by assumption are themselves left adjoints.
	Because they are left adjoints, we can tensor them to obtain another left adjoint functor $R_1 \otimes R_2\colon \DD_1 \otimes \DD_2 \to \CC_1 \otimes \CC_2$.
	It is then straightforward to check that tensoring the unit and counit of $L_i \dashv R_i$ exhibit an adjunction $L_1 \otimes L_2 \dashv R_1 \otimes R_2$, showing that $L_1 \otimes L_2$ is an internally left adjoint functor.
\end{proof}

We recall that the category $\CatK$ has a symmetric monoidal structure, developed in \cite[\S 4.8.1]{HA}.

\begin{cor}\label{psh-sm}
	The functor $\PShKM\colon \CatK \to \Mod_\MM^{\intL}$ of \cref{pshm-il} is symmetric monoidal.
\end{cor}

\begin{proof}
	Since the symmetric monoidal structure on $\Mod_\MM^{\intL}$ is inherited from $\Mod_\MM$, it suffices to show that $\PShKM\colon \CatK \to \Mod_\MM$ is symmetric monoidal.
	Indeed, \cite[Remark 4.8.1.8]{HA} shows that $\PShK\colon \CatK \to \PrL$ is symmetric monoidal, $- \otimes \MM\colon \PrL \to \Mod_\MM$ is symmetric monoidal, and by \cref{pshkm-mod-m}, $\PShK \otimes \MM \cong \PShKM$.
\end{proof}

Applying \cite[Corollary 7.3.2.7]{HA}, we immediately get:

\begin{thm}\label{atomic-psh-sm-adj}
	The adjunction
	$\PShKM\colon \CatK \rightleftarrows \Mod_\MM^{\intL}\colon (-)^{\atomic}$
	of \cref{atomic-psh-adj} is symmetric monoidal,
	i.e.\ $\PShKM$ is symmetric monoidal with a lax symmetric monoidal right adjoint $(-)^{\atomic}$.
\end{thm}

Note that for any operad $\OO$ we get an induced adjunction
$$
\PShKM\colon \Alg_\OO(\CatK) \rightleftarrows \Alg_\OO(\Mod_\MM^{\intL})\colon (-)^{\atomic}
$$
whose unit is an enhancement of the Yoneda map (landing in the atomics) to $\OO$-algebras.
Furthermore, for any $\CC \in \Alg_\OO(\Mod_\MM^{\intL})$ we see that $\CC^{\atomic} \subset \CC$ is in fact an $\OO$-monoidal subcategory.
We therefore get the following corollary, which generalizes \cite[Section 3]{Saul} and \cite[Corollary 4.8.1.12]{HA} from the case of $\MM = \Spaces, \clK = \emptyset$ and $\OO = \bbE_\infty$ and makes them natural.

\begin{cor}\label{yoneda-O}
	The Yoneda natural transformation lifts to a natural transformation
	$\yonKM\colon \iota_{\clK} \Rightarrow \PShKM(-)$ of functors $\Alg_\OO(\CatK) \to \Alg_\OO(\CAT)$.
	That is, the Yoneda map $\yonKM\colon \CC_0 \to \PShKM(\CC_0)$ is $\OO$-monoidal and natural in $\CC_0 \in \Alg_\OO(\CatK)$, and factors through the atomics $\PShKM(\CC_0)^{\atomic}$.
\end{cor}

Recall that in \cref{atomic-tensor} we showed that if $X \in \CC^{\atomic}$ and $m \in \MM^{\atomic}$ then $m \otimes X \in \CC^{\atomic}$.
Using \cref{atomic-psh-sm-adj}, we strengthen this into a module structure, using the fact that any lax symmetric monoidal functor lands in modules over the image of the unit.

\begin{cor}\label{atomics-mod-m-at}
	The functor of atomic objects factors as a lax symmetric monoidal functor $(-)^{\atomic} \colon \Mod_\MM^{\intL} \to \Mod_{\MM^{\atomic}}(\CatK)$.
\end{cor}

We also mention the following easy corollary of \cref{il-sm}.

\begin{lem}\label{mod-smash-tensor}
	Let $L\colon \MM_1 \to \MM_2$ be a smashing localization of modes and let $\NN$ be another mode
	Then, $L \otimes \id_\NN\colon \MM_1 \otimes \NN \to \MM_2 \otimes \NN$ is also a smashing localization of modes.
\end{lem}

\begin{proof}
	First, by \cite[Lemma 5.2.1]{AmbiHeight}, the functor $L \otimes \id_\NN$ is a localization as well.
	By \cite[Proposition 5.2.15]{AmbiHeight}, a localization of modes is smashing if and only if it is an internally left adjoint functor in $\PrL$.
	Thus, $L$ is internally left adjoint, and by \cref{il-sm} we conclude that $L \otimes \id_\NN$ is also internally left adjoint, so that it is also a smashing localization.
\end{proof}

\subsection{Atomic Modules}\label{subsec-at-mod}

In the remainder of the section we show that the atomic objects in $\LMod_R$ for $R \in \Alg(\MM)$ are the left dualizable left modules (in the sense of \cite[Definition 4.6.2.3]{HA}), summarized in \cref{atomic-ldbl-lax}.
We begin by collecting certain basic facts about the category of left modules from \cite[\S 4.8.5]{HA}.

\begin{thm}\label{mod-func}
	There is a symmetric monoidal functor
	$\LMod_{(-)}\colon \Alg(\MM) \to \Mod_\MM^\intL$,
	sending $R$ to $\LMod_R$ and $f\colon R \to S$ to $f_!\colon \LMod_R \to \LMod_S$.
\end{thm}

\begin{proof}
	Let $\clK$ denote the collection of all (small) categories, then by \cite[Remark 4.8.5.17]{HA} there is a symmetric monoidal functor
	$\LMod_{(-)}\colon \Alg(\MM) \to \Mod_\MM(\CAT_\clK)$.
	As in \cite[Notation 4.8.5.10]{HA}, $\LMod_R$ is presentable and $f_!$ is left adjoint to $f^*$.
	Furthermore, \cite[Corollary 4.2.3.7 (2)]{HA} shows that the right adjoint $f^*$ of $f_!$ is itself a left adjoint, so that the functor lands in $\Mod_\MM^\intL$.
	As the symmetric monoidal structure on $\Mod_\MM^\intL$ is restricted from $\Mod_\MM(\CAT_\clK)$, the factorization $\LMod_{(-)}\colon \Alg(\MM) \to \Mod_\MM^\intL$ is indeed symmetric monoidal.
\end{proof}

We now recall the following result about left dualizability and adjunctions.

\begin{prop}\label{ldbl-adj}
	Let $X \in \LMod_R, Y \in \RMod_R$.
	Then $Y$ is left dual to $X$ if and only if there is an adjunction
	$$
	X \otimes_{\unit_\MM} -\colon \MM \rightleftarrows \LMod_R\colon Y \otimes_R -
	.
	$$
\end{prop}

\begin{proof}
	We explain how this follows from \cite[Proposition 4.6.2.18]{HA}, with $\CC = \MM$, $A = R^{\mrm{rev}}$ and the roles of $X$ and $Y$ reversed (see also \cite[Remark 4.6.3.16]{HA}).

	For the first direction, assume that there is an adjunction and let $\eta\colon \id_\MM \Rightarrow Y \otimes_R X \otimes_{\unit_\MM} -$ be the unit.
	By the adjunction, we know that for each $P \in \MM$ and $Q \in \LMod_R$ the composition
	$$
	\hom(X \otimes_{\unit_\MM} P, Q)
	\to \hom(Y \otimes_R X \otimes_{\unit_\MM} P, Y \otimes_R Q)
	\xrightarrow{- \circ \eta_P} \hom(P, Y \otimes_R Q)
	$$
	is an equivalence.
	Since both functors in the adjunction preserve colimits, and the categories are in the mode $\MM$, the adjunction is $\MM$-linear.
	Therefore the two maps
	$$
	\eta_P\colon P \to Y \otimes_R X \otimes_{\unit_\MM} P
	,\quad
	\eta_{\unit_\MM} \otimes_{\unit_\MM} \id_P \colon P \to Y \otimes_R X \otimes_{\unit_\MM} P
	$$
	coincide.
	This shows that $c = \eta_{\unit_\MM}$ satisfies condition $(*)$ of the cited proposition.

	Similarly, for the other direction, if $Y$ is left dual to $X$ then the coevaluation map $c\colon \unit_\MM \to Y \otimes_R X$ gives an ($\MM$-linear) natural transformation $c \otimes_{\unit_\MM} -\colon \id_\MM \Rightarrow Y \otimes_R X \otimes_{\unit_\MM} -$, which is a unit of an adjunction by condition $(*)$ of the cited proposition.
\end{proof}

\begin{lem}
	Let $R \in \Alg(\MM)$, then $R \in \LMod_R$ is atomic.
	$R \otimes m$ is also atomic for any $m \in \MM^{\atomic}$.
\end{lem}

\begin{proof}
	Consider \cite[Corollary 4.2.3.7 (2)]{HA} where both $\CC$ and $\MM$ in the reference's notation are our $\MM$, $A = \unit_\CC$ and $B = R$.
	Then, the functor $\LMod_B \to \LMod_A$ is $\hom^\MM(R, -)\colon \LMod_R \to \MM$, which therefore commutes with all colimits, showing that $R$ is atomic.
	The second part follows from \cref{atomic-tensor}.
\end{proof}

\begin{prop}\label{Mod-R-mol}
	Let $R \in \Alg(\MM)$.
	Then $\LMod_R$ is molecular with $R$ as an atomic generator.
\end{prop}

\begin{proof}
	The previous lemma shows that $R$ is indeed atomic, and we need to show that it generates $\LMod_R$ under colimits and the action of $\MM$.
	Specifically, we will show that $\LMod_R$ is generated under colimits from $R \otimes m$ for $m \in \MM$.
	By \cite[Corollary 2.5]{MonTow}, this is equivalent to showing that $\hom(R \otimes m, -)\colon \LMod_R \to \Spaces$ are jointly conservative.
	Note that $R \otimes -\colon \MM \to \LMod_R$ is the left adjoint of $\underlying\colon \LMod_R \to \MM$, so that $\hom(R \otimes m, -) \cong \hom(m, \underlying)$.
	These functors are indeed jointly conservative since $\underlying\colon \LMod_R \to \MM$ is conservative, and $\hom(m, -)\colon \MM \to \Spaces$ over all $m \in \MM$ are jointly conservative.
\end{proof}

Let $X \in \LMod_R$, and consider the functor $\hom^\MM(X, -)\colon \LMod_R \to \MM$.
Note that $\hom^\MM(X, R)$ is equipped with a canonical right $R$-module structure which we denote by $X^\vee \in \RMod_R$.
In addition, there is a canonical map $X^\vee \otimes_R - \to \hom^\MM(X, -)$.

\begin{prop}\label{atomic-ldbl}
	$\LMod_R^\ldbl = \LMod_R^{\atomic}$, that is the left dualizable objects are atomic.
\end{prop}

\begin{proof}
	Recall that for $X \in \LMod_R$, the functor $X \otimes_{\unit_\MM} -\colon \MM \to \LMod_R$ is left adjoint to $\hom^\MM(X, -)$.

	If $X$ is left dualizable, then $X \otimes_{\unit_\MM} -$ is left adjoint to $Y \otimes_R -$ for some $Y \in \RMod_R$ by \cref{ldbl-adj}.
	By the uniqueness of adjoints we get that $\hom^\MM(X, -) \cong Y \otimes_R -$.
	Since $Y \otimes_R -$ commutes with colimits, we get that $X$ is atomic.

	Now assume that $X$ is atomic.
	The two functors $\hom^\MM(X, -)$ and $X^\vee \otimes_R -$ are colimit preserving, i.e.\ morphisms in $\Mod_\MM$, thus also commute with tensor from $\MM$.
	\cref{Mod-R-mol} shows that $\LMod_R$ is generated from $R$ by these operations, and by the construction of $X^\vee$, they agree on $R$, so the canonical map between the two is an equivalence.
	This shows that $X \otimes_{\unit_\MM} -$ is left adjoint to $\hom^\MM(X, -) \cong X^\vee \otimes_R -$, concluding by \cref{ldbl-adj}.
\end{proof}

\begin{remark}\label{mod-calg}
	If $R \in \CAlg(\MM)$, then left and right $R$-modules coincide, and the category of modules $\Mod_R$ is equipped with a symmetric monoidal structure for which dualizable modules coincide with left dualizable modules, thus also with atomic objects, that is $\Mod_R^\dbl = \Mod_R^{\atomic}$.
\end{remark}

Combining \cref{atomic-psh-sm-adj}, \cref{mod-func} and \cref{atomic-ldbl} we get the following main result.

\begin{cor}\label{atomic-ldbl-lax}
	There is a lax symmetric monoidal functor
	$
	\LMod_{(-)}^{\atomic}\colon \Alg(\MM) \to \CatK
	$,
	and $\LMod_R^{\atomic} = \LMod_R^\ldbl$.
\end{cor}

As a by product, we also obtain the following result.

\begin{lem}
	Let $F\colon \MM \to \NN$ be a map of modes, then it sends $\MM$-atomic objects to $\NN$-atomic objects.
\end{lem}

\begin{proof}
	By \cref{atomic-ldbl}, the $\MM$-atomic objects in $\MM$ are the dualizable objects.
	Since $F$ is symmetric monoidal it sends dualizable objects to dualizable objects.
	Thus the $\MM$-atomic objects are sent to dualizable objects in $\NN$.
	Again by \cref{atomic-ldbl}, the dualizable objects in $\NN$ are $\NN$-atomic.
\end{proof}

	\section{Day Convolution}\label{sec-day}

The Day Convolution on functor categories was developed in \cite{Saul, HA}.
In this section we prove results about the Day convolution, specifically its functoriality in the source and target.
The results of this section are used in \cref{cmon-mul} to show that the mode symmetric monoidal structure on higher commutative monoids coincides with the localization of the Day convolution.
This is subsequently used in \cref{k-lax} to endow higher semiadditive algebraic K-theory with a lax symmetric monoidal structure.

We begin by recalling the universal property of the Day convolution:

\begin{thm}[{\cite[Remark 2.2.6.8]{HA}}]\label{day-univ-prop}
	Let $\II, \CC$ be symmetric monoidal categories, and assume that $\CC$ has all colimits and that its tensor product preserves colimits in each coordinate separately.
	Then, there is a symmetric monoidal structure on $\Fun(\II, \CC)$, called the \tdef{Day convolution} denoted by $\mdef{\circledast}$, satisfying the following universal property:
	There is an equivalence of functors $\CMon(\Cat) \to \Cat$
	$$
	\Funlax(- \times \II, \CC)
	\iso \Funlax(-, \Fun(\II, \CC))
	,
	$$
	which lifts the equivalence of functors $\Cat \to \Cat$
	$$
	\Fun(- \times \II, \CC)
	\iso \Fun(-, \Fun(\II, \CC))
	.
	$$
\end{thm}

\begin{example}
	Let $\II$ be a symmetric monoidal category.
	Then $\II^\op$ is also endowed with a symmetric monoidal structure, and $\Spaces$ can be endowed with the cartesian structure, yielding the Day convolution on $\PSh(\II) = \Fun(\II^\op, \Spaces)$.
	By \cite[Remark 4.8.1.13]{HA}, this agrees with the symmetric monoidal structure on $\PSh(\II)$ of \cite[Remark 4.8.1.8]{HA} used in the proof of \cref{psh-sm}.
\end{example}

\begin{prop}\label{day-range}
	Let $\II, \CC$ and $\DD$ be symmetric monoidal categories, and assume that $\CC$ and $\DD$ have all colimits and that their tensor product preserve colimits in each coordinate.
	Let $F\colon \CC \to \DD$ be a functor and let $\tilde{F}\colon \Fun(\II, \CC) \to \Fun(\II, \DD)$ be the functor induced by post-composition.
	If $F$ is lax symmetric monoidal, then so is $\tilde{F}$.
	If $F$ is colimit preserving, then so is $\tilde{F}$.
	If $F$ is both colimit preserving and symmetric monoidal, then so is $\tilde{F}$.
\end{prop}

\begin{proof}
	We begin with the first part.
	The identity functor of $\Fun(\II, \CC)$ is (lax) symmetric monoidal, therefore by the universal property of the Day convolution, the corresponding functor $\Fun(\II, \CC) \times \II \to \CC$ is also lax symmetric monoidal.
	Post-composition of this functor with the lax symmetric monoidal functor $F$ gives a lax symmetric monoidal functor $\Fun(\II, \CC) \times \II \to \DD$.
	Using the universal property again, we get that $\tilde{F}\colon \Fun(\II, \CC) \to \Fun(\II, \DD)$ is also lax symmetric monoidal.
	
	For the second part, if $F$ is colimit preserving, then since colimits in functor categories are computed level-wise, $\tilde{F}$ is colimit preserving.

	Lastly, we assume that $F$ is both colimit preserving and symmetric monoidal.
	We already know from the second part that $\tilde{F}$ is colimit preserving.
	We show that the lax symmetric monoidal structure from the first part is in fact symmetric monoidal.
	Recall that by \cite[Example 2.2.6.17]{HA}, the Day convolution of $X, Y\in \Fun(\II, \CC)$ is given on objects by
	$$
	(X \circledast Y)(i)
	\cong \colim_{i_1 \otimes i_2 \to i} X(i_1) \otimes Y(i_2)
	.
	$$
	The lax symmetric monoidal structure of $\tilde{F}$ is then given by the canonical map:
	\begin{align*}
		(\tilde{F}X \circledast \tilde{F}Y)(i)
		&\cong \colim_{i_1 \otimes i_2 \to i} FX(i_1) \otimes FY(i_2)\\
		&\xrightarrow{\text(1)} \colim_{i_1 \otimes i_2 \to i} F(X(i_1) \otimes Y(i_2))\\
		&\xrightarrow{\text(2)} F(\colim_{i_1 \otimes i_2 \to i} X(i_1) \otimes Y(i_2))\\
		&\cong (\tilde{F}(X \circledast Y))(i)
	\end{align*}
	where map (1) uses the fact $F$ is lax symmetric monoidal, and (2) is the assembly map.
	Since $F$ is symmetric monoidal (1) is an equivalence, and since $F$ is colimit preserving (2) is an equivalence, showing that $\tilde{F}$ is in fact symmetric monoidal.
\end{proof}

Our next goal is to study the behavior of the Day convolution under the change of the source $\II$, namely given a symmetric monoidal functor $p\colon \II \to \JJ$, what can we say about $p_!\colon \Fun(\II, \CC) \to \Fun(\JJ, \CC)$ and $p^*\colon \Fun(\JJ, \CC) \to \Fun(\II, \CC)$.

We wish to thank Lior Yanovski for suggesting the following argument to prove \cref{p-ls-prl}.

\begin{lem}\label{p-ls-spaces}
	Let $\CC = \Spaces$ equipped with the carestian structure, then $p_!\colon \Fun(\II, \Spaces) \to \Fun(\JJ, \Spaces)$ is symmetric monoidal.
\end{lem}

\begin{proof}
	This is \cite[Remark 4.8.1.8]{HA} applied to $p^\op\colon \II^\op \to \JJ^\op$.
\end{proof}

\begin{lem}\label{p-ls-fun}
	Let $\sK$ be a symmetric monoidal category, and let $\CC = \Fun(\sK, \Spaces)$ equipped with the Day convolution, then $p_!\colon \Fun(\II, \CC) \to \Fun(\JJ, \CC)$ is symmetric monoidal.
\end{lem}

\begin{proof}
	By applying the universal property of the Day convolution twice, we know that there is an equivalence of functors:
	$$
	\Funlax(- \times \II \times \sK, \Spaces)
	\iso \Funlax(- \times \II, \Fun(\sK, \Spaces))
	\iso \Funlax(-, \Fun(\II, \Fun(\sK, \Spaces)))
	$$
	Using the same reasoning with the roles of $\II$ and $\sK$ reversed, and the equivalence $\II \times \sK \cong \sK \times \II$, the universal property of the Day convolution implies that there is a symmetric monoidal equivalence:
	$$
	\Fun(\sK, \Fun(\II, \Spaces))
	\iso \Fun(\II, \Fun(\sK, \Spaces))
	$$
	Similarly, we have an equivalence with $\JJ$ in place of $\II$.
	
	\cref{p-ls-spaces} constructs a symmetric monoidal functor $p_!\colon \Fun(\II, \Spaces) \to \Fun(\JJ, \Spaces)$.
	Since this functor is also colimit preserving, by \cref{day-range}, post-composition with it gives a symmetric monoidal functor
	$$
	\Fun(\sK, \Fun(\II, \Spaces))
	\to \Fun(\sK, \Fun(\JJ, \Spaces))
	,
	$$
	which under the equivalences above gives the desired map $p_!\colon \Fun(\II, \CC) \to \Fun(\JJ, \CC)$.
\end{proof}

\begin{prop}\label{p-ls-prl}
	Let $\CC \in \CAlg(\PrL)$, then $p_!\colon \Fun(\II, \CC) \to \Fun(\JJ, \CC)$ is symmetric monoidal.
\end{prop}

\begin{proof}
	\cite[Proposition 2.2]{LocSym} shows that there is a symmetric monoidal reflective localization $\tilde{\CC} \to \CC$, for some $\tilde{\CC} = \Fun(\sK, \Spaces)$.
	Then $\Fun(\II, \tilde{\CC}) \to \Fun(\II, \CC)$ given by post-composition is also a reflective localization, and by \cref{day-range} it is also symmetric monoidal, and the same holds with $\JJ$ in place of $\II$.
	
	By \cref{p-ls-fun}, we have a symmetric monoidal functor $p_!\colon \Fun(\II, \tilde{\CC}) \to \Fun(\JJ, \tilde{\CC})$.
	Composing this with the symmetric monoidal localization $\Fun(\JJ, \tilde{\CC}) \to \Fun(\JJ, \CC)$, we get a symmetric monoidal functor
	$\Fun(\II, \tilde{\CC}) \to \Fun(\JJ, \CC)$.
	This functor factors through the symmetric monoidal localization $\Fun(\II, \tilde{\CC}) \to \Fun(\II, \CC)$, yielding the symmetric monoidal structure on $p_!\colon \Fun(\II, \CC) \to \Fun(\JJ, \CC)$.
\end{proof}

Applying this to the special case where $p$ is the unit map $* \to \II$ we get

\begin{cor}\label{day-from-c}
	There is a map $F\colon \CC \to \Fun(\II, \CC)$ in $\CAlg(\PrL)$.
\end{cor}

Furthermore, applying \cite[Corollary 7.3.2.7]{HA} to the adjunction $p^* \dashv p_!$ we get

\begin{cor}
	Let $\CC \in \CAlg(\PrL)$, then $p^*\colon \Fun(\JJ, \CC) \to \Fun(\II, \CC)$ is lax symmetric monoidal.
\end{cor}

\begin{remark}
	One can directly use the universal property of the Day convolution to show that $p^*$ is lax symmetric monoidal, even only assuming that $p$ is lax symmetric monoidal.
	In fact, one can use the main result of \cite{lax-adj} to construct an oplax symmetric monoidal structure on $p_!$ in this way while only assuming that $p$ is lax symmetric monoidal, and prove that it is symmetric monoidal in case $p$ is.
	However, we have not shown that the lax symmetric monoidal structure on $p^*$ obtained in the above corollary coincides with the one obtained directly from the universal property of the Day convolution.
\end{remark}

\begin{prop}\label{day-prl-tensor}
	Let $\II \in \CAlg(\Cat) \cong \CMon(\Cat)$ and $\CC \in \CAlg(\PrL)$.
	Then the equivalence $\Fun(\II, \Spaces) \otimes \CC \iso \Fun(\II, \CC)$ is symmetric monoidal.
\end{prop}

\begin{proof}
	Since the tensor product is the coproduct in $\CAlg(\PrL)$, to upgrade the equivalence into a symmetric monoidal functor, it suffices to upgrade the functors $\Fun(\II, \Spaces) \to \Fun(\II, \CC)$ and $\CC \to \Fun(\II, \CC)$ to symmetric monoidal functors.
	Indeed, by \cref{day-range} post-composition with the unit map $\Spaces \to \CC$ yields a symmetric monoidal functor $\Fun(\II, \Spaces) \to \Fun(\II, \CC)$,
	and \cref{day-from-c} shows that $\CC \to \Fun(\II, \CC)$ is symmetric monoidal.
\end{proof}

	\section{Higher Commutative Monoids}\label{sec-cmon}

In this section we recall the notion of ($p$-typical) $m$-commutative monoids as developed in \cite{Harpaz} and \cite{AmbiHeight} (see \cref{def-cmon}), and their relationship to higher semiadditivity (see \cref{cmon-mode}), which feature prominently in the definition of higher semiadditive algebraic K-theory in \cref{sa-k-def}.
A key result of this section is \cref{cmon-mul}, which shows that for $\CC \in \CAlg(\PrL)$, the symmetric monoidal structures on $\CMonm(\CC)$ coming from the mode structure on $\CMonm(\Spaces)$ and from the Day convolution coincide.
This result is used in \cref{k-lax} to endow higher semiadditive algebraic K-theory with a lax symmetric monoidal structure.

\subsection{Definition and Properties}

\begin{defn}
	Let $\CC \in \PrL$ be a presentable category.
	We define the category of \tdef{($p$-typical) pre-$m$-commutative monoids} in $\CC$ by
	$\mdef{\PCMonm[\CC]} = \Fun(\Span(\Spacesm)^\op,\CC)$.
	We define the \tdef{underlying object} functor $\mdef{\underlying}\colon \PCMonm[\CC]\to\CC$ by pre-composition with $* \to \Span(\Spacesm)^\op$, i.e.\ $\underline{X} = X(*)$.
\end{defn}

\begin{defn}\label{def-cmon}
	We say that a pre-$m$-commutative monoid $X \in \PCMonm[\CC]$ is a \tdef{($p$-typical) $m$-commutative monoid} if it satisfies the $m$-Segal condition, i.e.\ the assembly map $X(A) \to \lim_{A}\underline{X}$ is an equivalence for any $m$-finite $p$-space $A$.
	Their category is the full subcategory
	$\mdef{\CMonm[\CC]} = \Fun^{\seg}(\Span(\Spacesm)^\op,\CC) \subseteq \PCMonm[\CC]$.
\end{defn}

\begin{remark}\label{seg-lim}
	The $m$-Segal condition for $X$ is equivalent to $X$ preserving limits indexed by $A \in \Spacesm$.
\end{remark}

\begin{prop}\label{underlying-conservative}
	The restriction of the underlying functor $\underlying\colon \CMonm[\CC]\to\CC$ is conservative.
\end{prop}

\begin{proof}
	Follows immediately from the $m$-Segal condition.
\end{proof}

\begin{lem}\label{filt-colims-small-lims}
	Let $\CC \in \PrL$ be a $\kappa$-presentable category.
	Then, $\mu$-filtered colimits commute with $\mu$-small limits in $\CC$, for any $\mu \geq \kappa$.
\end{lem}

\begin{proof}
	First, the case $\CC = \Spaces$ is \cite[Proposition 5.3.3.3]{HTT}.
	Second, the case $\CC = \PSh(\CC_0)$ follows from the previous case, since limits and colimits are computed level-wise in functor categories.
	Lastly, for the general case we have that $\CC \cong \mrm{Ind}_\kappa(\CC^\kappa)$.
	By \cite[Proposition 5.3.5.3]{HTT}, $\CC \subseteq \PSh(\CC^\kappa)$ is closed under $\kappa$-filtered colimits.
	Additionally, \cite[Corollary 5.3.5.4 (3)]{HTT} shows that it is also closed under limits, since limits commute with limits.
	To conclude, $\CC$ is closed under $\mu$-filtered colimits and $\mu$-small limits in $\PSh(\CC^\kappa)$ for any $\mu \geq \kappa$, and by the second case the result holds for $\PSh(\CC^\kappa)$ for any $\mu$.
\end{proof}

\begin{prop}
	The inclusion $\CMonm[\CC] \subseteq \PCMonm[\CC]$ preserves all limits and $\CMonm[\CC]$ is presentable.
\end{prop}

\begin{proof}
	We essentially repeat the proof in \cite[Lemma 5.17]{Harpaz}.
	Recall that $\CMonm[\CC]$ is the full subcategory of $\PCMonm[\CC]$ on those functors that preserve limits index by $A \in \Spacesm$.
	As limits commute with limits, and limits are computed level-wise in $\PCMonm[\CC]$, we get that $\CMonm[\CC]$ is closed under limits.
	Let $\kappa$ be any cardinal such that $\CC$ is $\kappa$-presentable and all $A \in \Spacesm$ are $\kappa$-small, then by \cref{filt-colims-small-lims}, limits indexed by $A \in \Spacesm$ commute with $\kappa$-filtered colimits in $\CC$.
	Again, as colimits are computed level-wise in $\PCMonm[\CC]$, we get that $\CMonm[\CC]$ is closed under $\kappa$-filtered colimits.
	It follows that $\CMonm[\CC]$ is presentable by the reflection principle of \cite{RP}.
	Alternatively, the presentability follows from \cite[Lemma 2.11 (iv)]{Seg}.
\end{proof}

\begin{defn}
	Let $\CC \in \PrL$.
	We denote by $\mdef{\Lseg}\colon \PCMonm[\CC] \to \CMonm[\CC]$ the left adjoint of the inclusion $\CMonm[\CC] \subseteq \PCMonm[\CC]$.
\end{defn}

Given a functor $F\colon \CC \to \DD$ we get an induced functor $F\colon \PCMonm[\CC] \to \PCMonm[\DD]$ by post-composing level-wise with $F$.

\begin{prop}\label{lim-cmon}
	If $F\colon \CC \to \DD$ commutes with limits indexed by any $m$-finite $p$-space $A$,
	then the restriction of $F\colon \PCMonm[\CC] \to \PCMonm[\DD]$ to $\CMonm[\CC]$ lands in $\CMonm[\DD]$, yielding a functor
	$F\colon \CMonm[\CC] \to \CMonm[\DD]$.
\end{prop}

\begin{proof}
	Follows immediately from the characterization given in \cref{seg-lim} and the fact that $F$ commutes with these limits.
\end{proof}

\subsection{Higher Commutative Monoids and Semiadditivity}

The underlying object functor $\underline{(-)}\colon \CMonm[\Spaces] \to \Spaces$ has a left adjoint $\Fseg \colon \Spaces \to \CMonm[\Spaces]$.
The fact that this endows $\CMonm[\Spaces]$ with the structure of a mode was first proved in \cite[Corollary 5.19]{Harpaz} (for $m < \infty$), and subsequently developed in \cite[Proposition 5.3.1]{AmbiHeight}.

\begin{thm}[{\cite[Proposition 5.3.1]{AmbiHeight}}]\label{cmon-mode}
	$\CMonm[\Spaces]$ is a \emph{mode}, that is, it is an idempotent presentable symmetric monoidal category.
	As such, it classifies the property of being ($p$-typically) $m$-semiadditive.
	Furthermore, for any $\CC \in \PrL$ there is an equivalence $\CMonm[\CC] \cong \CMonm[\Spaces] \otimes \CC$,
	and tensoring the unit map $\Spaces \to \CMonm[\Spaces]$ with $\CC$ yields a left adjoint functor $\Fseg\colon \CC \to \CMonm[\CC]$.
\end{thm}

\begin{remark}
	The cited papers prove the result in the non $p$-typical case, i.e.\ for all $m$-finite spaces, but the same proofs work for the $p$-typical case.
\end{remark}

Following \cite[Definition 5.3.3]{AmbiHeight}, we make the following definition.

\begin{defn}\label{tsadim-def}
	We define $\mdef{\tsadim} = \CMonm[\Sp_{(p)}] \cong \CMonm[\Spaces] \otimes \Sp_{(p)}$.
	This is the mode which classifies the property of being a \emph{$p$-local stable and ($p$-typically) $m$-semiadditive} presentable category.
	There is a canonical map of modes, which we call the \tdef{group-completion} $\mdef{(-)^{\gpc}}\colon \CMonm[\Spaces] \to \tsadim$.
\end{defn}

\begin{example}
	The case $m = 0$ reproduces $\tsadi^{[0]} \cong \Sp_{(p)}$ and the group-completion is the map
	$$
	(-)^{\gpc}\colon \CMon[\Spaces] \to \CMongl[\Spaces] \cong \conSp \hookrightarrow \Sp_{(p)}.
	$$
\end{example}

We recall the following:

\begin{prop}\label{lim-cat-lam}
	The inclusions $\Catmfin \to \Cat$, $\Catst \to \Cat$ and $\Catmfinst \to \Catmfin$ create all limits.
\end{prop}

\begin{proof}
	The first part is \cite[Corollary 5.3.6.10]{HTT} for $\clK = \emptyset$ and $\clK' = \Spacesm$.
	The second part is \cite[Theorem 1.1.4.4]{HA}.
	The third part follows from the combination of the first two.
\end{proof}

Combining \cref{lim-cmon} and \cref{lim-cat-lam}, we get:

\begin{cor}\label{cmon-cat-lam}
	The inclusions $\Catmfinst \to \Catmfin$ and $\Catmfin \to \Cat$
	induce inclusions $\CMonm[\Catmfinst] \to \CMonm[\Catmfin]$ and $\CMonm[\Catmfin] \to \CMonm[\Cat]$.
\end{cor}

\begin{prop}\label{tsadi-atomics}
	Let $\MM = \CMonm[\Sp]$, then for any $\CC \in \Mod_{\MM}$, the category $\CC^{\atomic}$ is stable and $m$-semiadditive.
	Furthermore, the functor taking the atomics gives a lax symmetric monoidal functor $(-)^{\atomic} \colon \Mod_{\MM}^{\intL} \to \Catmfinst$.
	The same holds for the mode $\tsadim \cong \MM \otimes \Sp_{(p)}$, and the corresponding functor $(-)^{\atomic} \colon \Mod_{\tsadim}^{\intL} \to \Catmfinst$ is the restriction of the previous functor.
\end{prop}

\begin{proof}
	We first show that for any $\CC \in \Mod_{\MM}$, $\CC^{\atomic}$ is stable and $m$-semiadditive.
	Recall from \cite[Proposition 4.3.9]{HL} that for any $\CC$ in the mode $\CMonm[\Spaces]$ and $m$-finite $p$-space $A$, $A$-shaped limits in $\CC$ commute with all colimits, showing that all $m$-finite $p$-space $A$ are absolute limits in $\CMonm[\Spaces]$, and by \cref{abs-mod} also in $\MM$.
	Since $\CC^{\atomic}$ is a full subcategory of the $m$-semiadditive category $\CC$, and is closed under all colimits indexed by any $m$-finite $p$-space $A$, it is in fact $m$-semiadditive by \cite[Proposition 2.1.4 (4)]{AmbiHeight}.
	The stability statement follows from \cref{absolute-stable}.

	In particular, this shows that $\clK$, the collection of all finite categories and $m$-finite $p$-spaces, is a collection of absolute limits of $\MM$.
	\cref{atomic-psh-sm-adj} then shows that there is a lax symmetric monoidal functor $(-)^{\atomic} \colon \Mod_{\MM}^{\intL} \to \Cat_\clK$.
	Recall that there is a fully faithful functor $(-)^\otimes\colon \CMon(\Cat)^\lax \to \Op$ from the category of symmetric monoidal categories and lax symmetric monoidal functors to operads.
	Note that $\Catmfinst \subset \Cat_\clK$ is the full subcategory on those categories which are in addition stable, but it is not a sub-symmetric monoidal category, since the unit of $\Cat_\clK$ is not stable.
	However, it is true that the tensor product of a family of categories in either category is the same, in particular $\Cat_{\mfin}^{\st,\otimes}$ is a sub-operad of $\Cat_\clK^\otimes$.
	Therefore, we get that the map of operads $(-)^{\atomic}\colon \Mod_{\MM}^{\intL,\otimes} \to \Cat_\clK^\otimes$ factors through the operad $\Cat_{\mfin}^{\st,\otimes}$, which corresponds to the desired lax symmetric monoidal functor $(-)^{\atomic} \colon \Mod_{\MM}^{\intL} \to \Catmfinst$.
	
	Note that $\tsadim \cong \MM \otimes \Sp_{(p)}$ is a smashing localization of $\MM$.
	The argument above works for $\tsadim$ in place of $\MM$.
	Furthermore, by \cref{atomic-smashing}, the atomic objects with respect to either one are the same, showing that that $(-)^{\atomic} \colon \Mod_{\tsadim}^{\intL} \to \Catmfinst$ is indeed the restriction of $(-)^{\atomic} \colon \Mod_{\MM}^{\intL} \to \Catmfinst$.
	We also note that the lax symmetric monoidal structure on the latter restricts to the lax symmetric monoidal structure on the former.
	To see that, observe that the symmetric monoidal left adjoint of the latter $\PShK^{\tsadim}\colon \Catmfinst \to \Mod_{\tsadim}^{\intL}$ factors as the composition of the symmetric monoidal functors $\Catmfinst \xrightarrow{\PShKM} \Mod_{\MM}^{\intL} \xrightarrow{- \otimes \tsadim} \Mod_{\tsadim}^{\intL}$, so the lax symmetric monoidal right adjoint factors accordingly.
\end{proof}

\subsection{Tensor Product of Higher Commutative Monoids}

Let $\CC \in \CAlg(\PrL)$ be a presentably symmetric monoidal category.
In this subsection, we endow $\CMonm[\CC]$ with two symmetric monoidal structures, and show that they coincide.
The first, which we call the \emph{mode symmetric monoidal structure} (see \cref{mode-sm}), comes from the fact that $\CMonm[\Spaces]$ is a mode.
The second, which we call the \emph{localized Day convolution} (see \cref{localized-day-sm}), is obtained by localizing the Day convolution on $\PCMonm[\CC]$.
Finally, in \cref{cmon-mul} we show that the two structures coincide.

Recall that by \cref{cmon-mode}, $\CMonm[\Spaces]$ is a mode, and in particular it is equipped with a symmetric monoidal structure.

\begin{defn}\label{mode-sm}
	Let $\CC \in \CAlg(\PrL)$ be a presentably symmetric monoidal category.
	The equivalence $\CMonm[\CC] \cong \CMonm[\Spaces] \otimes \CC$ of \cref{cmon-mode} endows $\CMonm[\CC]$ with a presentably symmetric monoidal structure which we call the \tdef{mode symmetric monoidal structure} and denote by $\mdef{\otimes}$.
	Furthermore, by construction, $\Fseg\colon \CC \to \CMonm[\CC]$ is endowed with a symmetric monoidal structure.
\end{defn}

In a different direction, consider the category $\Span(\Spacesm)$.
Since $\Spacesm$ is closed under products, it has a cartesian monoidal structure.
By \cite[Theorem 1.2 (iv)]{Spans}, its span category $\Span(\Spacesm)$ is endowed with a symmetric monoidal structure given on objects by their cartesian product in $\Spacesm$.
Therefore, the opposite category $\Span(\Spacesm)^\op$ is also endowed with a symmetric monoidal structure.

\begin{remark}
	The symmetric monoidal structure on $\Span(\Spacesm)$ that we use is \emph{not} the cartesian or cocartesian structure.
	In fact, the cartesian and cocartesian structures coincide (since products and coproducts coincide in $\Span(\Spacesm)$, being a semiadditive category), and are given on objects by the disjoint union of spaces, whereas the symmetric monoidal structure we use is given on objects by the product of spaces.
\end{remark}

\begin{defn}
	Let $\CC \in \CAlg(\PrL)$ be a presentably symmetric monoidal category.
	We endow $\PCMonm[\CC] = \Fun(\Span(\Spacesm)^\op,\CC)$ with the \tdef{Day convolution} of \cref{day-univ-prop}, which we denote by $\mdef{\circledast}$.
	By \cref{day-from-c}, we have a symmetric monoidal functor $F\colon \CC \to \PCMonm[\CC]$.
\end{defn}

Our next goal, achieved in \cref{sh-compat}, is to show that the Day convolution is compatible with the reflective localization $\Lseg\colon \PCMonm[\CC] \to \CMonm[\CC]$, endowing it with a localized symmetric monoidal structure.
Recall from \cite[Example 2.2.1.7]{HA} that for a symmetric monoidal category $\DD$, we say that a reflective localization $L\colon \DD \to \DD_0$ is compatible with the symmetric monoidal structure, if and only if for any $L$-equivalence $X \to Y \in \DD$ and $Z \in \DD$, the morphism $X \otimes Z \to Y \otimes Z \in \DD$ is an $L$-equivalence.

\begin{remark}
	By the Yoneda lemma, a map $X \to Y$ is an $L$-equivalence if and only if for any $T \in \DD_0$ the induced map
	$\hom(Y, T) \to \hom(X, T)$
	is an equivalence.
\end{remark}

\begin{lem}\label{compat-int-hom}
	Let $\DD$ and $L$ be as above, and assume further that the symmetric monoidal structure on $\DD$ is closed.
	Then, the reflective localization is compatible with the symmetric monoidal structure if and only if for any $L$-equivalence $X \to Y \in \DD$ and $T \in \DD_0$ the induced map on internal homs
	$\hom^\DD(Y, T) \to \hom^\DD(X, T)$
	is an equivalence.
\end{lem}

\begin{proof}
	Fix a map $X \to Y \in \DD$.
	By the Yoneda lemma,
	$\hom^\DD(Y, T) \to \hom^\DD(X, T)$ is an equivalence for any $T \in \DD_0$,
	if and only if the map
	$\hom(Z, \hom^\DD(Y, T)) \to \hom(Z, \hom^\DD(X, T))$
	is an equivalence for any $Z \in \DD, T \in \DD_0$.
	By adjunction, the latter holds if and only if
	$\hom(Y \otimes Z, T) \to \hom(X \otimes Z, T)$
	is an equivalence for any $Z \in \DD, T \in \DD_0$.
	By the Yoneda lemma, this holds if and only if $X \otimes Z \to Y \otimes Z$ is an $L$-equivalence for any $Z \in \DD$.
\end{proof}

\begin{lem}\label{tensor-compat}
	Let $\DD, \DD_0, \EE \in \CAlg(\PrL)$ and let $L\colon \DD \to \DD_0$ be a reflective localization which is compatible with symmetric monoidal structure.
	Then $L \otimes \id\colon \DD \otimes \EE \to \DD_0 \otimes \EE$ is also a reflective localization compatible with the symmetric monoidal structure on $\DD \otimes \EE$.
\end{lem}

\begin{proof}
	We let $L' = L \otimes \id$.
	First note that $L'$ is indeed a reflective localization by \cite[Lemma 5.2.1]{AmbiHeight}.
	Using \cite[Proposition 2.2.1.9]{HA} we endow $\DD_0$ with the localized symmetric monoidal structure, making $L$ into a symmetric monoidal functor.
	Since $\otimes$ is the coproduct of $\CAlg(\PrL)$, this makes the categories and the map $L'\colon \DD \otimes \EE \to \DD_0 \otimes \EE$ symmetric monoidal.
	Now let $X \to Y \in \DD \otimes \EE$ be an $L'$-equivalence.
	For any $Z \in \DD \otimes \EE$, we have
	$$
	L'(X \otimes Z \to Y \otimes Z)
	\cong
	(L' X \iso L' Y) \otimes (L' Z \xrightarrow{\id} L' Z)
	$$
	which is an equivalence.
\end{proof}

Note that by the Yoneda lemma, for any $A \in \Spacesm$, the object $\yon(A) \in \PCMonm[\Spaces]$ corepresents the evaluation at $A$ functor $\PCMonm[\Spaces] \to \Spaces$
given by $X \mapsto  X(A)$.
We also note that these functors over all $A \in \Spacesm$ are jointly conservative.

\begin{lem}\label{UA-cmon}
	Let $X \in \PCMonm[\Spaces]$, and $A \in \Spacesm$, then 
	$$
	\hom^{\PCMonm[\Spaces]}(\yon(A), X)
	\cong X(A \times -)
	.
	$$
	In particular, if $X \in \CMon[\Spaces]$, then so is $\hom^{\PCMonm[\Spaces]}(\yon(A), X)$.
\end{lem}

\begin{proof}
	By \cite[Corollary 4.8.1.12]{HA}, the Yoneda embedding $\yon\colon \Span(\Spacesm) \to \PCMonm[\Spaces]$ is symmetric monoidal, so that $\yon(A) \otimes \yon(-) \cong \yon(A \times -)$.
	Therefore, we get
	\begin{align*}
		\hom^{\PCMonm[\Spaces]}(\yon(A), X)(-)
		&\cong \hom(\yon(-), \hom^{\PCMonm[\Spaces]}(\yon(A), X))\\
		&\cong \hom(\yon(A) \otimes \yon(-), X)\\
		&\cong \hom(\yon(A \times -), X)\\
		&\cong X(A \times -)
		.
	\end{align*}
	For the second part, note that $X(A \times -)$ is the pre-composition of $X$ with $A \times -\colon \Span(\Spacesm)^\op \to \Span(\Spacesm)^\op$ which preserves limits indexed by $B \in \Spacesm$.
	By \cref{seg-lim}, the $m$-Segal condition is equivalent to preservation of limits indexed by $B \in \Spacesm$, so the result follows.
\end{proof}

\begin{lem}\label{sh-S-compat}
	The reflective localization $\Lseg\colon \PCMonm[\Spaces] \to \CMonm[\Spaces]$ is compatible with the Day convolution.
\end{lem}

\begin{proof}
	By \cref{compat-int-hom}, it suffices to show that for any $\CMonm[\Spaces]$-equivalence $X \to Y$ and $T \in \CMonm[\Spaces]$ the induced map
	$\hom^{\PCMonm[\Spaces]}(Y, T) \to \hom^{\PCMonm[\Spaces]}(X, T)$
	is an equivalence.
	Since the evaluations at $A \in \Spacesm$ are jointly conservative, it suffices to show that for any $A \in \Spacesm$ the map
	$\hom(\yon(A), \hom^{\PCMonm[\Spaces]}(Y, T)) \to \hom(\yon(A), \hom^{\PCMonm[\Spaces]}(X, T))$
	is an equivalence.
	By adjunction, this is equivalent to showing that
	$\hom(Y, \hom^{\PCMonm[\Spaces]}(\yon(A), T)) \to \hom(X, \hom^{\PCMonm[\Spaces]}(\yon(A), T))$
	is an equivalence.
	By assumption $X \to Y$ is an $\CMonm[\Spaces]$-equivalence and $T \in \CMonm[\Spaces]$, so the result follow from the second part of \cref{UA-cmon}.
\end{proof}

\begin{prop}\label{sh-compat}
	The reflective localization $\Lseg\colon \PCMonm[\CC] \to \CMonm[\CC]$ is compatible with the Day convolution.
\end{prop}

\begin{proof}
	Consider the following commutative diagram in $\PrL$:
	$$\begin{tikzcd}
		\PCMonm[\Spaces] \otimes \CC \arrow{r}{\sim} \arrow{d}{} & \PCMonm[\CC] \arrow{d}{} \\
		\CMonm[\Spaces] \otimes \CC \arrow{r}{\sim} & \CMonm[\CC]
	\end{tikzcd}$$
	The bottom map is an equivalence by \cref{cmon-mode}.
	The top map is a symmetric monoidal equivalence by \cref{day-prl-tensor}.
	By \cref{sh-S-compat}, $\Lseg\colon \PCMonm[\Spaces] \to \CMonm[\Spaces]$ is compatible with the Day convolution, so by \cref{tensor-compat} the left map is also compatible with the symmetric monoidal structure.
	Therefore, the right map is also compatible with the Day convolution.
\end{proof}

\begin{defn}\label{localized-day-sm}
	Using \cite[Proposition 2.2.1.9]{HA}, we endow $\CMonm[\CC] \subseteq \PCMonm[\CC]$ with the induced symmetric monoidal structure, which we call the \tdef{localized Day convolution} and denote by $\mdef{\hat{\circledast}}$.
	This makes the functor $\Lseg\colon \PCMonm[\CC] \to \CMonm[\CC]$ symmetric monoidal (with respect to the localized Day convolution), thus by \cite[Corollary 7.3.2.7]{HA} the right adjoint $\CMonm[\CC] \subseteq \PCMonm[\CC]$ is lax symmetric monoidal.
\end{defn}

The main result of this subsection is the following:

\begin{thm}\label{cmon-mul}
	Let $\CC \in \CAlg(\PrL)$, then the mode symmetric monoidal structure and the localized Day convolution on $\CMonm[\CC]$ coincide, making the following diagram in $\CAlg(\PrL)$ commute:
	$$\begin{tikzcd}
	 & \PCMonm[\CC] \arrow{d}{\Lseg} \\
	\CC \arrow{r}[swap]{\Fseg} \arrow{ru}{F} & \CMonm[\CC]
	\end{tikzcd}$$
\end{thm}

We begin by proving the result for $\CC = \Spaces$.

\begin{lem}
	The localized Day convolution and the mode symmetric monoidal structure on $\CMonm[\Spaces]$ coincide, and $\Lseg F \cong \Fseg$.
\end{lem}

\begin{proof}
	The right adjoint of $\Fseg$ is the underlying object functor $\underline{(-)}\colon \CMonm[\Spaces] \to \Spaces$, which by adjunction is represented by $\Fseg(*) \cong \unit^{\otimes}$.
	By \cite[Corollary 4.8.1.12]{HA}, the Yoneda embedding $\yon\colon \Span(\Spacesm) \to \PCMonm[\Spaces]$ is symmetric monoidal, and in particular the unit of $\PCMonm[\Spaces]$ is $\unit^{\circledast} \cong \yon(*)$.
	Using the Yoneda lemma we get
	\begin{align*}
		\hom_{\CMonm[\Spaces]}(\unit^{\hat{\circledast}},X)
		&\cong \hom_{\CMonm[\Spaces]}(\Lseg \unit^{\circledast},X)\\
		&\cong \hom_{\PCMonm[\Spaces]}(\unit^{\circledast},X)\\
		&\cong \hom_{\PCMonm[\Spaces]}(\yon(*),X)\\
		&\cong \underline{X}
		.
	\end{align*}
	Therefore, $\unit^{\hat{\circledast}}$ also represents $\underline{(-)}$, so that $\unit^{\hat{\circledast}} \cong \unit^{\otimes}$.
	Since $\CMonm[\Spaces]$ is a mode, it has a unique presentably symmetric monoidal structure with the given unit as in \cite[Proposition 5.1.6]{AmbiHeight}, so that localized Day convolution and the mode symmetric monoidal structure on $\CMonm[\Spaces]$ coincide.
	Since there is a unique map of modes $\Spaces \to \CMonm[\Spaces]$, the functors $\Lseg F$ and $\Fseg$ coincide.
\end{proof}

\begin{proof}[Proof of \cref{cmon-mul}]
	Consider the following diagram in $\CAlg(\PrL)$ where we endow $\CMonm[\CC]$ with the localized Day convolution structure (and the rest of the categories are endowed with a single symmetric monoidal structure, as we have shown that the two structures on $\CMonm[\Spaces]$ coincide.)
	$$\begin{tikzcd}
		\PCMonm[\Spaces] \otimes \CC \arrow{r}{\sim} \arrow{d}{} & \PCMonm[\CC] \arrow{d}{} \\
		\CMonm[\Spaces] \otimes \CC \arrow{r}{\sim} & \CMonm[\CC]
	\end{tikzcd}$$
	The bottom map is an equivalence by \cref{cmon-mode}, and we wish to upgrade it to a symmetric monoidal equivalence.
	
	The top map is a symmetric monoidal equivalence by \cref{day-prl-tensor}.
	As in the proof of \cref{sh-compat}, both the left and the right maps are symmetric monoidal.
	This shows that the bottom map is the localization of the top map, and thus inherits the structure of a symmetric monoidal equivalence.
\end{proof}

	\section{Higher Cocartesian Structure}\label{sec-cocart}

Endowing a category $\CC \in \Cat$ with a symmetric monoidal structure is the same as providing a lift $\CC^\otimes \in \CMon[\Cat]$.
If $\CC$ has finite coproducts, it has a \emph{cocartesian structure} $\CC^\cocart$ given by the coproduct.
An Eckmann--Hilton style argument characterizes it as the unique symmetric monoidal structure that commutes with coproducts (in all coordinates together), namely satisfying
$$
(X \otimes Y) \sqcup (Z \otimes W)
\cong (X \sqcup Z) \otimes (Y \sqcup W)
.
$$
Building on \cite[Theorem 5.23]{Harpaz}, in this section we define the ($p$-typical) $m$-cocartesian structure as an $m$-commutative monoid structure, and in \cref{cocart-is-rest} we show that it enjoys the expected properties, which in particular gives a construction of the ordinary cocartesian structure.
The results of this section feature in the definition of higher semiadditive algebraic K-theory in \cref{sa-k-def}, by preserving the $m$-commutative monoid structure afforded by the $m$-cocartesian structure.

\begin{defn}\label{lam-monoidal}
	The category of categories with a \tdef{($p$-typical) $m$-symmetric monoidal structure} is $\CMonm[\Cat]$.
	That is, an $m$-symmetric monoidal structure on $\CC \in \Cat$ is a lift $\CC^\otimes \in \CMonm[\Cat]$.
\end{defn}

In \cite[Theorem 5.23]{Harpaz} and \cite[Proposition 2.2.7]{AmbiHeight} it is shown that the category $\Catmfin$ of categories admitting colimits indexed by $m$-finite $p$-spaces is itself an ($p$-typically) $m$-semiadditive category for any $-2 \leq m \leq \infty$ (the proofs in the cited papers are not in the $p$-typical case, but the same proofs work in the $p$-typical case).
In other words, the underlying functor
$$\underlying\colon \CMonm[\Catmfin] \to \Catmfin$$
is an equivalence.
We denote its inverse by $\mdef{(-)^{\mcocart}}\colon \Catmfin \iso \CMonm[\Catmfin]$.
We recall from \cref{cmon-cat-lam} that there is an inclusion $\CMonm[\Catmfin] \to \CMonm[\Cat]$.

\begin{defn}
	For every $\CC \in \Catmfin$, we call $\mdef{\CC^{\mcocart}} \in \CMonm[\Cat]$ the \tdef{$m$-cocartesian structure} on $\CC$.
	When $m$ is clear from the context, we shall write $\CC^\cocart$ for $\CC^{\mcocart}$.
\end{defn}

Our next goal is to justify this name.
In particular, we will show that for every $m$-finite $p$-space $A$, the map $\CC^A \to \CC$ induced by evaluating $\CC^\cocart$ at $A \to *$ is given by taking the colimit over $A$.
More precisely, for any $\CC \in \Cat$, let $\CC^* \in \Fun((\Spacesm)^\op,\Cat)$ be the functor $\Fun(-, \CC)$, given by sending $m$-finite $p$-space $A$ to $\CC^A$ and $q\colon A \to B$ to $q^*\colon \CC^B \to \CC^A$.
If we assume that $\CC \in \Catmfin$, then $q^*\colon \CC^B \to \CC^A$ has a left adjoint $q_!\colon \CC^A \to \CC^B$.
By passing to the left adjoints, we obtain a functor $\CC_! \in \Fun(\Spacesm,\Cat)$.
The main result of this section is then:

\begin{thm}\label{cocart-is-rest}
	The restriction of $\CC^\cocart \in \CMonm[\Cat]$ along the right-way maps $\Spacesm \to \Span(\Spacesm)$ is $\CC^*$, and similarly along the wrong-way maps $(\Spacesm)^\op \to \Span(\Spacesm)$ is $\CC_!$.
\end{thm}

To prove this, we first note that not only each $q^*$ has a left adjoint $q_!$, but they also satisfy the Beck--Chevalley condition.
This means that $\CC^*$ is in fact in $\Fun^{\BC}((\Spacesm)^\op,\Cat)$ (where $\Fun^{\BC}$ are functors such that each morphism is mapped to a right adjoint, such that the Beck--Chevalley condition is satisfied).
We will use Barwick's unfurling construction \cite[Definition 11.3]{Unfurling}.
Barwick works in a more general context, allowing to prescribe only certain right- and wrong-way morphisms, but we shall not use this generality.
After straightening, the unfurling construction for $(\Spacesm)^\op$ takes a functor $F \in \Fun^{\BC}((\Spacesm)^\op,\Cat)$, and produces a new functor $\unfurling(F) \in \Fun(\Span(\Spacesm)^\op,\Cat)$, and enjoys the following properties:

\begin{thm}[{\cite[Proposition 11.6 and Theorem 12.2]{Unfurling}}]
	For any $F \in \Fun^{\BC}((\Spacesm)^\op,\Cat)$,
	the restriction of $\unfurling(F)$ along the right-way maps $\Spacesm \to \Span(\Spacesm)$ is $F$, and similarly along the wrong-way $(\Spacesm)^\op \to \Span(\Spacesm)$ is the functor sending the morphisms to the left adjoints.
\end{thm}

Using this result, we our now in position to prove \cref{cocart-is-rest}.

\begin{proof}[Proof of \cref{cocart-is-rest}]
	By Barwick's theorem, $\unfurling(\CC^*)$ has the properties we ought to prove for $\CC^\cocart$, so it suffices to show that $\CC^\cocart \cong \unfurling(\CC^*)$.
	Furthermore, recall that that the underlying functor $\underlying\colon \CMonm[\Catmfin] \to \Catmfin$ is an equivalence, so each $\CC \in \Catmfin$ has a unique lift to $\CMonm[\Catmfin]$.
	Therefore, it suffices to show that $\unfurling(\CC^*) \in \Fun(\Span(\Spacesm)^\op,\Cat)$ is in $\CMonm[\Catmfin]$, i.e.\ that it satisfies the $m$-Segal condition and that it takes values in  $\Catmfin$.
	
	First, by construction, $\CC^*$ satisfies the $m$-Segal condition.
	Since the restriction of $\unfurling(\CC^*)$ along $(\Spacesm)^\op \to \Span(\Spacesm)$ is $\CC^*$, it follows that it satisfies the $m$-Segal condition as well, thus $\unfurling(\CC^*) \in \CMonm[\Cat]$.
	
	Second, we need to show that $\unfurling(\CC^*)$ lands in $\Catmfin$.
	By assumption $\CC \in \Catmfin$, thus the same holds for $\CC^A$ for all $m$-finite $p$-space $A$.
	For morphisms, we need to show they are sent to functors that commute with colimits indexed by any $m$-finite $p$-space $A$.
	Any morphism in $\Span(\Spacesm)$ is the composition of a right-way and a wrong-way map, so we can check these separately.
	So let $q\colon A \to B$ be a morphism of $m$-finite $p$-spaces.
	Since $q_!$ is a left adjoint, it commutes with colimits indexed by any $m$-finite $p$-space $A$, so it is a morphism in $\Catmfin$.
	Since colimits in functor categories are computed level-wise, the functor $q^*$ commutes with them, so it is also a morphism in $\Catmfin$.
\end{proof}

\begin{remark}
	In light of Barwick's construction, one could define the $m$-cocartesian structure simply by $\CC^\cocart = \unfurling(\CC^*)$.
	The reason why we define it via the equivalence $(-)^{\cocart}\colon \Catmfin \iso \CMonm[\Catmfin]$ is two fold.
	First, this construction characterizes $\CC^\cocart$ in a universal way.
	Second, Barwick's unfurling construction, although much more general then our definition, is not shown to be functorial in $F$, which will be used crucially for $\CC^\cocart$ in our definition of semiadditive algebraic K-theory.
\end{remark}

\begin{thm}\label{cocart-equiv}
	The restriction of
	$(-)^{\cocart}\colon \Catmfin \iso \CMonm[\Catmfin]$
	to $\Catmfinst$ lands in $\CMonm[\Catmfinst]$,
	and induces an equivalence
	$(-)^{\cocart}\colon \Catmfinst\iso\CMonm[\Catmfinst]$.
\end{thm}

\begin{proof}
	Let $\CC\in\Catmfinst$.
	We know that $\CC^\cocart\in\CMonm[\Catmfin]$.
	By \cref{lim-cat-lam}, for any $m$-finite $p$-space $A$, $\CC^A$ is computed the same in $\Cat$, $\Catmfin$ and $\Catmfinst$, and in particular it is stable.
	Furthermore, for any $q\colon A \to B$, both $q_!$ and $q^*$ are exact.
	Thus $\CC^\cocart \in \Fun(\Span(\Spacesm)^\op,\Catmfinst)$.
	Again by \cref{lim-cat-lam}, it satisfies the $m$-Segal condition so that $\CC^\cocart \in \CMonm[\Catmfinst]$.
	
	For a functor $F\colon \CC \to \DD$ in $\Catmfinst$, we have an induced functor $\CC^\cocart \to \DD^{\cocart}$ in $\CMonm[\Catmfin]$.
	At every point it is given by $F\circ-\colon \CC^A \to \DD^{A}$, which is exact, thus a map in $\Catmfinst$.
	This finishes the first part, showing that $(-)^{\cocart}$ lands in $\CMonm[\Catmfinst]$.
	
	By \cref{cmon-cat-lam}, the inclusion $\Catmfinst \to \Catmfin$ induces an inclusion $\CMonm[\Catmfinst]\to\CMonm[\Catmfin]$.
	The maps
	$\underline{(-)}\colon 
	\CMonm[\Catmfin]
	\rightleftarrows
	\Catmfin
	\colon (-)^{\cocart}$,
	which are inverses to each other, restrict to maps between the subcategories, which are therefore inverses to each other as well.
\end{proof}

	\section{Semiadditive Algebraic K-Theory}\label{sec-k}

In this section we define an $m$-semiadditive version of algebraic K-theory.
We begin by recalling the construction of ordinary algebraic K-theory, and present it in a way which is amenable to generalizations.
We then generalize the definition to construct $m$-semiadditive algebraic K-theory in \cref{sa-k-def}, and connect it to ordinary algebraic K-theory in \cref{sheafification-ord}.
We leverage this connection in \cref{k-lax} to endow the functor of $m$-semiadditive algebraic K-theory with a lax symmetric monoidal structure.
This is later used to prove \cref{Km-height-0} and \cref{Km-cJW}, two of the main results of this paper.

\subsection{Ordinary Algebraic K-Theory}

We recall the definition of the $S_\bullet$-construction for stable categories and exact functors.
One defines the functor $S_{\bullet}\colon \Catst \to \Spaces^{\Delta^\op}$ by letting $S_n \CC$ be the subspace of those functors $X\colon [n]^{[1]} \to \CC$ that satisfy:
\begin{enumerate}
	\item $X_{ii}=0$,
	\item For all $i\leq j\leq k$ the following is a bicartesian square
	$$\begin{tikzcd}
	X_{ij} \arrow{r}{} \arrow{d}{} & X_{ik} \arrow{d}{} \\
	X_{jj} \arrow{r}{} & X_{jk}
	\end{tikzcd}$$
	that is, $X_{ij}\to X_{ik}\to X_{jk}$ is a (co)fiber sequence.
\end{enumerate}

The \tdef{algebraic K-theory \emph{space}} functor $\mdef{\underline{\KK}}\colon \Catst \to \Spaces$ is then defined as the composition $\underline{\KK}(\CC) = \Omega|S_\bullet\CC|$.
One then proceeds to lift to (connective) spectra, e.g.\ by means of iterated $S_\bullet$-construction.
We will give an equivalent construction of the spectrum structure, which will be easier to generalize.
To that end, we show the following:

\begin{lem}\label{S-lims}
	The functor $S_\bullet\colon \Catst \to \Spaces^{\Delta^\op}$ commutes with limits.
\end{lem}

\begin{proof}
	For each $n$, the functor $S_n\colon \Catst \to \Spaces$ is equivalent to $\hom([n-1],-)$, and in particular it commutes with limits.
	Since limits in the functor category $\Spaces^{\Delta^\op}$ are computed level-wise, this implies that $S_\bullet$ commutes with limits as well.
\end{proof}

This together with \cref{lim-cmon} implies that we get an induced functor $S_{\bullet}\colon \CMon[\Catst]\to\CMon[\Spaces]^{\Delta^\op}$.
Employing \cref{cocart-equiv}, we give the following definition.

\begin{defn}\label{k-def}
	We define \tdef{algebraic K-theory} $\mdef{\KK}\colon \Catst \to \Sp$ by $\KK(\CC)=\Omega|(S_\bullet(\CC^\cocart))^{\gpc}|$, that is, as the following composition
	$$
		\Catst
		\xrightarrow{(-)^\cocart} \CMon[\Catst]
		\xrightarrow{S_\bullet} \CMon[\Spaces]^{\Delta^\op}
		\xrightarrow{(-)^{\gpc}} \Sp^{\Delta^\op}
		\xrightarrow{|-|} \Sp
		\xrightarrow{\Omega} \Sp
		.
	$$
\end{defn}

\begin{lem}
	The composition of $\KK\colon \Catst \to \Sp$ with $\Omega^\infty\colon \Sp \to \Spaces$ is $\underline{\KK}$.
\end{lem}

\begin{proof}
	First note that $(-)^{\gpc}\colon \CMon[\Spaces] \to \Sp$ is a left adjoint, and therefore commutes with the colimit $|-|$, and that $\Omega((-)^{\gpc}) \cong \Omega$ as functors $\CMon[\Spaces] \to \Sp$.
	This shows that our definition of algebraic K-theory is equivalent to the composition
	$$
		\Catst
		\xrightarrow{(-)^\cocart} \CMon[\Catst]
		\xrightarrow{S_\bullet} \CMon[\Spaces]^{\Delta^\op}
		\xrightarrow{|-|} \CMon[\Spaces]
		\xrightarrow{\Omega} \Sp
		.
	$$
	Consider the following diagram:
	$$\begin{tikzcd}
		\Catst \arrow{r}{(-)^\cocart} \arrow[d, equals] \arrow[dr, phantom, "\text{\small{(1)}}", blue]
		& \CMon[\Catst] \arrow{r}{S_\bullet} \arrow{d}{\underlying} \arrow[dr, phantom, "\text{\small{(2)}}", blue]
		& \CMon[\Spaces]^{\Delta^\op} \arrow{r}{|-|} \arrow{d}{\underlying} \arrow[dr, phantom, "\text{\small{(3)}}", blue]
		& \CMon[\Spaces] \arrow{r}{\Omega} \arrow{d}{\underlying} \arrow[dr, phantom, "\text{\small{(4)}}", blue]
		& \Sp \arrow{d}{\underlying}
		\\
		\Catst \arrow[r, equals]
		& \Catst \arrow{r}{S_\bullet}
		& \Spaces^{\Delta^\op} \arrow{r}{|-|}
		& \Spaces \arrow{r}{\Omega}
		& \Spaces
	\end{tikzcd}$$
	Square (1) commutes because $(-)^\cocart$ and $\underlying$ are inverses by \cref{cocart-equiv}.
	Square (2) commutes by the definition of the extension of $S_\bullet$ to $\CMon$.
	Square (3) commutes since the underlying commutes with geometric realizations.
	Square (4) commutes because $\Omega$ is a limit and the underlying is a right adjoint functor.
	Finally, the top-right composition is $\Omega^\infty \KK$, whereas the left-bottom composition is $\underline{\KK}$.
\end{proof}

We now claim that the above definition of the spectrum structure coincides with the standard one.
Note that by construction $\KK$ in fact lands in connective spectra.

\begin{prop}\label{k-unique-deloop}
	There is a unique lift of $\underline{\KK}\colon \Catst \to \Spaces$ to $\KK\colon \Catst \to \conSp$.
\end{prop}

\begin{proof}
	The universal property of commutative monoids given in \cite[Corollary 5.15]{Harpaz} implies that if $\DD$ has finite products and $\CC$ is semiadditive, then
	$
	\Fun^\times(\CC, \DD)
	\cong \Fun^\times(\CC, \CMon[\DD])
	$.
	Since $\CMongl[\DD] \subseteq \CMon[\DD]$ is a full subcategory closed under products,
	$$
	\Fun^\times(\CC, \CMongl[\DD])
	\subseteq \Fun^\times(\CC, \CMon[\DD])
	$$
	is a full subcategory.
	Therefore, the forgetful
	$$
	\Fun^\times(\CC, \CMongl[\DD])
	\to \Fun^\times(\CC, \DD)
	$$
	is fully faithful, meaning that product preserving functors $\CC \to \DD$ have unique or no lifts to $\CMongl[\DD]$.
	In particular, for $\DD = \Spaces$, using the equivalence $\CMongl[\Spaces] \cong \conSp$, we get that the forgetful
	$$
	\Fun^\times(\CC, \conSp)
	\to \Fun^\times(\CC, \Spaces)
	$$
	is fully faithful.

	Applying this to the case $\CC = \Catst$, the result follows since $\underline{\KK}$ has a lift, which is therefore unique.
\end{proof}

\subsection{Definition of Semiadditive Algebraic K-Theory}

We restrict the $S_\bullet$-construction to $\Catmfinst$, and use the same notation i.e.\ $S_\bullet\colon \Catmfinst \to \Spaces^{\Delta^\op}$.
\cref{lim-cat-lam} shows that $\Catmfinst \to \Catst$ preserve limits, thus by \cref{S-lims}, the restriction $S_\bullet\colon \Catmfinst \to \Spaces^{\Delta^\op}$ preserves limits as well, so using \cref{lim-cmon} again we get an induced functor $S_{\bullet}\colon \CMonm[\Catmfinst] \to \CMonm[\Spaces]^{\Delta^\op}$.
Employing \cref{cocart-equiv}, we give the following definition.

\begin{defn}\label{sa-k-def}
	We define \tdef{$m$-semiadditive algebraic K-theory} $\mdef{\Km}\colon \Catmfinst \to \tsadim$ by $\Km(\CC)=\Omega|(S_\bullet(\CC^\cocart))^{\gpc}|$, that is, as the following composition
	$$
		\Catmfinst
		\xrightarrow{(-)^\cocart} \CMonm[\Catmfinst]
		\xrightarrow{S_\bullet} \CMonm[\Spaces]^{\Delta^\op}
		\xrightarrow{(-)^{\gpc}} (\tsadim)^{\Delta^\op}
		\xrightarrow{|-|} \tsadim
		\xrightarrow{\Omega} \tsadim
		.	
	$$
\end{defn}

\begin{example}\label{K0-is-K}
	\cref{k-unique-deloop} shows that the case $m = 0$ recovers the $p$-localization of the ordinary K-theory of stable categories.
\end{example}

\begin{prop}\label{Klam-limits}
	The functor $\Km\colon \Catmfinst \to \tsadim$ is an $m$-semiadditive functor, i.e.\ commutes with all limits indexed by an $m$-finite $p$-space.
	In particular, $\Km(\CC^A) \cong \Km(\CC)^A$ for any $m$-finite $p$-space $A$.
\end{prop}

\begin{proof}
	The functor $\Km\colon \Catmfinst \to \tsadim$ is defined as the composition of functors between $m$-semiadditive categories (\cite[Proposition 2.1.4 (1) and (2)]{AmbiHeight} imply the $m$-semiadditivity of $\CMonm[\Spaces]^{\Delta^\op}$ and $(\tsadim)^{\Delta^\op}$).
	All of the functors either preserve all limits (in the case of $(-)^\cocart, S_\bullet$ and $\Omega$) or preserve all colimits (in the case of $(-)^{\gpc}$ and $|-|$).
	In particular they are all $m$-semiadditive functors, thus the composition is an $m$-semiadditive functor as well.
\end{proof}

\subsection{Relationship to Ordinary Algebraic K-Theory}

\cref{Klam-limits} shows that $\Km$ is an $m$-semiadditive functor, and in particular satisfies $\Km(\CC^A) \cong \Km(\CC)^A$ for any $m$-finite $p$-space $A$.
One may wonder if $\Km$ can be obtained by forcing ordinary algebraic K-theory to satisfy this condition.
In this subsection we show a more general result of this sort.
To be more specific, let $m_0 \leq m$, then \cref{Kmmz} introduces a functor $\Kmmz$, which associates to $\CC \in \Catmfinst$ the pre-$m$-commutative monoid given on objects by $A \mapsto \Kmz(\CC^A)$.
The main result of this subsection is \cref{sheafification}, which shows that forcing the $m$-Segal condition on $\Kmmz$ is indeed $\Km$.
In particular, the case $m_0 = 0$ yields an alternative definition of $m$-semiadditive algebraic K-theory, by forcing $A \mapsto \KK(\CC^A)$ to satisfy the $m$-Segal condition.

Consider the inclusion $i\colon \CMonm[\Catmfinst] \subseteq \PCMonm[\Catmfinst] \to \PCMonm[\Catmzfinst]$.
Using this we are lead to the main definition.

\begin{defn}\label{Kmmz}
	We define the functor $\mdef{\Kmmz}\colon \Catmfinst \to \PCMonm[\tsadimz]$ by the following composition:
	$$
		\Catmfinst
		\xrightarrow{(-)^\cocart} \CMonm[\Catmfinst]
		\xrightarrow{i} \PCMonm[\Catmzfinst]
		\xrightarrow{\Kmz\circ-} \PCMonm[\tsadimz]
	$$
\end{defn}

We recall that for any $\DD$ we have an equivalence $\CMonm[\DD] \cong \CMonm[{\CMonmz[\DD]}]$ (which is given by sending $X \in \CMonm[\DD]$ to the iterated commutative monoid given on objects by $A \mapsto (B \mapsto X(A \times B))$).
In particular, we can consider it as a full subcategory $\CMonm[\DD] \subseteq \PCMonm[{\CMonmz[\DD]}]$, and this inclusion has a left adjoint $\Lseg$.

Applying the above for $\DD = \Sp_{(p)}$ shows that we have an inclusion of a full subcategory $\tsadim \subseteq \PCMonm[\tsadimz]$ with a left adjoint $\Lseg$.
This allows us consider $\Km\colon \Catmfinst \to \tsadim$ as a functor to $\PCMonm[\tsadimz]$.

\begin{thm}\label{sheafification}
	There is a natural equivalence $\Lseg \Kmmz \iso \Km$ of functors $\Catmfinst \to \tsadim$.
\end{thm}

\begin{proof}
	We show that the square
	$$\begin{tikzcd}
		\Catmfinst \arrow{r}{\Kmmz} \arrow[equal]{d}{}
		& \PCMonm[\tsadimz] \arrow{d}{\Lseg} \\
		\Catmfinst \arrow{r}{\Km}
		& \tsadim
	\end{tikzcd}$$
	commutes by admitting it as a (horizontal) composition of commutative squares, following the definition of the two functors.

	Both functors start with
	$\Catmfinst \xrightarrow{(-)^\cocart} \CMonm[\Catmfinst]$.
	
	Consider the square below.
	The lift of $S_\bullet$ to (pre-)$m$-commutative monoids is given by post-composition.
	From this description, the upper composition in fact lands in $\CMonm[{\CMonmz[\Spaces]^{\Delta^\op}}]$, so $\Lseg$ acts on the image as the identity, making the square commute.
	$$\begin{tikzcd}
		\CMonm[\Catmfinst] \arrow{r}{i} \arrow[equal]{d}{}
		& \PCMonm[{\Catmzfinst}] \arrow{r}{S_\bullet(-)^\mzcocart}
		& \PCMonm[{\CMonmz[\Spaces]^{\Delta^\op}}] \arrow{d}{\Lseg} \\
		\CMonm[\Catmfinst] \arrow{rr}{S_\bullet}
		&
		& \CMonm[\Spaces]^{\Delta^\op}
	\end{tikzcd}$$
	
	The following square commutes because all maps are left adjoints and the square of right adjoints commutes because they are all forgetfuls.
	$$\begin{tikzcd}
		\PCMonm[{\CMonmz[\Spaces]^{\Delta^\op}}] \arrow{r}{(-)^{\gpc}} \arrow{d}{\Lseg}
		& \PCMonm[(\tsadimz)^{\Delta^\op}] \arrow{d}{\Lseg} \\
		\CMonm[\Spaces]^{\Delta^\op} \arrow{r}{(-)^{\gpc}}
		& (\tsadim)^{\Delta^\op}
	\end{tikzcd}$$
	
	The following square commutes because $\Lseg$ is a left adjoint, thus commutes with colimits.
	$$\begin{tikzcd}
		\PCMonm[(\tsadimz)^{\Delta^\op}] \arrow{r}{|-|} \arrow{d}{\Lseg}
		& \PCMonm[\tsadimz] \arrow{d}{\Lseg} \\
		(\tsadim)^{\Delta^\op} \arrow{r}{|-|}
		& \tsadim
	\end{tikzcd}$$
	
	Lastly, $\Lseg$ is an exact functor between stable categories, thus it commutes with finite limits, so the following square commutes.
	$$\begin{tikzcd}
		\PCMonm[\tsadimz] \arrow{r}{\Omega} \arrow{d}{\Lseg}
		& \PCMonm[\tsadimz] \arrow{d}{\Lseg} \\
		\tsadim \arrow{r}{\Omega}
		& \tsadim
	\end{tikzcd}$$
\end{proof}

In particular, restricting to the case $m_0 = 0$, we get that the functor $\Kzm$ given by $A \mapsto \KK(\CC^A)$ satisfies the following:

\begin{cor}\label{sheafification-ord}
	There is a natural equivalence $\Lseg \Kzm \iso \Km$ of functors $\Catmfinst \to \tsadim$.
	In particular, if $\Kzm(\CC)\colon \Span(\Spacesm)^\op \to \Sp_{(p)}$ satisfies the $m$-Segal condition, then $\Km(\CC) \cong \Kzm(\CC)$.
\end{cor}

\begin{cor}\label{K-to-Km}
	There is a natural transformation $\KK \to \underline{\Km}$ of functors $\Catmfinst \to \Sp_{(p)}$.
\end{cor}

\begin{proof}
	By adjunction, the equivalence $\Lseg \Kzm \iso \Km$ of \cref{sheafification-ord} corresponds to a natural transformation $\Kzm \to \Km$ of functors $\Catmfinst \to \PCMonm[\Sp_{(p)}]$, where the target lands in $\tsadim$.
	Evaluating the pre-$m$-commutative monoids at the point gives the desired natural transformation.
\end{proof}

Recall that $\SpTn$ is $\infty$-semiadditive.
In particular, $\SpTn \cong \CMonm[\SpTn]$.
Additionally, there is a canonical map of modes $\LTn^{\tsadim}\colon \tsadim \to \SpTn$.
Recall from \cite[Corollary 5.5.14]{AmbiHeight} that $\LTn^{\tsadi}\colon \tsadi \to \SpTn$ is a smashing localization, here we prove a slight generalization.

\begin{lem}\label{tsadim-Tn-smash}
	For any $m \geq 1$ and $n \geq 0$, the functor $\LTn^{\tsadim}\colon \tsadim \to \SpTn$ is a smashing localization of modes.
\end{lem}

\begin{proof}
	Recall that $\Lnf\colon \Sp \to \LnfSp$ is a smashing localization.
	Tensoring with the stable mode $\tsadim$, by \cref{mod-smash-tensor} we get that $\tsadim \to \tsadim \otimes \LnfSp$ is a smashing localization of modes.
	Now, by \cite[Theorem G]{AmbiHeight} we know that $\tsadim \otimes \LnfSp \cong \SpTze \times \cdots \times \SpTn$, and the projection to a factor is a smashing localization of modes as well.
	Therefore, the composition $\LTn^{\tsadim}\colon \tsadim \to \SpTn$ is also smashing localization of modes.
\end{proof}

\begin{prop}\label{tsadi-to-SpTn}
	The following square commutes:
	$$\begin{tikzcd}
	\PCMonm[\Sp_{(p)}] \arrow{r}{\LTn} \arrow{d}{\Lseg}
	& \PCMonm[\SpTn] \arrow{d}{\Lseg} \\
	\tsadim \arrow{r}{\LTn^{\tsadim}}
	& \SpTn
	\end{tikzcd}$$
\end{prop}

\begin{proof}
	First recall that by definition $\tsadim = \CMonm[\Sp_{(p)}]$, and as explained above, $\SpTn \cong \CMonm[\SpTn]$.
	All of the morphisms in the square in the statement are left adjoints.
	Using the two identifications and passing to the right adjoints we obtain the square:
	$$\begin{tikzcd}
	\PCMonm[\Sp_{(p)}]
	& \PCMonm[\SpTn] \arrow[hook']{l}{} \\
	\CMonm[\Sp_{(p)}] \arrow[hook]{u}{}
	& \CMonm[\SpTn] \arrow[hook']{l}{} \arrow[hook]{u}{}
	\end{tikzcd}$$
	This square commutes as all morphisms are inclusions, thus the original square of left adjoints commutes as well.
\end{proof}

\begin{cor}\label{sheafification-Tn}
	There is an equivalence $\Lseg \LTn \Kzm \iso \LTn^{\tsadim} \Km$.
	In particular, if $\LTn \Kzm(\CC)$ satisfies the $m$-Segal condition, then $\LTn \KK(\CC) \cong \LTn^{\tsadim} \Km(\CC)$.
\end{cor}

\begin{proof}
	The first part follows immediately from \cref{sheafification-ord} and \cref{tsadi-to-SpTn}.
	For the second part, if $\LTn \Kzm(\CC)$ satisfies the $m$-Segal condition, then by the first part there is an equivalence $\LTn \Kzm(\CC) \cong \LTn^{\tsadim} \Km(\CC)$.
	The equivalence $\CMonm[\SpTn] \cong \SpTn$, given by taking the underlying object, identifies $\LTn \Kzm(\CC)$ with $\LTn \KK(\CC)$, showing that indeed $\LTn \KK(\CC) \cong \LTn^{\tsadim} \Km(\CC)$.
\end{proof}

\begin{cor}\label{K-to-Km-Tn}
	There is a natural transformation $\LTn \KK \to \LTn^{\tsadim} \Km$ of functors $\Catmfinst \to \SpTn$.
\end{cor}

\begin{proof}
	Similarly to \cref{K-to-Km}, by adjunction, the equivalence $\Lseg \LTn \Kzm \iso \LTn^{\tsadim} \Km$ of \cref{sheafification-Tn} corresponds to a natural transformation $\LTn \Kzm \to \LTn^{\tsadim} \Km$ of functors $\Catmfinst \to \PCMonm[\SpTn]$, where the target lands in $\CMonm[\SpTn] \cong \SpTn$.
	Evaluating the pre-$m$-commutative monoids at the point gives the desired natural transformation.
\end{proof}

\subsection{Multiplicative Structure}

Using \cref{sheafification-ord} we leverage the lax symmetric monoidal structure on algebraic K-theory developed in \cite[Corollary 1.6]{BGT} and \cite[Proposition 3.8]{MulK} to construct a lax symmetric monoidal structure on $m$-semiadditive algebraic K-theory.

Recall that for any collection of indexing categories $\clK$, $\CatK$ has a symmetric monoidal structure constructed in \cite[\S 4.8.1]{HA}.
If $\clK$ contains all finite categories, then $\Catst_\clK$ is the full subcategory on those categories that are in addition stable, which is also endowed with a symmetric monoidal structure (but is not a sub-symmetric monoidal category of $\Cat_\clK$, whose unit is not stable).

\begin{lem}
	The inclusion $\Catmfinst \to \Catst$ is lax symmetric monoidal.
\end{lem}

\begin{proof}
	We argue similarly to the proof of \cref{tsadi-atomics}.
	Recall again that there is a fully faithful functor $(-)^\otimes\colon \CMon(\Cat)^\lax \to \Op$ from the category of symmetric monoidal categories and lax symmetric monoidal functors to operads.
	For any collection of indexing categories $\clK$ the operad $\CatK^{\st,\otimes}$ is a sub-operad of $\Cat_\clK^\otimes$.

	The category $\Catmfinst$ is the case where $\clK$ is the collection of all finite categories and $m$-finite $p$-spaces, and $\Catst$ is the case where $\clK'$ the collection of all finite categories.
	The same proof of \cite[Corollary 4.8.1.4]{HA} shows that the inclusion $\Cat_\clK \to \Cat_{\clK'}$ is lax symmetric monoidal.
	We thus get a map of operads $\Cat_\clK^\otimes \to \Cat_{\clK'}^\otimes$.
	The restriction of this map to $\Cat_{\mfin}^{\st,\otimes}$ lands in $\Cat^{\st,\otimes}$, so we get a map of operads $\Cat_{\mfin}^{\st,\otimes} \to \Cat^{\st,\otimes}$, that is, a lax symmetric monoidal functor $\Catmfinst \to \Catst$.
\end{proof}

\begin{prop}\label{pre-k-lax}
	The functor $\Kzm\colon \Catmfinst \to \PCMonm[\Sp_{(p)}]$ is lax symmetric monoidal.
\end{prop}

\begin{proof}
	Recall that $\Kzm$ is given by the composition:
	\begin{align*}
		\Catmfinst
		&\xrightarrow{(-)^\cocart} \CMonm[\Catmfinst]\\
		&\subseteq \PCMonm[\Catmfinst]\\
		&\to \PCMonm[\Catst]\\
		&\xrightarrow{\KK\circ-} \PCMonm[\Sp_{(p)}]
		.
	\end{align*}
	The first functor is symmetric monoidal by \cref{cmon-mul}, which also shows that the second map is lax symmetric monoidal as the right adjoint of the symmetric monoidal functor $\Lseg$.
	The third and fourth maps are post-composition with the lax symmetric monoidal functors $\Catmfinst \to \Catst$ and $\KK$, which are therefore also lax symmetric monoidal by \cref{day-range}.
\end{proof}

\begin{thm}\label{k-lax}
	The functor $\Km\colon \Catmfinst \to \tsadim$ is lax symmetric monoidal.
\end{thm}

\begin{proof}
	Recall from \cref{sheafification-ord} that $\Km \cong \Lseg \Kzm$.
	\cref{pre-k-lax} shows that $\Kzm$ is lax symmetric monoidal, and $\Lseg$ is symmetric monoidal by \cref{cmon-mul}.
\end{proof}

By construction, the lax symmetric monoidal structure on $\Km$ is compatible with that of $\Kzm$, so we immediately get the following:

\begin{cor}\label{K-to-Km-sm}
	The natural transformations $\KK \to \underline{\Km}$ and $\LTn \KK \to \LTn^{\tsadim} \Km$ from \cref{K-to-Km} and \cref{K-to-Km-Tn} (respectively) are symmetric monoidal.
\end{cor}

Recall from \cref{mod-func} that there is a lax symmetric monoidal functor
$\LMod_{(-)}\colon \Alg(\tsadim) \to \Mod_{\tsadim}^\intL$.
In addition, \cref{tsadi-atomics} gives a lax symmetric monoidal functor $(-)^{\atomic} \colon \Mod_{\tsadim}^{\intL} \to \Catmfinst$.
Composing all of the above we arrive at the following definition:

\begin{defn}\label{ring-sa-k-def}
	We define the lax symmetric monoidal functor $\Km\colon \Alg(\tsadim) \to \tsadim$ by
	$\mdef{\Km}(R) = \Km(\LMod_R^{\atomic})$, i.e.\ the composition
	$$
	\Alg(\tsadim)
	\xrightarrow{\LMod_{(-)}} \Mod_{\tsadim}
	\xrightarrow{(-)^{\atomic}} \Catmfinst
	\xrightarrow{\Km} \tsadim
	.
	$$
\end{defn}

\begin{remark}\label{k-r-dbl}
	Recall from \cref{atomic-ldbl} that the atomics coincide with the left dualizable modules, therefore $\Km(R) = \Km(\LMod_R^{\ldbl})$.
\end{remark}

\begin{remark}\label{Tn-K}
	Note that if $R \in \Alg(\SpTn)$ and $m \geq n$, then since $\SpTn$ is a smashing localization of $\tsadim$ by \cref{tsadim-Tn-smash}, we know that $\LMod_R(\tsadim) \cong \LMod_R(\SpTn)$, and the atomics are the left dualizable objects.
	Therefore we get that
	$\Km(R) \cong \Km(\LMod_R(\SpTn)^\ldbl)$.
\end{remark}

\begin{remark}
	As $\Km\colon \Alg(\tsadim) \to \tsadim$ is lax symmetric monoidal, we conclude that it sends $\OO \otimes \bbE_1$-algebras in $\tsadim$ to $\OO$-algebras in $\tsadim$ for any operad $\OO$.
	In particular, for the case $\OO = \bbE_n$, we get that $\Km$ sends $\bbE_{n+1}$-algebras to $\bbE_n$-algebras, and for $\OO = \bbE_\infty$ we get $\Km\colon \CAlg(\tsadim) \to \CAlg(\tsadim)$.
	As in \cref{mod-calg}, in this case $\Km(R) = \Km(\Mod_R^\dbl)$.
\end{remark}

	\section{Redshift}\label{sec-red}

Recall that the redshift philosophy predicts that algebraic K-theory increases height by $1$.
In this section we prove some results concerning the interplay between semiadditive height and higher semiadditive algebraic K-theory.

An immediate application of the redshift result of \cite[Theorem B]{AmbiHeight}, gives an upper bound, showing that if $R \in \Alg(\tsadim)$ has semiadditive height $\leq n$ for some finite $n < m$, then $\Km(R)$ has semiadditive height $\leq n+1$ (see \cref{redshift-upper}).
Furthermore, in \cref{redshift-lower} we show that if $R$ has semiadditive height exactly $n$, and has (height $n$) $p$-th roots of unity (see \cref{p-roots}), then $\Km(R)$ has semiadditive height exactly $n+1$, i.e.\ lands in $\tsadi_{n+1}$.
In particular, the Lubin--Tate spectrum $\rE_n$ has this property, so we conclude that $\Km(\rE_n) \in \tsadi_{n+1}$ (see \cref{KEn-tsadinpo}).

\subsection{Semiadditive Height}

We begin by recalling the notion of (semiadditive) height from \cite[Definition 3.1.6]{AmbiHeight} and making a few observations which will be used to study the interaction between height and semiadditive algebraic K-theory.
We recall from \cite[Definition 3.1.3]{AmbiHeight} that for every $m$-semiadditive category $\DD$, and finite $n \leq m$, there is a natural transformation of the identity $p_{(n)}\colon \id_\DD \Rightarrow \id_\DD$, also denoted by $|\BB^n C_p|$, which is given on an object $Y \in \DD$ by
$$
	p_{(n)}\colon Y
	\xrightarrow{\Delta} Y^{\BB^n C_p}
	\xrightarrow{\mrm{Nm}^{-1}} \BB^n C_p \otimes Y
	\xrightarrow{\nabla} Y
	,
$$
using the fact that the norm map is an equivalence.
Alternatively, as $\DD$ is $m$-semiadditive, its objects have a canonical $m$-commutative monoid structure in $\DD$, so that the map is given by $q_! q^*$ where $q\colon \BB^n C_p \to *$ is the unique map.

\begin{defn}[{\cite[Definition 3.1.6]{AmbiHeight}}]
	Let $Y \in \DD$, then its \tdef{semiadditive height} is defined as follows:
	\begin{enumerate}
		\item $\height(Y) \leq n$ if $p_{(n)}\colon Y \to Y$ is invertible.
		\item $\height(Y) > n$ if $Y$ is $p_{(n)}$-complete (i.e., for every $Z \in \DD$ with $p_{(n)}\colon Z \to Z$ invertible, $\hom(Z, Y) = *$).
		\item $\height(Y) = n$ if $\height(Y) \leq n$ and $\height(Y) > n-1$.
	\end{enumerate}
	We denote by $\DD_{\leq n}$ the full subcategory of objects $Y \in \DD$ with $\height(Y) \leq n$, and similarly $\DD_{>n}$ for objects of height $> n$ and $\DD_n$ for object of height exactly $n$.
\end{defn}

\begin{prop}[{\cite[Theorem A]{AmbiHeight}}]\label{leqn-infty-sa}
	Let $\DD$ be an $m$-semiadditive category which admits all limits and colimits indexed by $\pi$-finite $p$-spaces, and let $n \leq m$ be a finite number, then $\DD_{\leq n}$ is $\infty$-semiadditive.
\end{prop}

\begin{prop}[{\cite[Proposition 3.1.13]{AmbiHeight}}]\label{height-dec}
	Let $F\colon \DD \to \EE$ be an $m$-semiadditive functor, between $m$-semiadditive categories, and let $n \leq m$ be a finite number.
	Then $F$ sends $p_{(n)}$ of $\DD$ to $p_{(n)}$ of $\EE$.
	In particular, if $\height(Y) \leq n$, then $\height(FY) \leq n$, that is, it induces a functor $F\colon \DD_{\leq n} \to \EE_{\leq n}$.
\end{prop}

\begin{cor}\label{adjunction-leqn}
	Let $\DD$ and $\EE$ be $m$-semiadditive categories, $F\colon \DD \rightleftarrows \EE \colon G$ an adjunction, and $n \leq m$ a finite number.
	Then the adjunction restricts to an adjunction $F\colon \DD_{\leq n} \rightleftarrows \EE_{\leq n} \colon G$.
	Furthermore, if $\DD$ and $\EE$ admit all limits and colimits indexed by $\pi$-finite $p$-spaces, then the restricted functors are $\infty$-semiadditive.
\end{cor}

\begin{proof}
	Since $G$ and $F$ preserves limits and colimits respectively, they are $m$-semiadditive.
	By \cref{height-dec}, their restrictions to objects of height $\leq n$ land in objects of height $\leq n$.
	Since by \cref{leqn-infty-sa} $\DD_{\leq n}$ and $\EE_{\leq n}$ are $\infty$-semiadditive,
	and the restricted functors preserve limits or colimits, they are in fact $\infty$-semiadditive.
\end{proof}

\begin{prop}\label{tsadi-n}
	Let $n \leq m$ be a finite number, then the mode $\tsadim_{\leq n} \cong \tsadi_{\leq n}$ is independent of $m$, and is the mode classifying the property of being stable $p$-local $\infty$-semiadditive and having all objects of height $\leq n$.
	Furthermore, it decomposes as a product
	$$
	\tsadi_{\leq n} \cong \tsadi_0 \times \cdots \times \tsadi_n
	,
	$$
	where $\tsadi_k$ is the mode classifying the property of being stable $p$-local $\infty$-semiadditive and having all objects of height exactly $n$.
\end{prop}

\begin{proof}
	If $n < m$, then this follows immediately from \cite[Theorem 4.2.7 and Theorem 5.3.6]{AmbiHeight}.
	If $n = m$, then the same results show that
	$$
	\tsadim \cong \tsadi_0 \times \cdots \times \tsadi_{n-1} \times \tsadim_{>n-1}
	.
	$$
	By \cite[Proposition 4.2.1]{AmbiHeight}, $\tsadim_{>n-1}$ is a recollement of $\tsadi_n$ and $\tsadim_{> n}$, but $(\tsadim_{> n})_{\leq n} = 0$, so the result follows upon taking objects of height $\leq n$.
\end{proof}

Consider the case $\DD = \Catmfinst$.
In this case, the objects are themselves categories $\CC \in \DD$ on which $p_{(n)}$ acts, and can have heights $\height(\CC)$ as objects of $\Catmfinst$.

\begin{prop}\label{pn-formula}
	Let $\CC \in \Catmfinst$.
	For any $m$-finite $p$-space $A$, the map $|A|\colon \CC \to \CC$ is given by $|A|(X) \cong \colim_A X$.
	In particular, $p_{(n)}(X) \cong \colim_{\BB^n C_p} X$.
\end{prop}

\begin{proof}
	Recall that if we consider the objects of $\Catmfinst$ as equipped with the canonical $\CMonm$ structure, then $p_{(n)} \cong q_! q^*$ where $q\colon A \to *$ is the unique map.
	\cref{cocart-is-rest} and \cref{cocart-equiv} then show that $q^*\colon \CC \to \CC^A$ is taking the constant diagram and that $q_!\colon \CC^A \to \CC$ is computing the colimit.
\end{proof}

\subsection{Upper Bound}

\begin{prop}\label{redshift-small}
	Let $\CC \in \Catmfinst$ and assume that $\height(\CC) \leq n$ as an object of $\Catmfinst$ for some finite $n \leq m$,
	then $\height(\Km(\CC)) \leq n$.
	That is, $\Km$ restricts to a functor
	$\Km\colon \Catmfinstleqn \to \tsadi_{\leq n}$.
\end{prop}

\begin{proof}
	$\Km$ is $m$-semiadditive by \cref{Klam-limits}, so the result follows from \cref{height-dec}.
\end{proof}

By \cref{leqn-infty-sa}, $\Catmfinstleqn$ and $\tsadi_{\leq n}$ are $\infty$-semiadditive.

\begin{prop}\label{indep-lam}
	$\Km\colon \Catmfinstleqn \to \tsadi_{\leq n}$ is $\infty$-semiadditive.
	Furthermore, let $n \leq m_0 \leq m$ with $n$ a finite number, then
	$\Km\colon \Catmfinstleqn \to \tsadi_{\leq n}$
	and the restriction of $\Kmz$ to the (not full) subcategory $\Catmfinstleqn$ coincide.
\end{prop}

\begin{proof}
	By construction, $\Km\colon \Catmfinst \to \tsadim$ is a composition of left and right adjoints between $m$-semiadditive categories (as used in the proof of \cref{Klam-limits}), all of which admit all limits and colimits (as all categories are presentable).
	By \cref{adjunction-leqn}, $\Km\colon \Catmfinstleqn \to \tsadi_{\leq n}$ is the composition the adjoints restricted to objects of height $\leq n$, which are $\infty$-semiadditive functors.

	For the second part, note that $\Kmz\colon \Catmfinstleqn \to \tsadi_{\leq n}$ is also $\infty$-semiadditive and in particular $m$-semiadditive.
	Recall from \cref{sheafification} that there is an equivalence $\Lseg \Kmmz \iso \Km$, so when $\Kmmz$ is restricted to $\Catmfinstleqn$ it already satisfies the $m$-Segal condition and is thus equivalent to $\Km$.
\end{proof}

We recall the following redshift result, which we view as the step along the construction at which redshift happens.

\begin{thm}[\pwrap{Semiadditive Redshift {\cite[Theorem B]{AmbiHeight}}}]\label{semiadditive-redshift-2}
	Let $\CC \in \Catmfin$ be an $m$-semiadditive category and let $n < m$ be a finite number.
	Then, $\height(X) \leq n$ for all $X \in \CC$ if and only if $\height(\CC) \leq n + 1$ as an object of $\Catmfin$.
\end{thm}

\begin{cor}\label{leq-n-atomic}
	Let $\CC \in \Mod_{\tsadi_{\leq n}} \subset \Mod_{\tsadim}$ for some finite number $n < m$,
	then $\CC^{\atomic} \in \Catmfinst$ is an $m$-semiadditive category and has $\height(\CC^{\atomic}) \leq n+1$ as an object of $\Catmfinst$.
\end{cor}

\begin{proof}
	It is in $\Catmfinst$ and $m$-semiadditive by \cref{tsadi-atomics}.
	By \cref{tsadi-n}, $\tsadi_{\leq n}$ classifies the property of having all objects of height $\leq n$, so together with \cref{semiadditive-redshift-2}, this implies that $\height(\CC^{\atomic}) \leq n+1$ as an object of $\Catmfinst$.
\end{proof}

\begin{remark}
	The mode $\tsadi_{\leq n}$ is a smashing localization of $\tsadim$ by \cite[Theorem 4.2.7]{AmbiHeight}.
	By \cref{atomic-smashing} we get that atomics in $\CC \in \Mod_{\tsadi_{\leq n}}$ with respect to either mode coincide.
	Additionally, for $R \in \Alg(\tsadi_{\leq n})$ left modules over $R$ in either mode agree.
\end{remark}

\begin{thm}\label{redshift-upper}
	Let $\CC \in \Mod_{\tsadi_{\leq n}} \subset \Mod_{\tsadim}$ for some finite number $n < m$,
	then $\Km(\CC^{\atomic}) \in \tsadi_{\leq n+1}$.
	In particular, if $R \in \Alg(\tsadi_{\leq n})$ then $\Km(R) \in \tsadi_{\leq n+1}$.
\end{thm}

\begin{proof}
	Combine \cref{leq-n-atomic} and \cref{redshift-small}.
\end{proof}

\cref{redshift-small} shows that $p_{n}\colon \Km(\CC) \to \Km(\CC)$ is invertible, but in fact we can prove the following stronger result if we assume that $\CC$ is $m$-semiadditive.
Note that as we know now that $\Km(\CC) \in \tsadi_{\leq n}$, it is an object of an $\infty$-semiadditive category, so that $p_{(k)}$ is defined for all $k$.

\begin{prop}\label{pk-1}
	Let $\CC \in \Catmfinst$ be an $m$-semiadditive category with $\height(\CC) \leq n+1$ as an object of $\Catmfinst$ for some finite $n < m$.
	Then $p_{(k)}\colon \Km(\CC) \to \Km(\CC)$ is the identity for every $k \geq n+1$.
	In particular, for $\CC \in \Mod_{\tsadi_{\leq n}}$, the map $p_{(k)}\colon \Km(\CC^{\atomic}) \to \Km(\CC^{\atomic})$ is the identity for every $k \geq n+1$.
\end{prop}

\begin{proof}
	Recall from \cref{semiadditive-redshift-2} that for every $X \in \CC$ we have $\height(X) \leq n$, i.e.\ $|\BB^n C_p|\colon X \to X$ is invertible.
	\cite[Proposition 2.4.7 (1)]{AmbiHeight} applied to the case $A = \BB^{n+1} C_p$ shows that $\colim_{\BB^{n+1} C_p} X \xrightarrow{\nabla} X$ is an equivalence.
	By \cref{pn-formula}, $p_{(n+1)}\colon \CC \to \CC$ is given by $p_{(n+1)}(X) \cong \colim_{\BB^{n+1} C_p} X$, which by the above is $X$ itself, i.e.\ $p_{(n+1)}$ is the identity.
	By \cite[Proposition 2.4.7]{AmbiHeight}, if $p_{(k)}$ is invertible then $p_{(k+1)}$ is also invertible and is its inverse, finishing by induction.
	For the second part apply \cref{leq-n-atomic}.
\end{proof}

\subsection{Lower Bound}

\begin{prop}\label{height-geq-1}
	Let $\CC \in \Catmfinst$ be an $m$-semiadditive category with $\height(\CC) \leq n+1$ as an object of $\Catmfinst$ for some finite $n < m$.
	Then $\Km(\CC) \in \tsadi_{1 \leq \bullet \leq n+1}$.
	In particular, for $\CC \in \Mod_{\tsadi_{\leq n}}$ we have that $\Km(\CC^{\atomic}) \in \tsadi_{1 \leq \bullet \leq n+1}$.
\end{prop}

\begin{proof}
	Recall from \cref{redshift-small} that $\Km(\CC) \in \tsadi_{\leq n+1}$, and that $\tsadi_{\leq n+1}$ is $\infty$-semiadditive.
	We thus need to show that $L_0 \Km(\CC) = 0$ where $L_0\colon \tsadi_{\leq n+1} \to \tsadi_0$ is the projection.
	By \cref{height-dec}, the map $L_0$ preserves $p_{(k)}$, so by \cref{pk-1} we know that $p_{(k)}$ acts as the identity on $L_0 \Km(\CC)$ for every $k \geq n+1$.
	Since $L_0 \Km(\CC)$ is of height $\leq 0$, it is $p_{(0)} = p$-invertible, and \cite[Proposition 2.4.7]{AmbiHeight} then shows inductively that $p_{(2k)} \cong p$ for any $k \geq 0$.
	Choose some $k$ such that $2k \geq n+1$, then we get that $p \cong p_{(2k)} \cong 1$, so that $p - 1$ is the zero morphism on $L_0 \Km(\CC)$.
	However, since $\tsadi_0$ is $p$-local, this morphism is invertible, thus $L_0 \Km(\CC) = 0$.
	For the second part apply \cref{leq-n-atomic}.
\end{proof}

Let $\unit$ be the unit of $\tsadi_n$, and consider $\unit[\BB^n C_p] = \colim_{\BB^n C_p} \unit \in \CAlg(\tsadi_n)$ (which carries an action of $(\ZZ/p)^\times \cong \Aut(C_p)$).
We then have the following result.

\begin{prop}\label{mod-at-pn}
	Let $\CC \in \Mod_{\tsadi_n}$, then
	$p_{(n)}\colon \CC^{\atomic} \to \CC^{\atomic}$ is given by $\unit[\BB^n C_p] \otimes -$.
\end{prop}

\begin{proof}
	Recall that the action of $\tsadi_n$ on $\CC$ commutes with colimits.
	Then, by \cref{pn-formula}
	$$
		p_{(n)}(-)
		\cong \colim_{\BB^n C_p} (-)
		\cong \colim_{\BB^n C_p} (\unit \otimes -)
		\cong (\colim_{\BB^n C_p} \unit) \otimes -
		= \unit[\BB^n C_p] \otimes -
		.
	$$
\end{proof}

Motivated by this result, we further study the action of $\unit[\BB^n C_p]$.
To that end, we recall the following:

\begin{defn}[{\cite[Proposition 4.5 and Definition 4.7]{Cyclo}}]\label{R-bncp}
	The ring $\unit[\BB^n C_p]$ splits ($(\ZZ/p)^\times$-equivariantly) as a product in $\CAlg(\tsadi_n)$
	$$
		\unit[\BB^n C_p] \cong \unit \times \mdef{\unit[\omegapn]}
	$$
	where $\unit[\omegapn] \in \CAlg(\tsadi_n)$ is called the \tdef{(height $n$) $p$-th cyclotomic extension} of $\unit$.
	For any $R \in \Alg(\tsadi_n)$, we define the $R$-algebra $R[\omegapn] = \unit[\omegapn] \otimes R$.
\end{defn}

\begin{defn}\label{p-roots}
	We say that a category $\CC \in \Mod_{\tsadi_n}$ has \tdef{(height $n$) $p$-th roots of unity} if
	$\unit[\omegapn] \otimes -$
	is equivalent to
	$\prod_{(\ZZ/p)^\times} -$
	as functors $\CC \to \CC$.
	Similarly, we say that $R \in \Alg(\tsadi_n)$ has \tdef{(height $n$) $p$-th roots of unity} if $R[\omegapn]$ is equivalent to $\prod_{(\ZZ/p)^\times} R$ as an $R$-algebra.
\end{defn}

\begin{remark}
	It makes sense to require the equivalence to be $(\ZZ/p)^\times$-equivariant, but we will not use this assumption.
\end{remark}

\begin{remark}
	For $R \in \Alg(\SpTn)$, it is always true that the (height $n$) $p$-th cyclotomic extension $R[\omegapn]$ has (height $n$) $p$-th roots of unity.
	Indeed, \cite[Proposition 5.2]{Cyclo} shows that $R[\omegapn]$ is a Galois extension of $R$, which immediately implies the condition of having (height $n$) $p$-th roots of unity.
	However, it is \emph{not} true in general for $R \in \Alg(\tsadi_n)$, by a counterexample constructed by Allen Yuan.
\end{remark}

\begin{example}[{\cite[Proposition 5.1]{Cyclo}}]\label{En-roots-of-unity}
	The Lubin--Tate spectrum $\rE_n \in \CAlg(\SpKn)$ has (height $n$) $p$-th roots of unity.
\end{example}

\begin{example}
	At the prime $p = 2$, any $\CC = \LMod_R$ for $R \in \SpTn$, and more generally $\CC \in \Mod_{\SpTn}$, has (height $n$) $p$-th roots of unity.
	This is true because \cite[Proposition 5.2]{Cyclo} shows that $\SS_{\Tn}[\omegapn]$ is a $(\ZZ/p)^\times$-Galois extension of $\SS_{\Tn}$.
	Since $p = 2$, it is a Galois extension of order $p-1=1$, i.e.\ $\SS_{\Tn}[\omegapn] \cong \SS_{\Tn}$, so the condition is automatic.
\end{example}

\begin{lem}
	Let $\CC$ be a monoidal category, $R \in \Alg(\CC)$ and $n$ a natural number.
	Consider $R^n$ as an algebra under $R$ via the diagonal map $\Delta\colon R \to R^n$.
	Then the composition of the extension of scalars and the restriction of scalars $\LMod_R \xrightarrow{\Delta_!} \LMod_{R^n} \xrightarrow{\Delta^*} \LMod_R$ is equivalent to the functor $(-)^n\colon \LMod_R \to \LMod_R$.
\end{lem}

\begin{proof}
	Consider the $i$-th projection $\pi_i\colon R^n \to R$.
	The extension of scalars along it gives a functor $\LMod_{R^n} \to \LMod_R$.
	The $n$ functors together give an equivalence $\LMod_{R^n} \iso \LMod_R^n$.
	Since $R \xrightarrow{\Delta} R^n \xrightarrow{\pi_i} R$ is the identity, so are $(\pi_i \Delta)_!$ and $(\pi_i \Delta)^*$.
	Therefore, $\LMod_R \xrightarrow{\Delta_!} \LMod_{R^n} \iso \LMod_R^n$ is the diagonal, and $\LMod_R^n \iso \LMod_{R^n} \to \LMod_R$ is the multiplication of modules, so we conclude that their composition is indeed $(-)^n\colon \LMod_R \to \LMod_R$.
\end{proof}

\begin{prop}\label{R-to-LMod-p-roots}
	Let $R \in \Alg(\tsadi_n)$ have (height $n$) $p$-th roots of unity, then $\LMod_R \in \Mod_{\tsadi_n}$ has (height $n$) $p$-th roots of unity.
\end{prop}

\begin{proof}
	First note that the map $\unit[\omegapn] \otimes -\colon \tsadi_n \to \tsadi_n$ is equivalent to the composition of the extension of scalars and restriction of scalars along $\unit \to \unit[\omegapn]$.
	Since by definition $R[\omegapn] = \unit[\omegapn] \otimes R$, we get an $R$-algebra map $R \to R[\omegapn]$.
	Then, the composition of the extension of scalars and restriction of scalars along this map gives a functor $\LMod_R \to \LMod_R$, which is equivalent to the action of $\unit[\omegapn]$ on $\LMod_R$.
	By assumption, $R[\omegapn] \cong \prod_{(\ZZ/p)^\times} R$ as an $R$-algebra, so the previous lemma implies that $\unit[\omegapn] \otimes -$ is equivalent to $\prod_{(\ZZ/p)^\times} -$.
\end{proof}

\begin{prop}\label{pn-p}
	Let $\CC \in \Mod_{\tsadi_n}$ have (height $n$) $p$-th roots of unity, then $p_{(n)}\colon \CC^{\atomic} \to \CC^{\atomic}$ is equivalent to $p\colon \CC^{\atomic} \to \CC^{\atomic}$.
\end{prop}

\begin{proof}
	Recall that the action of $\tsadi_n$ on $\CC$ commutes with colimits, so that $\unit[\BB^n C_p] \otimes -\colon \CC \to \CC$ commutes with colimits.
	Additionally, $\tsadi_n$ is semiadditive, so finite products and coproducts coincide, so that $\unit[\BB^n C_p] \otimes -\colon \CC \to \CC$ also commutes with finite products, and the same holds for the restriction to the atomics $\unit[\BB^n C_p] \otimes -\colon \CC^{\atomic} \to \CC^{\atomic}$.
	By \cref{mod-at-pn}, $p_{(n)} \cong \unit[\BB^n C_p] \otimes -$ on $\CC^{\atomic}$.
	Now, using \cref{R-bncp} and the assumption that $\CC$ has $p$-th roots of unity, we get an equivalence of functors $\CC^{\atomic} \to \CC^{\atomic}$
	$$
		p_{(n)}(-)
		\cong \unit[\BB^n C_p] \otimes -
		\cong (\unit \times \unit[\omegapn]) \otimes -
		\cong (\unit \times \prod_{(\ZZ/p)^\times} \unit) \otimes -
		\cong \prod_p \unit \otimes -
		\cong \prod_p (-)
		.
	$$
\end{proof}

\begin{thm}\label{redshift-lower}
	Let $\CC \in \Mod_{\tsadi_n}$ have (height $n$) $p$-th roots, and $n < m$ a finite number,
	then $\Km(\CC^{\atomic}) \in \tsadi_{n+1}$.
	In particular, if $R \in \Alg(\tsadi_n)$ has (height $n$) $p$-th roots of unity, then $\Km(R) \in \tsadi_{n+1}$.
\end{thm}

\begin{proof}
	\cref{height-geq-1} shows that $\Km(\CC^{\atomic}) \in \tsadi_{1 \leq \bullet \leq n+1}$.
	In particular, it is of height $> 0$, i.e.\ $p_{(0)} = p$-complete.
	By \cref{pn-p}, $p_{(n)} \cong p$, so that it is also $p_{(n)}$-complete, that is, of height $> n$.
	The last part follows from \cref{R-to-LMod-p-roots}.
\end{proof}

\begin{cor}\label{KEn-tsadinpo}
	Let $n < m$ be a finite number, then $\Km(\En) \in \CAlg(\tsadi_{n+1})$.
\end{cor}

\begin{proof}
	Follows from \cref{En-roots-of-unity} and \cref{redshift-lower}.
\end{proof}

	\section{Relationship to Chromatically Localized K-Theory}\label{sec-chrom-k}

In \cref{sec-red} we have shown that higher semiadditive algebraic K-theory interacts well with semiadditive height.
For example, $\height(\Km(\En)) = n+1$ when $m > n$ by \cref{KEn-tsadinpo}.
Note that the assumption $m > n$ is necessary to even define semiadditive height $n+1$.
In this section we study the connection between higher semiadditive algebraic K-theory and chromatic localizations of ordinary algebraic K-theory by other means, while also dropping the assumption $m > n$.

Let $R \in \Alg(\SpTn)$.
The inclusion $\SpTnpo \subset \Sp$ admits a left adjoint $\LTnpo\colon \Sp \to \SpTnpo$.
Since $\KK(R) \in \Sp$, we can consider $\LTnpo \KK(R) \in \SpTnpo$.
Similarly, by \cref{tsadim-Tn-smash}, there is an inclusion $\SpTnpo \subset \tsadim$ for any $m \geq 1$, which admits a left adjoint $\LTnpo^{\tsadim}\colon \tsadim \to \SpTnpo$.
Note that, in some senses, $\LTnpo^{\tsadim}$ is better behaved than $\LTnpo$, as it is a smashing localization.
Since $\Km(R) \in \tsadim$, we can consider $\LTnpo^{\tsadim} \Km(R) \in \SpTnpo$.
By \cref{K-to-Km-Tn}, there is a natural comparison map
$$
\LTnpo \KK(R) \to \LTnpo^{\tsadim} \Km(R)
\quad \in \SpTnpo.
$$
This raises two independent questions:
\begin{enumerate}
	\item Does $\Km(R)$ land in $\SpTnpo \subset \tsadim$?
 	\item Is the comparison map an equivalence?
\end{enumerate}
A positive answer to both questions will imply that $\Km(R) \cong \LTnpo \KK(R)$, see \cref{conj-Km-Tnpo-K}.
In \cref{Km-Tk-local} we show that the first question is closely related to the Quillen--Lichtenbaum conjecture for $R$, in the guise of having a non-zero finite spectrum $X$ such that $\KK(R) \otimes X$ is bounded above.
By \cref{sheafification-Tn}, the second question is equivalent to $\LTnpo \Kzm(R)$ satisfying the $m$-Segal condition.
More informally, having descent properties for $\Tnpo$-localized K-theory.

Using the Galois descent results for $\Tnpo$-localized K-theory of \cite{DescVan}, the second question is answered in the affirmative for $m = 1$ in \cref{Ko-desc}.

We then study the case where $R$ has height $0$.
The main result is \cref{Km-height-0}, showing that for any $p$-invertible algebra $R \in \Alg(\Sppinv)$ and $m \geq 1$, there is an equivalence
$$
\Km(R) \cong L_{\To}\KK(R)
\quad \in \SpTo.
$$
This is first proved for $R = \SSpinv$ by employing the Quillen--Lichtenbaum property of $\SSpinv$ together with \cref{Ko-desc} mentioned above.
The general case then follows via the lax symmetric monoidal structure on $\Km$.

Finally, we study the completed Johnson--Wilson spectrum $\cJW$ at height $n \geq 1$, endowed with the Hahn--Wilson \cite{BPex} $\bbE_3$-algebra structure (see \cref{BP-QL-orig}) and, more generally, any $R \in \Alg(\LMod_{\cJW})$.
In \cref{Km-cJW} we show that
$$\Km(R) \in \SpTnpo$$
for any $m \geq 1$, strengthening \cref{redshift-lower} for $\cJW$-algebras.
In the case $m = 1$, \cref{Ko-desc} implies that
$$
\Kmo(R) \cong \LTnpo \KK(R)
\quad \in \SpTnpo.
$$
To prove \cref{Km-cJW}, we first use the Quillen--Lichtenbaum result for $\BPn$ of \cite{BPex} and the lax symmetric monoidal structure on $\Km$ to show that $\Km(\cJW) \in \SpTze \times \cdots \times \SpTnpo$.
We would like to thank the anonymous referee for suggesting this argument.
Then, we compute the cardinality of the classifying space of the $k$-fold wreath product of $C_p$ at each chromatic height in two different ways.
We observe that they are compatible only in chromatic height $n+1$, concluding that $\Km(\cJW) \in \SpTnpo$.
Using the lax symmetric monoidal structure on $\Km$, this is generalized to any $\cJW$-algebra.

Throughout this section $F(n)$ denotes a type $n$ finite spectrum (for example, the generalized Moore spectrum $\SS/(p^{i_0}, v_1^{i_1}, \dotsc, v_{n-1}^{i_{n-1}})$).
Without loss of generality, we may assume that $F(n)$ is an algebra, i.e.\ $F(n) \in \Alg(\Sp)$, by replacing it by $F(n) \otimes DF(n) \cong \End(F(n))$.

\subsection{General Results}

We begin this subsection by recalling and slightly generalizing some results from \cite{TeleAmbi} and \cite{AmbiHeight} that will be used in the rest of the section.

\begin{lem}\label{tor-zero}
	Let $\CC$ be a symmetric monoidal stable $p$-local $1$-semiadditive category, and let $R \in \CAlg(\CC)$.
	If $1 \in \pi_0 R = \hom_{h\CC}(\unit_\CC, R)$ is a torsion element, then $R = 0$.
\end{lem}

\begin{proof}
	This is essentially \cite[Corollary 4.3.5]{TeleAmbi}, we repeat the argument for the convenience of the reader.
	By \cite[Theorem 4.3.2]{TeleAmbi}, the ring $\pi_0 R$ admits an additive $p$-derivation.
	By assumption, $\pi_0 R$ is $p$-local.
	Since $1 \in \pi_0 R$ is torsion, \cite[Proposition 4.1.10]{TeleAmbi} shows that it is also nilpotent, which means that $1 = 0 \in \pi_0 R$.
	Therefore $R = 0$.
\end{proof}

Fix $m \geq 1$.
Recall from \cite[Theorem 4.2.7]{AmbiHeight} that $\tsadim \cong \tsadi_0 \times \tsadim_{>0}$.
Also, by \cite[Example 5.3.7]{AmbiHeight}, $\tsadi_0 \cong \SpTze \cong \Sp_\QQ$.
The collection of objects $X \in \tsadim_{>0}$ such that $X \otimes F(n+2) = 0$ form a full subcategory, which is equivalent to $\tsadim_{>0} \otimes \LnpofSp \cong \SpTo \times \cdots \times \SpTnpo$ by \cite[Corollary 5.5.7]{AmbiHeight} and \cite[Theorem G]{AmbiHeight}.
We thus have:

\begin{lem}\label{chrom-char}
	Let $X \in \tsadim_{>0}$, then $X \in \SpTo \times \cdots \times \SpTnpo$ if and only if $X \otimes F(n+2) = 0$.
\end{lem}

Recall from \cite[Proposition 5.3.9]{AmbiHeight} that, similarly to the $\Kn$- and $\Tn$-localizations, the map of modes $\Sp \to \tsadi_n$ vanishes on all bounded above spectra when $n \geq 1$.
Here we prove a slight generalization of this result.

\begin{lem}\label{bounded-tsadi}
	Let $m \geq 1$, then the map of modes $G_{>0}\colon \Sp \to \tsadim_{>0}$ vanishes on all bounded above spectra.
\end{lem}

\begin{proof}
	We follow closely the argument of \cite[Proposition 5.3.9]{AmbiHeight}, diverging only the case of $\mbb{F}_p$.
	The class of spectra on which $G_{>0}$ vanishes is closed under colimits and desuspensions in $\Sp$.
	Hence, by a standard devissage argument, it suffices to show that $G_{>0}$ vanishes on $\QQ$ and $\mbb{F}_\ell$ for all primes $\ell$.
	First, $\QQ$ and $\mbb{F}_\ell$ for $\ell \neq p$ are $p$-divisible.
	Since $G_{>0}$ is $0$-semiadditive, $G_{>0}(\QQ)$ and $G_{>0}(\mbb{F}_\ell)$ are $p$-divisible as well, but all objects of $\tsadim_{>0}$ are $p$-complete, and so $G_{>0}(\QQ) = G_{>0}(\mbb{F}_\ell) = 0$.

	It remains to show that $G_{>0}(\mbb{F}_p) = 0$.
	Since $\mbb{F}_p \in \CAlg(\Sp)$ is an $\bbE_\infty$-algebra, and $G_{>0}$ is a map of modes, $G_{>0}(\mbb{F}_p) \in \CAlg(\tsadim_{>0})$ is an $\bbE_\infty$-algebra as well.
	Similarly, since $p = 0$ in $\mbb{F}_p$, the same holds in $\pi_0 G_{>0}(\mbb{F}_p)$.
	Thus, by \cref{tor-zero} with $\CC = \tsadim_{>0}$ and $R = G_{>0}(\mbb{F}_p)$, we know that $G_{>0}(\mbb{F}_p) = 0$ which concludes the proof.
\end{proof}

We now move on to proving the two main results of this subsection.

\begin{prop}\label{Km-Tk-local}
	Let $\CC \in \Alg(\Catmfinst)$ be an $m$-semiadditive category for some $m \geq 1$.
	Assume that there exists a category $\CC_0 \in \Alg(\Catst)$, a left $\CC_0$-module structure $\CC \in \LMod_{\CC_0}(\Catst)$, and that $\KK(\CC_0) \otimes F(n+2)$ is bounded above.
	Then $\Km(\CC) \in \SpTze \times \SpTo \times \cdots \times \SpTnpo$.
\end{prop}

\begin{remark}
	Note that $\CC_0$ need not be $m$-semiadditive or even have $m$-finite colimits.
	Moreover, we require no compatibility between the algebra structure on $\CC$ and its $\CC_0$-module structure.
\end{remark}

\begin{proof}
	Recall the decomposition $\tsadim \cong \SpTze \times \tsadim_{>0}$.
	By \cref{chrom-char}, it remains to show that the factor in $\tsadim_{>0}$ satisfies $\Km(\CC)_{>0} \otimes F(n+2) = 0$.
	Let $G\colon \Sp \to \tsadim$ denote the map of modes, the adjoint of the underlying functor $\underlying$.
	We also denote by $G_{>0}\colon \Sp \to \tsadim_{>0}$ the composition with the projection to $\tsadim_{>0}$.

	By assumption, $\KK(\CC_0) \otimes F(n+2) \in \Sp$ is bounded above, so that $G_{>0}(\KK(\CC_0) \otimes F(n+2)) = 0 \in \tsadim_{>0}$ by \cref{bounded-tsadi}.
	The composition $G_{>0}(\KK(-) \otimes F(n+2))$ is lax monoidal, since $\KK$ is lax (symmetric) monoidal by \cref{k-lax}, $F(n+2) \in \Alg(\Sp)$ is taken to be an algebra, and $G_{>0}$ is a map of modes.
	By assumption $\CC$ is a left module over $\CC_0$, thus we get that $G_{>0}(\KK(\CC) \otimes F(n+2))$ is a left module over $G_{>0}(\KK(\CC_0) \otimes F(n+2)) = 0$, and thus $G_{>0}(\KK(\CC) \otimes F(n+2)) = 0$ as well.

	Recall from \cref{K-to-Km-sm} that since $\CC$ is an algebra, we get an algebra map $\KK(\CC) \to \underline{\Km(\CC)} \in \Alg(\Sp)$.
	By the adjunction $G \dashv \underlying$ and composing with the projection to $\tsadim_{>0}$, we get an algebra map $G_{>0}(\KK(\CC)) \to \Km(\CC)_{>0} \in \Alg(\tsadim_{>0})$.
	Since the functor $G_{>0}$ is a functor between stable modes, it commutes with the action of $\Sp$.
	Therefore, tensoring the map with the algebra $F(n+2)$ yields $G_{>0}(\KK(\CC) \otimes F(n+2)) \to \Km(\CC)_{>0} \otimes F(n+2) \in \Alg(\tsadim_{>0})$.
	We have shown that the source is $0$, and since this is an algebra map, so is the target, which concludes the proof.
\end{proof}

In the next proposition we would like to use \cite[Theorem C]{DescVan}, which applies to $\Lnf\SS$-linear stable categories.
We recall that an $\Lnf\SS$-linear stable category is, by definition, a module over $\Perf(\Lnf\SS) = \Mod_{\Lnf\SS}^\dbl$ in $\Catst$.
Note that since $\LnfSp$ is a smashing localization of $\Sp$ we have that $\Mod_{\Lnf\SS} = \LnfSp$.
In particular, for $R \in \Alg(\SpTn)$, we have that $\LMod_R^{\atomic}$ is $\Lnf\SS$-linear, since $\LMod_R \in \Mod_{\SpTn} \subset \Mod_{\LnfSp}$ and left dualizable modules coincide with atomics by \cref{atomic-ldbl}.
Thus $\LMod_R^{\atomic}$ is an example for $\CC$ in the following proposition.

\begin{prop}\label{Ko-desc}
	Let $\CC \in \Catofinst$ be $\Lnf\SS$-linear for some $n > 0$ (i.e., it admits a $\Perf(\Lnf\SS)$-module structure in $\Catst$), then
	$$\LTnpo^{\tsadio} \Kmo(\CC) \cong \LTnpo \KK(\CC).$$
	In particular, if $\Kmo(\CC) \in \SpTnpo \subset \tsadio$, then $\Kmo(\CC) \cong \LTnpo \KK(\CC)$.
\end{prop}

\begin{proof}
	By \cref{sheafification-Tn} it suffices to show that $\LTnpo \Kzo(\CC)$ satisfies the $1$-Segal condition, that is, for any $1$-finite $p$-space $A$, the canonical map
	$$
	\LTnpo \KK(\lim_A \CC)
	\to \lim_A \LTnpo \KK(\CC)
	$$
	is an equivalence.
	As both sides take coproducts in $A$ to direct sums, we may assume that $A$ is connected, i.e.\ $A = \BB G$ for a finite $p$-group $G$.
	This is exactly \cite[Theorem C]{DescVan}.
\end{proof}

\subsection{Height \texorpdfstring{$0$}{0}}\label{ssec-height-0}

We recall a form of the Quillen--Lichtenbaum conjecture for $\SSpinv$.

\begin{prop}\label{QL}
	$F(2) \otimes \KK(\SSpinv)$ is bounded above.
\end{prop}

\begin{proof}
	By \cite[Theorem 1.1 and Theorem 1.2]{Hyp},
	$$
	\KK(\SSpinv) \to \Lof \KK(\SSpinv) \times_{\Lof \TC(\SSpinv)} \TC(\SSpinv)
	$$
	is an isomorphism on high enough $p$-local homotopy groups.
	Tensoring with a finite spectrum preserves the property of a map being an isomorphism on high enough homotopy groups, and $p = 0$ in $F(2)$, so it suffices to show that the right hand side vanishes after tensoring with $F(2)$.
	The tensor product of spectra commutes with finite limits, so it suffices to show that each term on the right hand side vanishes after tensoring with $F(2)$.

	By definition, any $\Lof$-local spectrum vanishes after tensoring with $F(2)$, which shows that both $\Lof \KK(\SSpinv)$ and $\Lof \TC(\SSpinv)$ vanish after tensoring with $F(2)$.
	
	Now, \cite[Definition II.1.8]{TC} exhibits the functor $\TC\colon \Alg(\Sp) \to \Sp$ as the composition
	$$
	\Alg(\Sp) \xrightarrow{\THH} \mrm{CycSp} \xrightarrow{\TC} \Sp.
	$$
	Since $\THH(\SSpinv) \cong \SSpinv$, and the second functor is exact, we get that $\TC(\SSpinv)$ is $p$-invertible.
	Therefore, $\TC(\SSpinv)/p = 0$, and in particular it vanishes after tensoring with $F(2)$.
\end{proof}

\begin{prop}\label{Sppinv}
	The category $\Mod_{\SSpinv}(\Sp) \cong \Sppinv$ is a mode, classifying the property of being a stable $p$-invertible presentable category.
	It is ($p$-typically) $\infty$-semiadditive, and a smashing localization of both $\Sp$ and of $\CMonm[\Sp]$.
	In addition, for any $m \geq 0$, there is a lax symmetric functor $(-)^{\atomic} \colon \Mod_{\Sppinv}^{\intL} \to \Catmfinst$, and $\CC^{\atomic} = \CC^\omega$.
\end{prop}

\begin{proof}
	The fact that $\Sppinv$ is a mode, classifying the property of being stable $p$-invertible presentable category, follows from \cite[Proposition 5.2.17 and Proposition 5.2.10]{AmbiHeight}.
	The fact that it is $p$-typically $\infty$-semiadditive follows from \cite[Proposition 2.3.4]{AmbiHeight}, which is essentially the same fact that $\Sp_\QQ$ is $\infty$-semiadditive.
	It is a smashing localization of both $\Sp$ and of $\CMonm[\Sp]$ because it can be obtained from both by inverting $p$.

	The fact that it is a smashing localization together with \cref{atomic-smashing} show that for $\CC \in \Mod_{\Sppinv}$ the $\Sppinv$-atomics, the $\Sp$-atomics and the $\CMonm[\Sp]$-atomics all coincide.
	In addition, \cref{stable-atomic} shows that $\Sp$-atomics are the same as compact objects.
	Finally, the functor $(-)^{\atomic} \colon \Mod_{\Sppinv}^{\intL} \to \Catmfinst$ is the restriction of $(-)^{\atomic} \colon \Mod_{\CMonm[\Sp]}^{\intL} \to \Catmfinst$ from \cref{tsadi-atomics}.
\end{proof}

\begin{cor}\label{p-inv-To}
	Let $\CC \in \Mod_{\Sppinv}$ and $m \geq 1$, then $\Km(\CC^\omega) \in \SpTo$.
\end{cor}

\begin{proof}
	First consider the universal case $\CC = \Sppinv$.
	\cref{QL} shows that $\KK(\SSpinv)$ is bounded above after tensoring with $F(2)$, so by \cref{Km-Tk-local} $\Km(\SSpinv) \in \SpTze \times \SpTo$.
	Similarly to the proof of \cref{leq-n-atomic}, $\Sppinv^\omega$ is $\infty$-semiadditive by \cref{tsadi-atomics}, and all of its objects have semiadditive height $0$, thus by \cref{height-geq-1} we get that $\height(\Km(\SSpinv)) > 0$, so that indeed $\Km(\SSpinv) \in \SpTo$.

	We now prove the general case.
	Recall that by \cref{Sppinv} and \cref{k-lax} $\Km((-)^{\atomic})$ is a composition of lax symmetric monoidal functors, and by \cref{Sppinv}, atomic objects coincide with compact objects.
	Since every $\CC \in \Mod_{\Sppinv}$ is by definition a module over $\Sppinv$, we get that $\Km(\CC^\omega)$ is a module over $\Km(\SSpinv)$ in $\tsadim$.
	As we have shown that $\Km(\SSpinv) \in \SpTo$, and $\SpTo$ is a smashing localization of $\tsadim$ by \cref{tsadim-Tn-smash}, we conclude that $\Km(\CC^\omega) \in \SpTo$ as well.
\end{proof}

\begin{thm}\label{Km-height-0}
	Let $\CC \in \Mod_{\Sppinv}$ and $m \geq 1$, then
	$$\Km(\CC^\omega) \cong L_{\To}\KK(\CC^\omega).$$
	In particular, for any $R \in \Alg(\Sppinv)$ we have $\Km(R) \cong L_{\To}\KK(R)$.
\end{thm}

\begin{proof}
	By \cref{indep-lam}, $\Km(\CC^\omega)$ is independent of $m \geq 1$, so we may assume that $m = 1$.
	Therefore, the result follows immediately from the combination of \cref{p-inv-To} and \cref{Ko-desc}.
\end{proof}

\begin{example}\label{Km-Qb}
	For any $1 \leq m \leq \infty$, there is an equivalence
	$\Km(\Qb) \cong \KU_p$.
\end{example}

\begin{proof}
	The combination of \cite[Corollary 4.7]{Kloc} and \cite[Main Theorem]{Kclsd} shows that there is an equivalence
	$\KK(\Qb)_p \cong \KK(\bbC)_p \cong \ku_p$.
	As $\KU_p$ is $\To$-local, and $\To$-localization is insensitive to connective covers, $L_{\To} \ku_p \cong \KU_p$, which shows that
	$L_{\To} \KK(\Qb) \cong \KU_p$,
	and the result follows by \cref{Km-height-0}.
\end{proof}

\subsection{Height \texorpdfstring{$n \geq 1$}{n ≥ 1}}\label{ssec-height-n}

Fix some finite height $1 \leq n < \infty$.
We let $\cJW \in \SpTn$ denote the completed Johnson--Wilson spectrum at height $n$, namely
$\cJW \cong \LTn \JW \cong \LTn \BPn$
(note that $\Kn$- and $\Tn$-localization coincide for $\MU$-modules by \cite[Corollary 1.10]{MUmod}).
The main input to this subsection is the following result:

\begin{thm}[{\cite[Corollary of Theorem B]{BPex}}]\label{BP-QL-orig}
	There exists an $\bbE_3$-$\BP$-algebra structure on $\BPn$ such that
	$$
	\KK(\BPn)_{(p)} \to \Lnpof \KK(\BPn)_{(p)}
	$$
	induces an isomorphism on $\pi_*$ for $* \gg 0$.
\end{thm}

Henceforth, we shall consider $\BPn$ as an $\bbE_3$-algebra with the structure from \cref{BP-QL-orig}, which also endows the localization $\cJW$ with a compatible $\bbE_3$-algebra structure.
An immediate corollary of this result is the following:

\begin{cor}\label{BP-QL}
	The spectrum $\KK(\BPn) \otimes F(n+2)$ is bounded above, where $F(n+2)$ is a type $n+2$ finite spectrum.
\end{cor}

\begin{lem}\label{Km-cJW-ze-npo}
	Let $m \geq 1$, then $\Km(\cJW) \in \SpTze \times \cdots \SpTnpo$.
\end{lem}

\begin{proof}
	We shall employ \cref{Km-Tk-local}.
	Consider the symmetric monoidal functor $\LTn\colon \Sp \to \SpTn$.
	We get an induced $\bbE_2$-monoidal functor on modules
	$$
	\LMod_{\BPn}(\Sp) \to \LMod_{\cJW}(\SpTn),
	$$
	which in turn induces an $\bbE_2$-monoidal functor on left dualizable objects
	$$
	\LMod_{\BPn}(\Sp)^\ldbl \to \LMod_{\cJW}(\SpTn)^\ldbl.
	$$
	The latter map exhibits the target as a module over the source.
	By \cref{BP-QL}, we know that $\KK(\LMod_{\BPn}(\Sp)^\ldbl) \otimes F(n+2) = \KK(\BPn) \otimes F(n+2)$ is bounded above.
	By \cref{atomic-ldbl} we know that
	$$
	\Km(\cJW) = \Km(\LMod_{\cJW}(\SpTn)^{\atomic})
	= \Km(\LMod_{\cJW}(\SpTn)^\ldbl)
	.
	$$
	The result now follows immediately from \cref{Km-Tk-local}.
\end{proof}

Our next goal, achieved in \cref{Km-cJW-npo}, is to show that in fact $\Km(\cJW) \in \SpTnpo$.
For that, we need to show that each $\Tk$-local part vanishes for $0 \leq k \leq n$, and in fact it will suffice to show this after tensoring with the Lubin--Tate spectrum $\Ek$.
To do so, we shall compute the action of $|\BB(\wr_k C_p)|$, where $\wr_k C_p$ is the $k$-fold wreath product of $C_p$ with itself, in two different ways.
First, in general in $\Ek$-modules (see \cref{wreath-val}), and second by studying the dimension of the $\cJW$-module $\cJW{}^{\BB(\wr_k C_p)}$ and deducing the result for the higher semiadditive algebraic K-theory (see \cref{Km-cJW-wreath-act}).

\begin{defn}
	Let $k \geq 0$.
	We denote by $\mdef{\ModEk} = \Mod_{\Ek}(\Sp_{\Kk})$ the category of $\Kk$-local $\Ek$-modules.
\end{defn}

We recall from \cite[Theorem 5.3.1]{TeleAmbi} that $\ModEk$ is $\infty$-semiadditive.

\begin{defn}
	For a $\pi$-finite space $A$, we denote
	$$
	\mdef{|A|_k} = |A|_{\ModEk} \in \pi_0(\Ek).
	$$
\end{defn}

Also note that by \cite[Proposition 2.2.6, see also Section 5.4]{TeleAmbi}, the cardinality $|A|_k$ in fact lands in $\Zp \subseteq \pi_0(\Ek)$.

\begin{prop}\label{wreath-val}
	Let $k \geq 1$ and let $G$ be a finite group.
	If $v_p(|\BB G|_k) > 0$, then $v_p(|\BB G \wr C_p|_k) = v_p(|\BB G|_k) - 1$.
	Moreover, $|\BB(\wr_k C_p)|_k$ is invertible.
\end{prop}

\begin{proof}
	This is similar to the proof of \cite[Proposition 4.3.7]{TeleAmbi}.
	Recall the additive $p$-derivation $\delta\colon \pi_0(\Ek) \to \pi_0(\Ek)$ from \cite[Definition 4.3.1]{TeleAmbi}, which by \cite[Theorem 4.2.12]{TeleAmbi} satisfies
	$$
	\delta|\BB G|_k
	= |\BB C_p|_k |\BB G|_k - |\BB G \wr C_p|_k.
	$$
	By \cite[Lemma 5.4.6]{AmbiHeight}, for $x \in \Zp \subset \pi_0(\Ek)$, if $v_p(x)>0$ then $v_p(\delta(x)) = v_p(x)-1$.
	Note that $v_p(|\BB C_p|_k |\BB G|_k) = v_p(|\BB C_p|_k) + v_p(|\BB G|_k) \geq v_p(|\BB G|_k)$.
	Combining the two, we get
	$$
	v_p(|\BB G \wr C_p|_k) = v_p(\delta|\BB G|_k) = v_p(|\BB G|_k) - 1.
	$$

	For the second part, by \cite[Lemma 5.3.3]{TeleAmbi}, we have $|\BB C_p|_k = p^{k-1}$.
	Thus, applying the first part $(k-1)$-times, we get $v_p(|\BB(\wr_k C_p)|_k) = v_p(|\BB C_p|_k)-(k-1) = 0$.
\end{proof}

For a finite $p$-group $G$, we consider $\hom(\Zp^n, G)$, the set of $n$-tuples commuting elements in $G$.
This set is endowed with an action of $\GL_n(\Zp)$ by pre-composition.
We shall in particular be interested in the integer
$$
\mdef{\fLn(G)} = |\hom(\Zp^n, G)/G| = \frac{|\hom(\Zp^{n+1}, G)|}{|G|},
$$
where the $G$-action in the second term is by conjugation point-wise, and the second equality is Burnside's lemma (\cite[Lemma 4.13]{HKR}).

\begin{prop}\label{wreath-comm}
	Let $G$ be a finite $p$-group, then
	$$
	\fLn(G \wr C_p) = \frac{\fLn(G)^p + (p^{n+1}-1)\fLn(G)}{p}.
	$$
\end{prop}

\begin{proof}
	We shall prove
	$$
	|\hom(\Zp^n, G \wr C_p)| = |\hom(\Zp^n, G)|^p + (p^n-1)|G|^{p-1}|\hom(\Zp^n, G)|.
	$$
	Since $|G \wr C_p| = p|G|^p$, the result immediately follows (when changing $n$ to $n+1$).

	In this proof, we will write elements of $C_p$ additively, namely, identify it with $\ZZ/p$.
	Elements of $G \wr C_p$ will be denoted by pairs $(a, f)$ where $a \in C_p$ and $f\colon C_p \to G$.
	Elements of $\hom(\Zp^n, C_p)$ will be denoted by $\alpha = (a_1, \dotsc, a_n)$.
	Similarly, elements of $\hom(\Zp^n, G \wr C_p)$ will be denoted by $\beta = ((a_1, f_1), \dotsc, (a_n, f_n))$.

	Consider the projection $G \wr C_p \to C_p$, which by post-composition induces $\pi\colon \hom(\Zp^n, G \wr C_p) \to \hom(\Zp^n, C_p)$.
	By construction of the action, $\pi$ is $\GL_n(\Zp)$-equivariant.
	We let $\hom^\alpha(\Zp^n, G \wr C_p) = \pi^{-1}(\alpha)$ denote the fiber.
	Observe that the $\GL_n(\Zp)$-action on $\hom(\Zp^n, C_p) \cong C_p^n$ has two orbits -- one which is the singleton $\alpha = 0$, and the second of size $p^n-1$ containing all $\alpha \neq 0$.
	We let $\alpha_1 = (1, 0, \dotsc, 0)$, an element of the second orbit.
	Since $\pi$ is equivariant, all of the fibers in the same orbit have the same size, so that
	$$
	|\hom(\Zp^n, G \wr C_p)|
	= |\hom^0(\Zp^n, G \wr C_p)| + (p^n-1)|\hom^{\alpha_1}(\Zp^n, G \wr C_p)|.
	$$
	We shall now compute the sizes of the two fibers.

	For the first fiber, observe that $\hom^0(\Zp^n, G \wr C_p) \subset \hom(\Zp^n, G \wr C_p)$ is the subset of those tuples landing in $G^p < G \wr C_p$.
	Therefore,
	$$
	|\hom^0(\Zp^n, G \wr C_p)| = |\hom(\Zp^n, G^p)| = |\hom(\Zp^n, G)|^p.
	$$

	For the second fiber, we being with the following observations.
	A pair of elements $(0, f), (0, f') \in G \wr C_p$ commute if and only if $f, f' \in G^p$ commute, i.e.\ if and only if $f(a), f'(a) \in G$ commute for all $a \in C_p$.
	A pair of elements $(1, f), (0, f') \in G \wr C_p$ commute if and only if $f(a) f'(1+a) = f'(a) f(a)$ for all $a \in C_p$.
	Applying this inductively, we see that $f'$ is uniquely determined by $f$ and $f'(0)$ and subject to the condition that $f'(0), \prod f(a) \in G$ commute.
	Using these observations, we can say when an $n$-tuple of elements in $G \wr C_p$, denoted $\beta = ((1, f_1), (0, f_2), \dotsc, (0, f_n))$, pairwise commute, i.e.\ $\beta \in \hom^{\alpha_1}(\Zp^n, G \wr C_p)$.
	The commutativity of $(1, f_1)$ and $(0, f_i)$ (for $i \geq 2$) says that $f_i$ is determined by $f_1$ and $f_i(0)$ and forces the condition that $f_i(0), \prod f_1(a)$ commute.
	The commutativity of $(0, f_i)$ and $(0, f_j)$ (for $i,j \geq 2$) forces the condition that $f_i(0)$ and $f_j(0)$ commute.
	We thus see that $\beta$ is uniquely specified by $f_1$ and $f_2(0), \dotsc, f_n(0)$, subject to the condition that $(\prod f_1(a), f_2(0), \dotsc, f_n(0))$ pairwise commute.
	In other words, we have a bijection
	\begin{gather*}
		\hom^{\alpha_1}(\Zp^n, G \wr C_p) \iso G^{p-1} \times \hom(\Zp^n, G),\\
		\beta \mapsto ((f_1(0), \dotsc, f_1(p-2)), (\prod f_1(a), f_2(0), \dotsc, f_n(0))).
	\end{gather*}

	To conclude, we have shown
	\begin{align*}
		|\hom(\Zp^n, G \wr C_p)|
		&= |\hom^0(\Zp^n, G \wr C_p)| + (p^n-1)|\hom^{\alpha_1}(\Zp^n, G \wr C_p)|\\
		&= |\hom(\Zp^n, G)|^p + (p^n-1)|G|^{p-1}|\hom(\Zp^n, G)|,
	\end{align*}
	as required.
\end{proof}

\begin{prop}\label{wreath-comm-Cp}
	For $1 \leq k \leq n$ we have $v_p(\fLn(\wr_k C_p)) = n - (k-1)$.
	In particular $\fLn(\wr_k C_p) \in \Zp$ is \emph{not} invertible.
\end{prop}

\begin{proof}
	First, for $k = 1$ we have
	$$
	\fLn(C_p) = \frac{|\hom(\Zp^{n+1}, C_p)|}{|C_p|} = p^n,
	$$
	so that $v_p(\fLn(C_p)) = n$.
	Now assume inductively that $v_p(\fLn(\wr_k C_p)) = n - (k-1)$.
	Since
	$$
	v_p(\fLn(\wr_k C_p)^p)
	= p v_p(\fLn(\wr_k C_p))
	> v_p(\fLn(\wr_k C_p))
	= v_p((p^{n+1}-1)\fLn(\wr_k C_p)),
	$$
	we get that
	$$
	v_p(\fLn(\wr_k C_p)^p + (p^{n+1}-1)\fLn(\wr_k C_p)) = v_p(\fLn(\wr_k C_p)).
	$$
	Thus, by \cref{wreath-comm}, we get the required result
	$$
	v_p(\fLn(\wr_{k+1} C_p)) = v_p(\fLn(\wr_k C_p)) - 1 = n - ((k+1)-1).
	$$
\end{proof}

We now recall the following result:

\begin{prop}\label{K-to-E}
	Let $X$ be a space with $\dim_{\Fp}(\Kn_0(X)) = d < \infty$ and $\Kn_1(X) = 0$.
	Then, $\cJW[X] \in \LMod_{\cJW}$ is a free $\cJW$-module of dimension $d$ and, in particular, dualizable.
\end{prop}

\begin{proof}
	This is the content of \cite[Proposition 8.4]{HS} (see also \cite[Proposition 3.4.3]{HL} and \cite[Lemma 5.1.7]{TeleAmbi}).
\end{proof}

\begin{prop}\label{Km-cJW-wreath-act}
	Let $1 \leq k \leq n$ and let $m \geq 1$.
	The action of $|\BB(\wr_k C_p)|$ on $\Km(\cJW) \in \tsadim$ is by multiplication by a non-invertible $p$-adic number.
\end{prop}

\begin{proof}
	By \cite[Theorem E]{HKR}, the group $\wr_k C_p$ is good in the sense of \cite[Definition 7.1]{HKR}, and in particular $\Kn_1(\BB(\wr_k C_p)) = 0$.
	Thus, by \cite[Theorem B and Lemma 4.13]{HKR}, we know that $\dim_{\Fp}(\Kn_0(\wr_k C_p)) = \fLn(\wr_k C_p)$.
	By \cref{K-to-E}, we conclude that $\cJW[\BB(\wr_k C_p)]$ is a free $\cJW$-module of dimension $\fLn(\wr_k C_p)$.
	Recall from \cref{pn-formula} that the action of $|\BB(\wr_k C_p)|$ on $\LMod_{\cJW}^{\atomic}$ as an object of $\Catmfinst$ is by $\cJW[\BB(\wr_k C_p)] \otimes (-)$, namely by multiplication by $\fLn(\wr_k C_p)$.
	Since $\Km$ is a $1$-semiadditive functor by \cref{Klam-limits}, the same holds for the action of $|\BB(\wr_k C_p)|$ on $\Km(\cJW)$ by \cite[Corollary 3.2.7]{TeleAmbi}.
	By \cref{wreath-comm-Cp}, this number is indeed a non-invertible $p$-adic number.
\end{proof}

\begin{lem}\label{Km-cJW-npo}
	Let $m \geq 1$, then $\Km(\cJW) \in \SpTnpo$.
\end{lem}

\begin{proof}
	By \cref{Km-cJW-ze-npo}, we know that $\Km(\cJW) \in \SpTze \times \cdots \SpTnpo$.
	It remains to show that the $\Tk$-local part, which for brevity we denote by $A_k \in \SpTk$, vanishes for every $0 \leq k \leq n$.

	We first deal with the case $1 \leq k \leq n$.
	Recall from \cite[Corollary 5.1.17]{TeleAmbi} that the functor
	$$
	\Ek[-]\colon \SpTk \to \ModEk
	$$
	is nil-conservative in the sense of \cite[Definition 4.4]{TeleAmbi}, that is, for any $A \in \Alg(\SpTk)$, if $\Ek[A] = 0$ then $A = 0$.
	Therefore, it suffices to check that $\Ek[A_k] = 0$.
	We have shown in \cref{Km-cJW-wreath-act} that $|\BB(\wr_k C_p)|$ acts on $\Km(\cJW)$ by multiplication by a non-invertible $p$-adic number.
	Since $\Ek[(-)_k]$ is a $1$-semiadditive functor by \cref{Klam-limits}, the same holds for $\Ek[A_k]$ by \cite[Corollary 3.2.7]{TeleAmbi}.
	We thus have shown that $|\BB(\wr_k C_p)|$ acts on $\Ek[A_k]$ as a $p$-adic number of positive valuation.
	Since by \cref{wreath-val} $|\BB(\wr_k C_p)|_k$ is invertible, and $\ModEk$ is $p$-complete, we conclude that $\Ek[A_k] = 0$.

	We now prove the remaining case $k = 0$.
	As above, \cref{Km-cJW-wreath-act} shows that $|\BB C_p|$ acts on $A_0$ by $\fLn(C_p) = p$.
	On the other hand, $|\BB C_p|_0 = p^{-1}$ by \cite[Lemma 5.3.3]{TeleAmbi}.
	Namely $p = p^{-1}$ on the rational spectrum $A_0$, thus $A_0 = 0$.
\end{proof}

\begin{thm}\label{Km-cJW}
	Let $R \in \Alg(\LMod_{\cJW})$ and $m \geq 1$, then $\Km(R) \in \SpTnpo$.
\end{thm}

\begin{proof}
	$\LMod_R$ is a right module over $\LMod_{\cJW}$.
	Recall from \cref{atomic-psh-sm-adj} that taking the atomics is a lax symmetric monoidal functor, and from \cref{k-lax} that $\Km$ is lax symmetric monoidal.
	Thus, we get that $\Km(R)$ is a right module over $\Km(\cJW)$.
	In addition, by \cref{tsadim-Tn-smash}, $\tsadim \to \SpTnpo$ is a smashing localization, and since $\Km(\cJW)$ lands in the smashing localization by \cref{Km-cJW-npo}, so does $\Km(R)$.
\end{proof}

\begin{cor}\label{Ko-cJW}
	Let $R \in \Alg(\LMod_{\cJW})$, then $\Kmo(R) \cong \LTnpo \KK(R)$.
	In particular, $\Kmo(\cJW) \cong \LTnpo \KK(\cJW)$.
\end{cor}

\begin{proof}
	This follows immediately from the combination of \cref{Km-cJW} and \cref{Ko-desc}.
\end{proof}

In work in progress with Carmeli and Yanovski \cite{cycloredshift} we show that \cref{Ko-cJW} holds for $m$-semiadditive K-theory for any $m \geq 1$.

	\bibliographystyle{alpha}
	\bibliography{refs}

\end{document}